\documentclass[sn-mathphys,Numbered]{sn-jnl}

\usepackage[all]{xy}
\usepackage{enumitem}
\usepackage{graphicx}%
\usepackage{multirow}%
\usepackage{amsmath,amssymb,amsfonts}%
\usepackage{amsthm}%
\usepackage{mathrsfs}%
\usepackage[title]{appendix}%
\usepackage{xcolor}%
\usepackage{textcomp}%
\usepackage{manyfoot}%
\usepackage{booktabs}%
\usepackage{algorithm}%
\usepackage{algorithmicx}%
\usepackage{algpseudocode}%
\usepackage{listings}%

\newcommand\restr[2]{{
		\left.\kern-\nulldelimiterspace 
		#1 
		\right|_{#2} 
}}

\def\Xint#1{\mathchoice
	{\XXint\displaystyle\textstyle{#1}}%
	{\XXint\textstyle\scriptstyle{#1}}%
	{\XXint\scriptstyle\scriptscriptstyle{#1}}%
	{\XXint\scriptscriptstyle\scriptscriptstyle{#1}}%
	\!\int}
\def\XXint#1#2#3{{\setbox0=\hbox{$#1{#2#3}{\int}$}
		\vcenter{\hbox{$#2#3$}}\kern-.5\wd0}}

\def\dashint{\Xint-}

\theoremstyle{thmstyleone}%
\newtheorem{teo}{Theorem}[section]
\newtheorem{lema}[teo]{Lemma}
\newtheorem{prop}[teo]{Proposition}
\newtheorem{coro}[teo]{Corollary}
\newtheorem{defi}[teo]{Definition}

\theoremstyle{thmstyletwo}%
\newtheorem{obs}{Remark}[section]%

\begin{document}

\title[Article Title]{Balanced metrics, Zoll deformations and isosystolic inequalities in $\mathbb{C}P^{\MakeLowercase{n}}$}


\author{\fnm{Luciano} \sur{L. Junior}}


\abstract{The $k$-systole of a Riemannian manifold is the infimum of the volume over all homologically non-trivial $k$-cycles. In this paper we discuss the behavior of the dimension two and co-dimension two systole of the complex projective space for distinguished classes of metrics, namely the homogeneous metrics and the balanced metrics. In particular, we argue that every homogeneous metric maximizes the systole in its volume-normalized conformal class, as well as that each Kähler metric locally minimizes the systole on the set of volume-normalized balanced metrics. The proof demands the implementation of integral geometric techniques, and a careful analysis of the second variation of the systole functional. As an application, we characterize the systolic behavior of almost-Hermitian 1-parameter Zoll-like deformations of the Fubini-Study metric.}




\maketitle

	\section{Introduction}

\indent The \textit{systole} of a closed Riemannian manifold is defined as the infimum of the length over all homotopically non-trivial loops. The interest in this geometric invariant started with C. Loewner, who proved that for every Riemannian metric on the two-dimensional torus, the systole is bounded by a universal constant times the square root of the area. This type of inequality is called \textit{isosystolic inequality}. Following his work, M. Pu provided an isosystolic inequality for the two-dimensional real projective space and characterized the equality case (\cite{Pu52}). The subject of \textit{systolic geometry} grew in interest with the stunning work of M. Gromov, who generalized Loewner's inequality for \textit{essential manifolds} (\cite{gromov_systole}). One of the reasons for such interest is the relation with different areas of mathematics, as, for instance, the link with \textit{isoperimetric inequalities}. A friendly introduction to the subject can be found in the following survey by L. Guth (\cite{guth2010metaphors}).      

Inspired by the works of C. Loewner and M. Pu, M. Berger proposed a definition of higher orders systoles (\cite{berger}). More concretely, if $(M^n,g)$ is a closed Riemannian manifold, we define the \textit{homological $k$-systole} with integer coefficients, or simply the \textit{$k$-systole}, as:  
$$\mathrm{Sys}_k(M,g)=\inf\{\mathrm{vol}_g(C)  :  \mbox{where $[C] \neq 0$ in ${H}_k(M,\mathbb{Z})$}\},$$
where the volume of a cycle is computed with respect to the $k$-dimensional Hausdorff measure induced by the Riemannian metric. From Cartan's Theorem the $1$-systoles are realized by geodesics. The $k$-systoles with $k > 1$, are realized by stable minimal submanifolds, possibly with singularities (\cite{federer69}). This creates a connection between systolic geometry and the theory of minimal submanifolds.   

Based on the aforementioned works, a natural question is the existence of isosystolic inequalities for the $k$-systole. However, perhaps because of the wilder nature of minimal submanifolds over geodesics, such inequalities are not expected. This phenomenon is known as \textit{systolic freedom} (\cite{Berger1992},\cite{Katz95}). Therefore, a more approachable problem is to study the points of local maximum and local minimum of the \textit{(volume) normalized systole}, 
$$\mathrm{Sys}_k^{\mathrm{nor}}(M,g)=\frac{{\mathrm{Sys}_k(M,g)}}{\mathrm{vol}(M,g)^{\frac{k}{n}}},$$
when restricted to distinguished subsets of Riemannian metrics. Note that the power in the volume is chosen in such way that the functional is invariant by scaling of the metric.

The first significant contribution in this regard comes from M. Berger (\cite{berger}), who demonstrated that in $\mathbb{C}P^n$, the Fubini-Study metric serves as the maximum for the normalized $2k$-systole within its conformal class, for all $1\leq k < n$. It is worth noting that in $\mathbb{C}P^n$, homology is only non-trivial for even dimensions. 

More recently, using the machinery of pseudo-holomorphic curves, M. Gromov and M. Berger also showed that in $\mathbb{C}P^2$, the Fubini-Study metric is a local maximum for the normalized $2$-systole.
{
	\renewcommand{\theteo}{A}
	\begin{teo}[Gromov-Berger, cf. \cite{gromov_pshol}, section $0.2.$B]\label{FSinCP2aremaximumforSys} 
		There exist an open neighborhood $g_{FS} \in \mathcal{U}$ of the Fubini-Study metric in the space $ (\mathrm{Riem}(\mathbb{C}P^2), C^{\infty})$ of Riemannian metrics, such that: 
		$$ \mathrm{Sys}^{\mathrm{nor}}_{2}(\mathbb{C}P^2,{g}) \leq \mathrm{Sys}^{\mathrm{nor}}_{2}(\mathbb{C}P^2,{g}_{FS}),$$
		for every metric $g \in \mathcal{U}$. Moreover, the equality holds if and only if there is an almost complex structure $J$ such that $(\mathbb{C}P^2,J,g)$ is almost Kähler.      
	\end{teo}                               
}
In contrast with the global result of M. Pu for $\mathbb{R}P^2$, this local statement is the best result we can expect in $\mathbb{C}P^2$. In fact, M. Katz and A. Suciu have proven that systolic freedom holds in this space (\cite{Katz}), excluding the possibility of a global version of this theorem.     

An interesting question is whether this theorem generalizes to the $(2n-2)$-systole in $\mathbb{C}P^n$, for $n>2$. However, given the significant differences in the character of pseudo-holomorphic curves and almost complex submanifolds of higher dimensions, this was not expected. In fact, in \cite{gromov_systole}, M. Gromov proved that this result is false for $n>2$, by exhibiting a family of almost Hermitian metrics approaching the Fubini-Study metric, each one with normalized systole larger than the Fubini-Study metric.   

Nevertheless, one question that remains and motivates part of our work is to characterize the behavior of the normalized $(2n-2)$-systole restricted to the set of Hermitian metrics of $\mathbb{C}P^n$, $n \geq 3$, that are compatible with the canonical complex structure. 

Our first observation is that, even when restricted to this smaller set, the Fubini-Study metric is not a local maximum for the normalized systole. In fact, we will prove that this metric represents a point of strict minimum for the normalized codimension two systole restricted to the class of homogeneous metrics in $\mathbb{C}P^{2n+1}$. Furthermore, we will also prove that systolic freedom holds within this class (see Theorem $\ref{systofhommetrics}$). 

The main observation leading to the aforementioned result is that every homogeneous metric is \textit{balanced}, meaning that the associated fundamental form is coclosed. Therefore, it is reasonable to ask if the Fubini-Study metric persist as a local minimum for the codimension two normalized systole in the larger set of balanced metrics, that are compatible with the canonical complex structure. This question leads us to our second result, namely: the Fubini-Study metric is a local minimum for the normalized $(2n-2)$-systole in $\mathbb{C}P^n$ when restricted to the infinite dimensional set of balanced metrics. Moreover, we characterize the equality case (see Theorem $\ref{sysofbal}$).

Note that in $\mathbb{C}P^2$, our result is consistent with the Gromov-Berger Theorem, as every balanced metric is Kähler when $n=2$. This leads us to conjecture that for the complex projective space, the normalized codimension two systole decreases under any non balanced perturbation of the Fubini-Study metric. For conformal directions this follows from the previously mentioned result of M. Berger, however the general case is still an open problem.         


Another aspect of the Gromov-Berger Theorem that we can draw inspiration from for generalizations is the rigidity statement. That is, the theorem guarantees the existence of an open neighborhood $\mathcal{U}$ of the Fubini-Study metric such that, if $g \in \mathcal{U}$ and 
$$\mathrm{Sys}_2^{\mathrm{nor}}(\mathbb{C}P^2,g) = \mathrm{Sys}_2^{\mathrm{nor}}(\mathbb{C}P^2,g_{FS}),$$
then there exists an almost complex structure $J$ such that $(\mathbb{C}P^2,J,g)$ is an almost Kähler manifold, i.e. its associated fundamental form is closed. By Taubes' uniqueness Theorem for symplectic structures on $\mathbb{C}P^2$ (\cite{Taubes95}), up to diffeomorphism and scaling we can assume that the fundamental form associated to the pair $(J,g)$ is the Fubini-Study form. In this case, the work of M. Gromov on pseudo-holomorphic curves (\cite{gromov_pshol}) implies that for every point and every tangent complex line there is a unique $J$-holomorphic $\mathbb{C}P^1$ that contains the point and is tangent to the given complex line (\cite{Sikorav04}, \cite{McKay_2006}). Moreover, each of these surfaces generates the homology of $\mathbb{C}P^2$ and realizes the $2$-systole.    

Note that we can interpret this property as a generalization of the classical Zoll condition (\cite{Zoll1903, Besse78}), following the approach of Ambrozio-Marques-Neves for the sphere (\cite{lucas_coda_andre}). In fact, recall that a Riemannian metric on a closed manifold is called Zoll if all its geodesics are closed and have the same length, which implies that through every point and every unit tangent vector there passes a unique simple, closed geodesic. A classical theorem of V. Guillemin (\cite{Guillemin}) guarantees an abundance of such Zoll metrics on $\mathbb{S}^2$.

While Guillemin's result concerns the existence of closed geodesics, Ambrozio-Marques-Neves extend this idea to the setting of minimal surfaces. Specifically, they prove that there is an abundance of Riemannian metrics on $\mathbb{S}^n$ which admit a \textit{generalized Zoll family}—a smooth family $\{\Sigma_\sigma\}_{\sigma \in \mathbb{RP}^n}$ of codimension one embedded minimal spheres such that for each $x \in S^n$ and each $(n-1)$-dimensional subspace $\pi \subset T_x S^n$, there exists a unique $\sigma \in \mathbb{RP}^n$ with $T_x \Sigma_\sigma = \pi$.
       
Moreover, they classify all metrics on the sphere that admit the canonical family of equators as a generalized Zoll family. This motivates us to propose the following definition.

{
	\renewcommand{\theteo}{B}
	\begin{defi}\label{zollmetricsdef}
		The set $\mathcal{Z}$ is defined as the class of almost Hermitian structures $(J, g)$ in $\mathbb{C}P^n$ whose admit a family $\{\Sigma^{2n-2}_\sigma\}_{\sigma \in \mathbb{C}P^n}$ of $(2n-2)$-dimensional submanifolds satisfying the following properties:
		\begin{enumerate}[label=\alph*),ref=(\alph*)]
			\item For every $\sigma \in \mathbb{C}P^n$ the submanifold $\Sigma_\sigma$ is diffeomorphic to $\mathbb{C}P^{n-1}$, minimal and $J$-almost complex.
			\item For every $(p ,\Pi) \in \mathrm{Gr}_{n-1}^{J}(T\mathbb{C}P^n)$, in the Grassmannian of $J$-almost complex hyperplanes, there exists a unique $\sigma \in \mathbb{C}P^n$ for which $p \in \Sigma_\sigma$ and $T_p \Sigma_\sigma = \Pi$. Moreover, the map $\mathrm{Gr}_{n-1}^{J}(T\mathbb{C}P^n) \ni (p,\Pi) \mapsto \sigma \in \mathbb{C}P^n$ is a submersion.
			\item The map $\mathbb{C}P^n \ni \sigma \mapsto \Sigma_\sigma$ is smooth in the sense of the graphical convergence.   
		\end{enumerate} 
		Moreover, we define the subset $\mathcal{Z}'$ as the elements $(J, g) \in \mathcal{Z}$, for which $H_{2n-2}(\mathbb{C}P^n, \mathbb{Z})$ is generated by $\Sigma_{\sigma}$, for every $\sigma \in \mathbb{C}P^n$. The family $\{\Sigma^{2n-2}\}_{\sigma \in \mathbb{C}P^n}$ is called the associated Zoll family.
\end{defi}}
Recall that, for a given family $\{\Sigma_\sigma\}_{\sigma \in \mathbb{C}P^n}$ of closed smooth submanifolds of $(\mathbb{C}P^n,g)$, the map $\mathbb{C}P^n \ni \sigma \mapsto \Sigma_\sigma$ is said to be \textit{smooth in the sense of graphical convergence} if, for every $\sigma_0 \in \mathbb{C}P^n$, there is an open neighborhood $\mathcal{W}$ of $\sigma_0$ and a smooth map $\mathcal{W} \ni \sigma \to \xi_\sigma \in \Gamma(N \Sigma_{\sigma_0})$ such that $\Sigma_\sigma = \left\{\mathrm{exp}_p\left(\xi_\sigma(p)\right) : p \in \Sigma_{\sigma_0}\right\}$ for every $\sigma \in \mathcal{W}$.    

Therefore, by our previous discussion, we can thus say that, if the metric $g$ in a neighborhood of the Fubini-Study metric satisfies the equality in Gromov-Berger Theorem, there exists an almost complex structure $J$ such that $(J,g) \in \mathcal{Z}'$. We prove the converse statement is true (see Theorem \ref{classWeaklyZoll}). In other words, in a neighborhood of the Fubini-Study metric, we can characterize the set $\mathcal{Z}'$ as the points of maximum of the normalized systole. This result can be compared with the relation between Zoll metrics and \textit{systoles} (i.e. least length closed geodesics) in $\mathbb{S}^2$ proved in (\cite{Alberto17}). 

Motivated by the previous characterization of the set $\mathcal{Z}'$ and the results of V. Guillemin (\cite{Guillemin}) and Ambrozio-Marques-Neves (\cite{lucas_coda_andre}) on Zoll deformations of the round metric in the sphere, we study $1$-parameter deformations of the Fubini-Study metric in $\mathcal{Z}$. In particular, using the classical deformation theory developed by K. Kodaira (\cite{Kodaira05}), we are able to show that such deformations must be balanced with respected with the canonical complex structure. In particular, such type of deformation must not decrease the normalized codimension two systole.             

The investigation of the  $(2n-2)$-systole invariant on $\mathbb{C}P^n$ leads naturally to the topic of balanced metrics, which plays a central role in this work. The first systematic work in this topic is due M.L. Michelson in the seminal article \cite{Michelsohn_1982}. Since then, balanced metrics arose in a variety of other contexts. For instance, in the theory of Twistor geometry over $4$-dimensional self-dual manifolds  (\cite{Atiyah78},\cite{Jixiang15}), Twistor geometry over hyperkähler manifolds (\cite{Verbitsky09}), Twistor geometry over hypercomplex manifolds (\cite{Tomberg15}), and also in the theory of complex Monge-Ampère equations (\cite{Yau08}).        

 In Section $\ref{section:MainResults}$ we state and discuss our main results in details. In Section $\ref{section:WeaklyZollmetrics}$ we classify the almost Hermitian manifolds which admit a large family of almost complex submanifolds that are also minimal submanifolds. In Section $\ref{section:homogeneousmetrics}$ we study the systole functional for the homogeneous metrics of the complex projective space. In Section $\ref{section:BalancedMetrics}$ we study the normalized systole restrict to the space of balanced metrics in $\mathbb{C}P^n$, for $n\geq 3$. In section $\ref{section:zolldeformation}$ we combine the above results to study $1$-parameter family of deformations of the Fubini-Study metric that lies in $\mathcal{Z}$. The paper also contains two appendices, one that discuss the relation of \textit{integral geometric formulas} with systolic inequalities, and the other that summarizes some classical results in the theory of Hermitian geometry.

\section{Main Results}\label{section:MainResults}

\subsection{The Class $\mathcal{W}_k$}
The criteria of integrability of the almost complex hyperplanes by minimal submanifolds, in the definition of $\mathcal{Z}$ (see definition $\ref{zollmetricsdef}$), can be viewed as a variation of the \textit{axiom of holomorphic planes} presented in the following paper by K. Yano and I. Mogi (\cite{YanoMogi55}), where they study integrability of complex planes by totally geodesic submanifolds, instead of minimal ones. This is a generalization of the classical \textit{axiom of $r$-planes}, defined and studied by E. Cartan. Moreover, the minimal counterpart of the axiom of $r$-planes was characterized by T. Hangan (\cite{Hangan96}, \cite{hangan1997beltrami}).

Therefore, we can draw inspiration in these works to propose the following generalization of the axiom of holomorphic planes.
{
	\renewcommand{\theteo}{C}
	\begin{defi}
		Let $(M^{2n}, J, g)$ be a $2n$-dimensional almost Hermitian manifold, with $n\geq2$. For each fixed integer $1\leq k \leq n-1$, we define the set $\mathcal{W}_k$ as the class of almost Hermitian structures $(J, g)$ satisfying the following property: for every $(p, \Pi) \in \mathrm{Gr}_k^{J}(TM)$, there exists a minimal and almost complex submanifold $\Sigma^{2k}_{p,\Pi}$ of $M$ such that $p \in \Sigma_{p,\Pi}$ and $T_p(\Sigma_{p,\Pi}) = \Pi$. 
	\end{defi}
}
The Section $\ref{section:WeaklyZollmetrics}$ will be devoted to the proof of the following classification theorem for  almost Hermitian structures that lies in $\mathcal{W}_k$. For the reader less familiarized with the theory of almost complex geometry, we refer Section $\ref{section:WeaklyZollmetricsSubsection:Prelinaries}$ for definitions.
{
	\renewcommand{\theteo}{D}
	\begin{teo}\label{classWeaklyZoll}
		Let $(M^{2n},J,g)$ be a $2n$-dimension almost Hermitian manifold, with $n \geq 2$. Then, we have:
		\begin{enumerate}[label=\alph*),ref=(\alph*)]
			\item\label{classWeaklyZoll1} The pair $(J,g)$ lies in $\mathcal{W}_1$ if and only if $(M,J,g)$ is quasi Kähler. 
			\item\label{classWeaklyZoll2} For each fixed $1<k<n-1$, the pair $(J,g)$ lies in $\mathcal{W}_k$ if and only if $(M,J,g)$ is Kähler.
			\item\label{classWeaklyZoll3} For $n\geq 3$, the pair $(J,g)$ lies in $\mathcal{W}_{n-1}$ if and only if $J$ is integrable and $(M,J,g)$ is balanced.  
		\end{enumerate}
\end{teo} }

The key computations for the proof of this theorem were inspired by the work of A. Gray (\cite{gray}), which contains a comprehensive study of the theory of almost complex geometry. In particular, it contains an important characterization of almost complex submanifolds that are also minimal.   

The technicality provided by the theory of almost complex structures can overshadow the simplicity of this statement. Therefore, we state the following corollary, with focus in the integrable case. Incidentally, it clarifies the relation of Theorem $\ref{classWeaklyZoll}$ and the theory of calibrations (\cite{HarveyLawson82}).  
{
	\renewcommand{\theteo}{E}
	\begin{coro}[Integrable Case]\label{classWeaklyZollInt} 
		Let $(M^{2n},J,g)$ be a $2n$-dimension Hermitian manifold, with $n \geq 2$, and let $\omega \in \Omega^{2}(M)$ be the associated fundamental form.
		\begin{enumerate}[label=\alph*),ref=(\alph*)]
			\item For each fixed $1\leq k < n-1$, we have that $(J,g) \in \mathcal{W}_k$ if and only if $d \omega = 0$.
			\item For $n\geq3$, $(J,g) \in \mathcal{W}_{n-1}$ if and only if $d\omega^{n-1}=0$.  
		\end{enumerate}
\end{coro}    }
One implication of the proof can be outlined as follows. Provided that $J$ is integrable, every element of $\mathrm{Gr}_{k}^{J}(TM)$ can be integrated by a germ of a complex submanifold. Therefore, if $\omega^k$ is a calibration, each of these germs must be a minimal submanifold, implying that $(J,g) \in \mathcal{W}_k$. Hence, the main content of the theorem is to prove that if we have enough minimal complex $2k$-submanifolds, then $\omega^k$ necessarily defines a calibration. 

A classical theorem in complex geometry due to Hirzebruch (\cite{Hirzebruch1957complex}), Kodaira and Yau (\cite{yau1977calabis}) states the uniqueness of the Kähler structure in $\mathbb{C}P^n$, up to biholomorphism. Combining this result with our classification we obtain the following.   
{
	\renewcommand{\theteo}{F}
	\begin{coro}
		Let $(M^{2n},J,g)$ be a $2n$-dimension almost Hermitian manifold, with $n \geq 2$. If $M$ is homeomorphic to $\mathbb{C}P^n$ and the pair $(J,g)$ lies in $\mathcal{W}_k$ for some $1<k<n-1$, then $(M,J,g)$ is a Kähler manifold biholomorphic to $\mathbb{C}P^n$.
\end{coro}  }

This corollary confirms our proposal that the relevant scenarios to study a Zoll-like integrability property in $\mathbb{C}P^n$ are the cases of pseudo-holomorphic curves and complex hypersurfaces, because the middle case presents a rigid structure. A counterpart of this observation for the axiom of (minimal) $r$-planes was proved by T. Hangan in \cite{hangan1997beltrami}. 

Corollary~\ref{classWeaklyZollInt} immediately implies that every almost complex Hermitian structure in $\mathcal{Z}$ must be integrable and balanced. Furthermore, we conjecture that this result can be refined in the sense that the only complex structure $J$ admitting a compatible Hermitian metric $g$ with $(J,g) \in \mathcal{Z}'$ is the canonical complex structure of $\mathbb{C}P^n$. This conjecture is motivated by the existence of a large number of complex hypersurfaces representing the generator class of $\mathbb{C}P^n$. In Section~\ref{section:zolldeformation}, we prove that this conjecture holds for $1$-parameter deformations of the Fubini-Study metric.

While our classification of generalized Zoll structures is algebraic, M. Mazzucchelli and S. Suhr (\cite{Mazzucchelli_Suhr}) characterized classical Zoll metrics on $\mathbb{S}^2$ through the coincidence of Lusternik-Schnirelmann min-max values. Subsequently, Ambrozio-Marques-Neves (\cite{ambrozio_marques_neves_24}) generalized this result by showing that metrics admitting generalized Zoll families on $\mathbb{S}^3$ are characterized by their spherical widths. It would be interesting to investigate whether a similar characterization via variational values exists for generalized Zoll structures in the complex projective space.          

\subsection{Systole of Homogeneous Metrics}
In \cite{berger}, M. Berger computed the $2k$-systole of $\mathbb{C}P^n$ endowed with the Fubini-Study metric, for $1 \leq k \leq n-1$. Moreover, using the integral geometric argument developed by M. Pu (see appendix $\ref{Appendix:IntegralGeometricFormulasandSystolic Inequalities}$), Berger also proved that the Fubini-Study metric is a maximum of the normalized $2k$-systole within its conformal class. In section $\ref{section:homogeneousmetrics}$, we will generalize Berger's results to the family of homogeneous metrics of the complex projective space, in the context of dimension two and codimension two systoles.        

Homogeneous metrics on $\mathbb{C}P^n$ have been classified by W. Ziller (\cite{Ziller1982}, section $3$). Besides the Fubini-Study metric and its rescalings, other homogeneous metrics exist only when $n$ is odd. These metrics behave similarly to the Berger metrics on the sphere, and they can be easily described by means of the \textit{Penrose fibration}, which is a fibration of $\mathbb{C}P^{2n+1}$ over $\mathbb{H}P^n$ with fibers $\mathbb{C}P^1$. 

In fact, if we denote the Penrose fibration by $\pi: \mathbb{C}P^{2n+1} \to \mathbb{H}P^n$, the family of homogeneous metrics can be constructed as follows. Consider the decomposition $T \mathbb{C}P^{2n+1} = \Lambda^0 \oplus \Lambda^1$, with $\Lambda^0=\mathrm{ker}d\pi$ and $\Lambda^1=(\mathrm{ker}d\pi)^{\perp}$, where the orthogonal complement is taken with respect to the Fubini-Study metric. Then, consider the family of metrics $g_{t} = t \restr{g_{FS}}{\Lambda^0} + \restr{g_{FS}}{\Lambda^1}$ for $t \in \mathbb{R}_{>0}$. As proved by Ziller, up to scaling and isometries, they are all the homogeneous metrics in $\mathbb{C}P^{2n+1}$. Since the normalized systole is invariant by scaling there is no loss of generality to restrict the study of homogeneous metrics to $\{g_{t}\}_{t \in \mathbb{R}_{>0}}$. Geometrically, the parameter $t\in \mathbb{R}_{>0}$ in the family $\{g_{t}\}_{t\in \mathbb{R}_{>0}}$ controls the volume of the fiber $\mathbb{C}P^1$ in $\mathbb{C}P^{2n+1}$.

This family displays a number of interesting properties. However, the most relevant for our work is the fact that each of these metrics is balanced. Hence, for every metric, the calibration argument implies that the complex submanifolds are area minimizing within their homology class. This fact, combined with the simplicity of the homology of the complex projective space, allows us to explicitly find the submanifold that realizes the dimension and codimension two systole for those metrics.    
{\renewcommand{\theteo}{G}
	\begin{prop}\label{sysofhomometrics}
		Suppose that $(\mathbb{C}P^m, J_{\mathrm{can}},g)$, $m\geq 2$, is balanced. Then, its codimension two systole satisfies
		\begin{equation}\label{eq:sysofhomometrics}
			\mathrm{Sys}_{2m-2}(\mathbb{C}P^{m},g) = \mathrm{area}_g(\mathbb{C}P_{\sigma}^{m-1}),
		\end{equation}
		where $\mathbb{C}P^{m-1}_{\sigma} \doteq \{[p] \in \mathbb{C}P^{m}: p \in \mathbb{S}^{2m+1} \text{ and } p \perp \sigma\},$ for each complex line $\sigma \in \mathbb{C}P^{m}$.
	\end{prop}
}
The above proposition settles the computation of the systole for the homogeneous metrics in the codimension two case. The dimension two case reduces to a comparison of the area of the fiber of the Penrose fibration against a linear $\mathbb{C}P^1$ that is traversal to the fibers. Combining these observations, we obtain the following.
{\renewcommand{\theteo}{H}
	\begin{teo}\label{systofhommetrics}
		The dimension and codimension two normalized systole functional for the family of homogeneous metrics $\{g_t\}_{t\in \mathbb{R}_{>0}}$ in $\mathbb{C}P^{2n+1}$, $n\geq1$, satisfies the following:
		\begin{enumerate}[label=\alph*),ref=(\alph*)]
			\item $\mathrm{Sys}^{\mathrm{nor}}_2(\mathbb{C}P^{2n+1},g_t)=  
			\begin{cases}
				\left(\frac{1}{(2n+1)!} \right)^{\frac{1}{2n+1}}t^{\frac{2n}{2n+1}}&,\mbox{ for } t \leq 1, \\
				\left(\frac{1}{(2n+1)!} \right)^{\frac{1}{2n+1}}{t}^{\frac{2n}{2n+1}-1}&,\mbox{ for } t\geq 1. \\
			\end{cases} $
			
			\item $\mathrm{Sys}^{\mathrm{nor}}_{4n}(\mathbb{C}P^{2n+1},{g}_t) = \left(\frac{1}{(2n+1)!}\right)^{\frac{1}{2n+1}} \left(2n+\frac{1}{t}\right)t^{\frac{1}{2n+1}}.$
		\end{enumerate}
	\end{teo}
}
The explicitness of the formulas presented in Theorem $\ref{systofhommetrics}$, enables us to derive two significant observations about the codimension two normalized systole of $\mathbb{C}P^{2n+1}$. The first is the minimality of the Fubini-Study metric over the set of homogeneous metrics. The second is the presence of the phenomena of systolic freedom within this set, both as $t$ goes to $0$ and $+\infty$. 

We remark,  that the systolic freedom in the class of Hermitian metrics was already observed by M. Berger (\cite{Berger1992}) and M. Gromov (\cite{gromov_systole}).   

The construction that leads to the Theorem $\ref{systofhommetrics}$ provides the minimal submanifolds that realize the systole for each case studied. This allows us to construct integral geometric formulas in the context of homogeneous metrics. Consequently, applying M. Pu and M. Berger's arguments we proved that each homogeneous metric maximizes the normalized systole within its conformal class. This generalizes Berger's result about the  of Fubini-Study metric.    
{\renewcommand{\theteo}{I}
	\begin{teo}\label{confclassofhommetr}
		Let $g$ be a homogeneous Riemannian metric in $\mathbb{C}P^{2n+1},$ for $n \geq 1$, and $\bar{g}$ a metric in the conformal class of $g$. For $k=1,2n$ we have 
		$$\mathrm{Sys}^{\mathrm{nor}}_{2k}(\mathbb{C}P^{2n+1},\bar{g}) \leq \mathrm{Sys}^{\mathrm{nor}}_{2k}(\mathbb{C}P^{2n+1},{g}).$$ 
		Moreover, a metric $\bar{g}$ attains the optimal bound if and only if is homothetic to $g$.
	\end{teo}
}

An analogous result for homogeneous metrics on $\mathbb{R}P^3$ was proven in (\cite{lucas_rafael_sistole_projcspace}, Theorem $1.1$).

\subsection{Systole of Balanced Metrics}
As previously observed, our results on the systole of homogeneous metrics suggests a study of the normalized systole over the set of balanced metrics, with respect with the canonical complex structure on $\mathbb{C}P^n$, $n\geq 3$. 

The main idea to take from those computations is the Proposition $\ref{sysofhomometrics}$, which allows us to conclude that the normalized systole over $\mathscr{B}$, the space of balanced metrics compatible with the canonical complex structure, is a smooth functional for an appropriated choice of topology in the space of Riemannian metrics.         
Moreover, equation \eqref{eq:sysofhomometrics} implies that the normalized systole is constant over $\mathscr{K}$, the space of Kähler metrics in $\mathscr{B}$. Therefore, the Hessian of the normalized systole functional contains at least the Kähler directions in its kernel. Consequently, it is not a positive form. 

However, a careful analysis of the first and second variation of this functional allows us to conclude that every Kähler metric determines a critical point. Even more, the Hessian is semi-positive definite, while the kernel is exactly determined by the Kähler directions. 

Applying a Taylor expansion argument, we can translate this infinitesimal property to a local behavior, obtaining the following theorem.           
{
	\renewcommand{\theteo}{J}
	\begin{teo}{\label{sysresttobalmet}}
		Let $n \geq 3$. There exists an open set $ \mathscr{K} \subset \mathcal{U} \subset \mathscr{B}$, in the ${C}^{2}$-topology, such that for every metric $g\in \mathcal{U}$, 
		$$\mathrm{Sys}_{2n-2}^{\mathrm{nor}}(\mathbb{C}P^n,g) \geq \mathrm{Sys}_{2n-2}^{\mathrm{nor}}(\mathbb{C}P^n,g_{FS}).$$
		Moreover, $g \in \mathcal{U}$ satisfies the equality if and only if $g \in \mathscr{K}$.      
	\end{teo} 
}

A question that remains open is to determine the local behavior of the co-dimension two normalized systole in the directions transversal to balanced metrics within the class of Hermitian metrics.     

\subsection{Deformations in $\mathcal{Z}$}  

A smooth family $t \mapsto (J_t,g_t)$ of almost Hermitian structures is said to be a \textit{$1$-parameter deformation of the Fubini-Study metric in $\mathcal{Z}$} if $(J_t,g_t) \in \mathcal{Z}$ and, for every time, there exists a family of Zoll families $\{\Sigma_{\sigma,t}\}_{\sigma \in \mathbb{C}P^n}$ such that the map $(\sigma,t) \mapsto \Sigma_{\sigma,t}$ is continuous in the sense of graphical convergence, and moreover $(J_0,g_0)$ and $\{\Sigma_{\sigma,0}\}_{\sigma \in \mathbb{C}P^n}$ are given by $(J_{\mathrm{can}},g_{FS})$ and $\{\mathbb{C}P^{n-1}_{\sigma}\}_{\sigma \in \mathbb{C}P^n}$. As earlier discussed, Ambrozio-Marques-Neves extensively studied this type of deformation in the context of families of co-dimension one spheres in spheres (\cite{lucas_coda_andre}). Inspired by their work, our our goal is to classify $1$-parameter deformation of the Fubini-Study metric in $\mathcal{Z}$ and study its systolic behavior.   

The first step to classify these deformations is to notice that the notion of $(J,g) \in \mathcal{Z}$ presented in the previous definition is a stronger version of the concept of belonging in $\mathcal{W}_{n-1}$. Therefore, by Theorem $\ref{classWeaklyZoll}$ we can assume that every $1$-parameter deformation of the Fubini-Study metric in $\mathcal{Z}$ consist of deformations by Hermitian structures. Hence, we can use the classical theory of deformations of complex manifolds develop by K. Kodaira (\cite{Kodaira05}) and A. Frölicher, A. Nijenhuis (\cite{Frolicher_Nijenhuis}) to obtain the following classification theorem, whose proof will be given in Section $\ref{section:zolldeformation}$.     
{
	\renewcommand{\theteo}{K}
	\begin{teo}
		Fix $n \geq 3$. Let $\mathbb{R} \ni t \mapsto (J_t, g_t) \in \mathcal{Z}$ be a smooth $1$-parameter deformation of the Fubini-Study metric in $\mathcal{Z}$. Then there exists $\varepsilon >0$ and a continuous map $(-\varepsilon, \varepsilon) \ni t \mapsto \theta(t) \in \mathrm{Diff}(\mathbb{C}P^n)$ such that, modulo isotopy, for every $t \in (-\varepsilon, \varepsilon)$ the following properties are satisfied:
		\begin{enumerate}[label=\alph*)]
			\item The almost complex structure $J_t$ is constant and equal to $J_{\mathrm{can}}$.
			\item The metric $g_t$ is balanced with respect to $J_{\mathrm{can}}$.
			\item The family $\{\Sigma_{\sigma,t}\}_{\sigma \in \mathbb{C}P^{n}}$ is given by $\left\{\mathbb{C}P^{n-1}_{\theta(t,\sigma)}\right\}_{\sigma\in \mathbb{C}P^{n}}.$ 
		\end{enumerate}
	\end{teo}
}
Combining this classification theorem with our previous analysis of co-dimension two normalized systole, we conclude that $1$-parameter deformation in $\mathcal{Z}$ of the Fubini-Study metric does not decrease the normalized systole. 
{
	\renewcommand{\theteo}{L}
	\begin{coro}
		Fix $n \geq 3$. Let $\mathbb{R} \ni t \mapsto (J_t, g_t) \in \mathcal{Z}$ be a smooth $1$-parameter deformation of the Fubini-Study metric in $\mathcal{Z}$. Then there exist an $\varepsilon>0$ such that, for every $t \in (-\varepsilon, \varepsilon)$, $$\mathrm{Sys}_{2n-2}^{\mathrm{nor}}(\mathbb{C}P^n,g_t) \geq \mathrm{Sys}_{2n-2}^{\mathrm{nor}}(\mathbb{C}P^n,g_{FS}).$$ 
	\end{coro}      
}

\section{The Class \texorpdfstring{$\mathcal{W}_k$}{Wk}}\label{section:WeaklyZollmetrics}
\subsection{Preliminaries}\label{section:WeaklyZollmetricsSubsection:Prelinaries} 

This section will be dedicated to fixing notation and recalling definitions of complex and almost complex geometry. This exposition is based in \cite{gray}.

\begin{defi}
	Let \( M \) be a smooth manifold of dimension \( 2n \). An almost complex structure on \( M \) is an endomorphism \( J \in \mathrm{Hom}(TM) \) such that \( J^2=-\mathrm{Id} \). A manifold \( (M^{2n},J) \) equipped with an almost complex structure is called an \textit{almost complex manifold}. Moreover, the almost complex  structure \( J \) is said to be \textit{integrable} if the Nijenhuis tensor \(
	\mathcal{N}_J(X,Y) \doteq [X,Y] + J[JX,Y] + J[X,JY] - [JX,JY]\) vanishes identically, in which case \( (M,J) \) is called a complex manifold.
\end{defi}

The fundamental relationship between the integrability of an almost complex structure and the existence of a holomorphic atlas is given by the Newlander–Nirenberg Theorem.

\begin{teo}\label{NNteo}[Newlander–Nirenberg]
	Let \( (M^{2n},J) \) be an almost complex manifold. Then \( \mathcal{N}_J = 0 \) if and only if \( M \) admits a holomorphic atlas \( \{(U_\alpha, \phi_\alpha)\} \) compatible with \( J \), that is, the transition functions are holomorphic and for each chart $\phi_{\alpha}: U_{\alpha}\to \phi_{\alpha}(U_{\alpha}) \subset \mathbb{C}^n$, it holds that \( (\phi_\alpha^{-1})^* J = i \), where $i$ is the canonical complex structure of $\mathbb{C}^n$.
\end{teo}

In the context of almost complex geometry, interesting Riemannian metrics to be studied are those that are compatible with the almost complex structure, in the following sense.

\begin{defi}
	Let $(M^{2n},J)$ be an almost complex manifold. A Riemannian metric $g$ on $M$ is said to be compatible with the almost complex structure $J$ (or $J$-compatible) if $g(J \cdot, J \cdot) = g( \cdot,  \cdot)$, and in this case we will write $J \in \mathrm{Iso}(TM,g)$. An almost complex manifold $(M,J,g)$ equipped with a Riemannian metric $g$ that is $J$-compatible is called an \textit{almost Hermitian manifold}.
\end{defi}

Suppose that $(M^{2n},J,g)$ is an almost Hermitian manifold. Let us define the \textit{fundamental $2$-form} associated to $(M,J,g)$:
$$\omega(\cdot,\cdot) \doteq g(J\cdot, \cdot) \in \Omega^2(M).$$
The anti-symmetry of $\omega$ is guaranteed by the compatibility condition $J \in \mathrm{Iso}(TM,g)$. In the general case, an almost Hermitian manifold does not satisfy any further compatibility condition between these two structures. However, it is worth to highlight a few conditions that arise naturally. For that, we introduce the following tensors.     

Before proceeding with definitions, we establish notation: for any Riemannian manifold $(N, g)$, we denote by $\nabla$ the Levi-Cevita connection associated with the metric $g$.

\begin{defi}
	Let $(M^{2n},J,g)$ be an almost Hermitian manifold and $X,Y \in \mathfrak{X}(M)$. We define: 
	\begin{enumerate}[label=\alph*),ref=(\alph*)] 
		\item $\mathcal{Q}(X,Y)=(\nabla_XJ)Y+(\nabla_{JX}J)JY$.
		\item $\mathcal{H}(X,Y)=(\nabla_XJ)Y-(\nabla_{JX}J)JY$. 
		\item $\mathcal{S}(X,Y)=(\nabla_X J)Y-(\nabla_Y J)X.$    
	\end{enumerate}
\end{defi}

Now, we proceed with the definition of distinct classes of almost Hermitian manifolds, which are established through compatibility conditions determined by the previously introduced tensors.
\begin{defi}\label{defiofalmoshertistrc}
	Let $(M^{2n},J,g)$ be an almost Hermitian manifold. We say that $(M,J,g)$ is:
	\begin{enumerate}[label=\alph*),ref=(\alph*)]
		\item Hermitian, if $J$ is integrable. 
		\item Kähler, if it is Hermitian and $\nabla J=0$. 
		\item Almost Kähler, if $d \omega=0$. 
		\item Quasi Kähler, if $\mathcal{Q}=0$.
		\item Balanced, if $d \omega^{n-1} = 0$.      
	\end{enumerate} 
\end{defi}

Note the following inclusions between the previously defined classes of almost Hermitian manifolds: Kähler $\subset$ almost Kähler $\subset$ quasi Kähler $\subset$ balanced. Additionally, if $n=2$, the balanced condition implies the almost Kähler condition. However, if $n>2$, all the inclusions are strict (\cite{gray16classes}). 

Another important relation for us is that every quasi Kähler manifold that is also Hermitian is Kähler. The proof of this fact is based on the following proposition.

\begin{prop}(cf. Corollary $4.2$ in \cite{gray})
	Let $(M^{2n},J,g)$ be an almost Hermitian manifold. Then $(M,J,g)$ is Hermitian if and only if $\mathcal{H}=0$. 
\end{prop}

\begin{coro}\label{QK+I=K}
	An almost Hermitian manifold that is quasi Kähler and Hermitian is Kähler.
\end{coro}
\begin{proof}
	Given $X,Y \in \mathfrak{X}(M)$, we have: $(\nabla_XJ)Y=\frac{1}{2}( \mathcal{Q}(X,Y)+\mathcal{H}(X,Y) ) = 0,$ by the previous result, as claimed. 
\end{proof}

Moving forward, we collect next some useful identities and properties in almost Hermitian geometry.    

\begin{prop}\label{basicid}
	Suppose that $(M^{2n},J,g)$ is an almost Hermitian manifold and $X,Y,Z \in \mathfrak{X}(M)$. Then: 
	\begin{enumerate}[label=\alph*),ref=(\alph*)]
		\item \label{basicid1} $(\nabla_X \omega)(Y,Z)=g((\nabla_XJ)Y,Z)$. 
		\item \label{basicid2} $(\nabla_X \omega)(Y,Z)=-(\nabla_X \omega)(Z,Y)$
		\item \label{basicid3} $(\nabla_X \omega)(JY,Z)=(\nabla_X \omega)(Y,JZ)$.
		\item \label{basicid4} $(\nabla_X \omega)(JY,Y)=0$.
		\item \label{basicid5} $\mathcal{N}_J(JX,Y)=-J \mathcal{N}_J(X,Y)$.
		\item \label{basicid6} Let $p \in M$, $v \in T_pM$ and $\{e_i,Je_i\}_{i=1}^n$ be an orthonormal basis of $T_pM$. Then the codifferential of the associated fundamental form $\omega$ is given by $\delta \omega(v)=\sum_{i=1}^ng(\mathcal{Q}(e_i,e_i),v).$    
	\end{enumerate}
\end{prop}

In order to conclude this section, we describe one of the primary tools that we will employ in this Chapter, the characterization of minimal submanifolds that are also an almost complex submanifold. This characterization can be found in the following paper of A. Gray, (\cite{gray}). Before we present this result, lets recall the definition of almost complex submanifold.     

\begin{defi}
	Let $(M^{2n},J)$ be an almost complex manifold and $\Sigma^{2k} \hookrightarrow M^{2n}$ a submanifold. We say that $\Sigma$ is an almost complex submanifold if for every $p \in \Sigma$ we have that $J(T_p \Sigma)=T_p \Sigma$.  
\end{defi}

Then, the aforementioned characterization of almost complex minimal submanifolds reads as follows.   

\begin{prop}{\label{MACchar}}(cf. Theorem $5.6$ in \cite{gray})
	Let $(M^{2n},g,J)$ be an almost Hermitian manifold, and $\Sigma^{2k} \hookrightarrow M^{2n}$ an almost complex submanifold. Then $\Sigma$ is a minimal submanifold of $(M,g)$ if and only if for every $p \in \Sigma$ and  $ v\in T_p^{\perp}\Sigma$,
	$$ \sum_{i=1}^{k}g(\mathcal{Q}(e_i,e_i),v)=0,$$
	where $\{e_i,Je_i\}_{i=1}^{k}$ is an orthonormal basis of $T_p\Sigma$. 
\end{prop}  
\begin{proof}
	In fact, the mean curvature vector $H$ of $\Sigma$ at the point $p \in \Sigma$ is given by:
	$$g(JH_p,v) = -\sum_{j=1}^{k} g (\mathcal{Q}(e_j,e_j),v),$$
	for every $v \in T^{\perp}_p\Sigma$.
\end{proof}

\subsection{The Classification Theorem} 

Before we proceed with the proof of Theorem $\ref{classWeaklyZoll}$ we recall the definition of the sets $\mathcal{W}_k$, in order to facilitate the read.    
\begin{defi}\label{defiofsetWk}
	Let $(M^{2n}, J, g)$ be a $2n$-dimensional almost Hermitian manifold, with $n\geq2$. For each fixed integer $1\leq k \leq n-1$, we define the set $\mathcal{W}_k$ as the class of almost Hermitian structures $(J, g)$ satisfying the following property: for every $(p, \Pi) \in \mathrm{Gr}_k^{J}(TM)$, there exists a minimal and almost complex submanifold $\Sigma^{2k}_{p,\Pi}$ of $M$ such that $p \in \Sigma_{p,\Pi}$ and $T_p(\Sigma_{p,\Pi}) = \Pi$. 
\end{defi}

The first step in the proof of Theorem $\ref{classWeaklyZoll}$ is to notice that the property defining the class $\mathcal{W}_k$ imposes restrictions, not only on the metric $g$, but also on the almost complex structure $J$. Indeed, we have the following well-know result. 

\begin{prop}\label{JintvsZoll}
	Let $(M^{2n},J,g)$ be an almost Hermitian manifold.
	\begin{enumerate}[label=\alph*),ref=(\alph*)]
		\item\label{JintvsZoll1}(cf. \cite{Nijenhuis_Woolf}) For every $(p,\Pi) \in \mathrm{Gr}_1^{J}(TM)$ there exists an almost complex submanifold $\Sigma^2_{p,\Pi}$ such that $T_p(\Sigma_{p,\Pi})=\Pi$.
		\item\label{JintvsZoll2}(cf. Theorem 15 in \cite{Kruglikov}) Fix $1< k \leq n-1$, if $(J,g) \in \mathcal{W}_k$ then the almost complex structure $J$ is integrable. In particular $(M,J,g)$ is Hermitian. 
	\end{enumerate} 
\end{prop}  

A direct consequence of this result is that the proof of Theorem $\ref{classWeaklyZoll}$ essentially reduces to show that $(J,g)\in \mathcal{W}_k$ implies $\mathcal{Q}=0$ for $k<n-1$ and $\delta \omega=0$ for $k=n-1$. This conclusion agrees with the intuition presented earlier in the introduction. 

The next proposition is the main step in order to translate the condition of being in $\mathcal{W}_k$ into tensorial properties on our manifold.

\begin{lema}\label{zolltensor}
	Let $(M^{2n},J,g)$ be an almost Hermitian manifold.
	\begin{enumerate}[label=\alph*),ref=(\alph*)]
		\item\label{zolltensor1} Fix $1\leq k \leq n-2$. If $(J,g) \in \mathcal{W}_k$, then $\mathcal{Q}$ is an anti-symmetric tensor.
		\item\label{zolltensor2} If $(J,g) \in \mathcal{W}_{n-1}$, then $\delta \omega =0$.  
	\end{enumerate}
\end{lema}
\begin{proof}
	{\hyperref[zolltensor1]{\textit{a})}} Fix $1\leq k \leq n-2$ and suppose that $(J,g)\in \mathcal{W}_k$.  We want to show that the tensor $\mathcal{Q}$ is anti-symmetric. But this is equivalent to prove that for all $p \in M$ and $u \in T_pM$ with unitary norm, we have that $\mathcal{Q}(u,u)=0$. So we fix $p \in M$ and $u\in T_pM$ with $|u|_g=1$. First, we observe that, in general, $g(\mathcal{Q}(u,u),w)=0$, for every $w \in \mathrm{Span}\{u,Ju\}$. Indeed, suppose that $w=au+bJu$, for $a,b \in \mathbb{R}$. Then using Proposition \ref{basicid} $\ref{basicid1}$, we see that: 
	\begin{align*}
		g(\mathcal{Q}(u,u),w)= &a(\nabla_u\omega)(u,u)+b(\nabla_u \omega)(u,Ju) \\
		+&a(\nabla_{Ju}\omega)(Ju,u)+b(\nabla_{Ju} \omega)(Ju,Ju).
	\end{align*}
	However, by $\ref{basicid2}$ and $\ref{basicid4}$ of Proposition \ref{basicid}, each term of the right hand side vanishes, which proves our claim. 
	
	In light of these observations, it remains to prove that $g(\mathcal{Q}(u,u),v)=0$ for every $v$ orthogonal to $\mathrm{span}\{u,Ju\}$, with $|v|_g=1$. 
	
	Fix such $v \perp \mathrm{span}\{u,Ju\}$. By definition of $v$, it is always possible to find an orthonormal basis $\{u,Ju,e_1,Je_1,...,e_{n-1},Je_{n-1}\}$ of $T_pM$ with $e_1=v$. Writing $\gamma_j=g(\mathcal{Q}(e_j,e_j),v)$, we have to prove that $\gamma_1=0$. In fact, we will prove at once that $\gamma_j=0$ for every $1 \leq j \leq n-1$.    
	
	Take $\mathcal{I}=\{1 \leq i_1 < ... < i_k \leq n-1\} $. By the definition of $\mathcal{W}_k$, there exists a minimal almost complex submanifold $\Sigma^{2k}_{\mathcal{I}}$ of $M$, such that $p \in \Sigma_{\mathcal{I}}$ and $ T_p\Sigma_{\mathcal{I}}= \mathrm{span}\{e_{i_1},Je_{i_1},...,e_{i_k},Je_{i_k}\}$. Then, noticing that $v \perp T_p\Sigma_{\mathcal{I}}$ and applying Proposition \ref{MACchar}, we have that  
	$$ \sum_{\mu=1}^{k}g(\mathcal{Q}(e_{i_{\mu}},e_{i_{\mu}}),v)=0.$$
	
	By the definition of $\gamma_j$, we obtain the following system of equations in terms of $\gamma_j$:   
	$$\gamma_{i_1}+...+\gamma_{i_k}=0, \; \forall \; \mathcal{I}=\{1 \leq i_1 < ... < i_k \leq n-1\}. $$
	Since $k<n-1$, the only solution of this system is the trivial one,  that is $\gamma_j=0$ for every $1 \leq j \leq n-1$, as claimed. 
	
	{$\hyperref[zolltensor2]{\textit{b)}}$} We will prove that for a point $p \in M$ and $u \in T_pM$ with unitary norm, we have that $\delta \omega(u)=0$. Using Proposition \ref{basicid} $\ref{basicid6}$ this is equivalent to $\sum_{i=1}^ng(\mathcal{Q}(e_i,e_i),u) =0$, where $\{e_i,Je_i\}_{i=1}^n$ is an orthonormal basis of $T_pM$. Since $|u|_g=1$, we can suppose that $e_1=u$. Arguing as in the beginning of the proof of the first item, we have that $g(\mathcal{Q}(u,u),u)=0$. Therefore, is enough to show that     
	\begin{equation}\label{eqgernbl}
		\sum_{i=1}^ng(\mathcal{Q}(e_i,e_i),v) =0,
	\end{equation}
	Applying the hypothesis that $(J,g) \in \mathcal{W}_{n-1}$ together with $u \perp \mathrm{span}\{e_j,Je_j\}_{j\geq2}$, we conclude the existence of a minimal almost complex submanifold $\Sigma_u$, such that $p\in \Sigma_u$ and $T_p \Sigma_u=\mathrm{span}\{e_j,Je_j\}_{j\geq2}$. Hence, $\eqref{eqgernbl}$ follows from Proposition $\ref{MACchar}$, completing the proof.
\end{proof}

The previous proposition covers the case were $1<k=n-1$. However, for $1 \leq k<n-1$, it is still necessary to prove that $\mathcal{Q}$ being anti-symmetric implies that $\mathcal{Q}$ is zero. For that, we need to understand how the tensor $\nabla J$ behaves under the commutation of the first and second variables. Recalling that $\mathcal{S}$ is the anti-symmetrization of $\nabla J$, we present the following. 

\begin{lema}
	Let $(M^{2n},J,g)$ be an almost Hermitian manifold. Given $X,Y \in \mathfrak{X}(M)$ we have 
	$$\mathcal{Q}(X,Y)+\mathcal{Q}(Y,X)=-2J\mathcal{S}(JX,Y)-J\mathcal{N}_J(X,Y).$$ 
\end{lema}
\begin{proof}
	The proof is a direct computation. By definition of $\nabla J$ and the symmetry of the connection, we have 
	\begin{align*}
		\mathcal{Q}(X,Y)+\mathcal{Q}(Y,X) &= (\nabla_XJ)Y + (\nabla_{JX}J)JY + (\nabla_YJ)X + (\nabla_{JY}J)JX \\ 
		&= \nabla_XJY-J\nabla_XY + \nabla_{JX}J^2Y - J \nabla_{JY}{JX} \\
		& \;\;\;\;+ \nabla_YJX-J\nabla_YX + \nabla_{JY}J^2X - J \nabla_{JX}{JY} 	\\
		&= \{\nabla_XJY-\nabla_{JY}X\} -\{\nabla_{JX}Y-\nabla_{Y}JX\}   \\ 
		&\;\;\;\;- J\{\nabla_{JX}JY+ \nabla_YX+ \nabla_{JY}JX + \nabla_XY\} \\ 
		&=[X,JY]-[JX,Y]-2J\{\nabla_{JX}JY+\nabla_YX\} \\
		&\;\;\;\;-J\{[JY,JX]+[X,Y]\}.
	\end{align*}
	On the other hand 
	\begin{align*}
		\nabla_{JX}JY+\nabla_YX &= \nabla_{JX}JY-\nabla_YJ^2X \\
		&=(\nabla_{JX}J)Y+J\nabla_{JX}Y-(\nabla_YJ)JX-J\nabla_YJX \\ 
		&=\mathcal{S}(JX,Y)+J[JX,Y].  
	\end{align*}
	Therefore combining this two equations we have that
	\begin{align*}
		\mathcal{Q}(X,Y)+\mathcal{Q}(Y,X) &=[X,JY]+[JX,Y]-2J\mathcal{S}(JX,Y)-J\{[JY,JX]+[X,Y]\} \\
		&= -2J\mathcal{S}(JX,Y)-J\mathcal{N}_J(X,Y), 
	\end{align*}
	concluding the proof.  
\end{proof}

\begin{coro}\label{PropofTenS}
	Let $(M,J,g)$ be an almost Hermitian manifold and suppose that $\mathcal{Q}$ is an anti-symmetric tensor. Then we have the following identities 
	\begin{enumerate}[label=\alph*),ref=(\alph*)]
		\item $(\nabla_{JX}J)Y=(\nabla_{Y}J)JX- \frac{1}{2}\mathcal{N}_J(X,Y)$
		\item $(\nabla_{X}J)Y=(\nabla_{Y}J)X- \frac{1}{2}J\mathcal{N}_J(X,Y)$.  
	\end{enumerate}  
\end{coro}

Now using this corollary and Proposition $\ref{basicid}$ we can prove a refinement of Proposition $\ref{zolltensor}$.

\begin{prop}\label{KASiffKN}
	Let $(M^{2n},J,g)$ be an almost Hermitian manifold satisfying $(J,g)\in \mathcal{W}_k$, for some fixed integer $1\leq k<n-1$. Then $\mathcal{Q}=0$.  
\end{prop}  
\begin{proof}
	Take $X,Y,Z \in \mathfrak{X}(M)$. By Proposition $\ref{basicid}$ and Corollary $\ref{PropofTenS}$ we have the following identities 
	\begin{align*}
		(\nabla_{JX} \omega)(JY,Z) &= (\nabla_{JX} \omega)(Y,JZ) \\
		&=g((\nabla_{JX}J)Y,JZ) \\
		&=g\left((\nabla_{Y}J)JX-\frac{1}{2}\mathcal{N}_{J}(X,Y),JZ \right)\\
		&=-g((\nabla_YJ)X,Z) -\frac{1}{2}g({N}_{J}(X,Y),JZ) \\
		&=-g((\nabla_XJ)Y,Z)-\frac{1}{2}g(J{N}_{J}(X,Y),Z) -\frac{1}{2}g({N}_{J}(X,Y),JZ)\\
		&=-(\nabla_X\omega)(Y,Z).   
	\end{align*}   
	That is, for every $X,Y$ and $Z$ in $\mathfrak{X}(M)$
	$$g(\mathcal{Q}(X,Y),Z)=g((\nabla_{JX}J)JY + (\nabla_{X}J)Y,Z)=0,$$
	implying that $\mathcal{Q}=0$, as claimed.
\end{proof}

Finally, we concatenate all the previous results to give a proof of Theorem $\ref{classWeaklyZoll}$.

\begin{proof}[Proof of Theorem $\ref{classWeaklyZoll}$]
	{$\hyperref[classWeaklyZoll1]{\textit{a)}}$} If  $(J,g)\in \mathcal{W}_1$, Proposition $\ref{KASiffKN}$ implies that $\mathcal{Q}=0$, so by definition $(M,J,g)$ is quasi Kähler. Conversely, suppose that $(M,J,g)$ is quasi Kähler. By Proposition $\ref{JintvsZoll}$ $\ref{JintvsZoll1}$ we have a family $\{\Sigma_{p,\Pi}\; | \; {(p,\Pi) \in \mathrm{Gr}_1^{J}(TM)}\}$ of almost complex submanifolds of $M$. By Proposition $\ref{MACchar}$ every almost complex submanifold of a quasi Kähler manifold is minimal.
	
	{$\hyperref[classWeaklyZoll2]{\textit{b)}}$} Fix $1<k<n-1$. Suppose that $(J,g) \in \mathcal{W}_k$. By Proposition $\ref{JintvsZoll}$ $\ref{JintvsZoll2}$ and Proposition $\ref{KASiffKN}$, we have that $(M,J,g)$ is Hermitian and quasi Kähler. Therefore, Corollary $\ref{QK+I=K}$ implies that $(M,J,g)$ is Kähler. 
	
	Conversely, assume $(M,J,g)$ is Kähler and let us show that $(J,g) \in \mathcal{W}_{k}$. By Proposition~\ref{MACchar}, it suffices to construct a family $\{\Sigma_{p,\Pi} : (p,\Pi) \in \mathrm{Gr}_k^{J}(TM)\}$ of complex submanifolds with $p\in \Sigma_{p,\Pi}$ and $T_p\Sigma_{p,\Pi}=\Pi$. Since $\Sigma_{p,\Pi}$ is not required to be complete, we may work locally. Fix $(p,\Pi) \in \mathrm{Gr}_k^{J}(TM)$. By the Newlander–Nirenberg Theorem (Theorem~\ref{NNteo}), there exists a local biholomorphism $\phi: (U,J) \to (\mathbb{C}^n,i)$ with $\phi(p)=0$. We then define $\Sigma_{p,\Pi}$ as the image under this biholomorphism of the complex plane $d\phi_p(\Pi) \subset \mathbb{C}^n$. 
	
	{$\hyperref[classWeaklyZoll3]{\textit{c)}}$} Fix $n \geq 3$. If $(J,g) \in \mathcal{W}_{n-1}$, then the desired conclusion follows immediately by Propositions $\ref{JintvsZoll}$ and $\ref{zolltensor}$. Now, assume that $(M,J,g)$ is balanced and Hermitian. Since $(M,J)$ is a complex manifold, we can argue as before and produce a family $\{\Sigma_{p,\Pi} \; | \;{(p,\Pi) \in \mathrm{Gr}_{n-1}^{J}(TM)}\}$ of complex submanifolds of $(M,J,g)$ satisfying $p\in \Sigma_{p,\Pi}$ and $T_p\Sigma_{p,\Pi}=\Pi$. Using $\delta\omega=0$ together with Proposition $\ref{basicid}$ $\ref{basicid6}$ and Proposition $\ref{MACchar}$, we easily check that every $(2n-2)$-dimensional complex submanifold of $(M,J,g)$ is minimal.                           
\end{proof}

\section{Systole of Homogeneous Metrics}\label{section:homogeneousmetrics}
This section will be structured as follows. In Section $\ref{descofhommet}$, we will present a description of the homogeneous metrics along with its properties. Section $\ref{2systole}$ contains the proofs of Theorems $\ref{systofhommetrics}$ and $\ref{confclassofhommetr}$ for the dimension two case, while Section $\ref{4nsystole}$ focuses on the proofs for the codimension two case.  

\subsection{Construction of Homogeneous Metrics}\label{descofhommet}

\indent	 As mentioned in the introduction, W. Ziller (\cite{Ziller1982}) classified the homogeneous metrics on complex projective space. Specifically, he proved that the only group acting transitively on $\mathbb{C}P^m$ with non trivial isotropy representation is $\mathrm{Sp}(n+1)$, for $m=2n+1$. The main objective of this preliminary section is to established that each of these metrics is balanced with respect to the canonical complex structure. To accomplish that, we first describe this action along with a detailed construction of the associated homogeneous metrics. 

Recall that the group $\mathrm{Sp}(n+1) \subset \mathrm{U}(2n+2)$, for $n\geq 1$, is given by 

\[\mathrm{Sp}(n+1)=\left\{
U= \begin{pmatrix}
	A & -\bar{B} \\ B & \bar{A} 
\end{pmatrix} \; : \; A,B \in \mathbb{M}_{n+1}(\mathbb{C}),\; U^*U= \mathrm{Id}  \right\}. \]      
This group acts transitively on $\mathbb{S}^{4n+3} \subset \mathbb{C}^{2n+2}$, where the stabilizer subgroup of $\textbf{e}_1=(1,0,...,0)$ is isomorphic to $\mathrm{Sp}(n)$. Since $\mathrm{Sp}(n+1) \subset \mathrm{U}(2n+2)$, this action induces a transitive action on $\mathbb{C}P^{2n+1}$, with base point $o=[\textbf{e}_1]$, and stabilizer group $\mathrm{Sp}(n) \times \mathrm{U}(1)$, where 
\[\mathrm{Sp}(n)\times \mathrm{U}(1) =\left\{
\begin{pmatrix}
	e^{i\theta} & 0 \\ 0 & U_0
\end{pmatrix} \in \mathrm{Sp}(n+1) : 
U_0 \in \mathrm{Sp}(n), \; e^{i\theta} \in \mathrm{U}(1) \right\}.\]             
Hence, $\mathbb{C}P^{2n+1}$ has the structure of the homogeneous space $\mathrm{Sp}(n+1)/\mathrm{Sp}(n) \times \mathrm{U}(1)$. At the level of Lie algebras, we have a decomposition $\mathfrak{sp}(n+1)=\mathfrak{sp}(n) \times \mathfrak{u}(1) \oplus \mathfrak{m}$, where we can identify $\mathfrak{m}$ with $T_o\mathbb{C}P^{2n+1}$. Moreover, $\mathfrak{m}$ can be chosen invariant under the adjoint action of $\mathrm{Sp}(n) \times \mathrm{U}(1)$ on $\mathfrak{sp}(n+1)$, and this action induces an irreducible decomposition $\mathfrak{m}=\mathfrak{m}_0\oplus \mathfrak{m}_1$. Explicitly, these spaces are given by: 

\begin{equation}\label{lie(sp(2))}
	\mathfrak{sp}(n+1)=\left\{ \begin{pmatrix}
		X & -Y^* \\ Y & -X^T 
	\end{pmatrix} \; : \; X,Y \in \mathbb{M}_{n+1}(\mathbb{C}),\; Y=Y^T,\;X^*=-X  \right\};
\end{equation}      

\[\mathfrak{m}_0=\left\{ \begin{pmatrix}
	0 & -Y^* \\ Y & 0 
\end{pmatrix} \; : \; 
Y=\begin{pmatrix}
	y & 0 \\ 0 & 0 
\end{pmatrix} ,\; y \in \mathbb{C}  \right\}; \]  

\[\mathfrak{m}_1=\left\{ \begin{pmatrix}
	X & -Y^* \\ Y & -X^T 
\end{pmatrix} \; : \; 
X=\begin{pmatrix}
	0 & -\bar{z} \\ z^T & 0 
\end{pmatrix},\;
Y=\begin{pmatrix}
	0 & {w} \\ w^T & 0 
\end{pmatrix},\;
z,w \in \mathbb{C}^n
\right\}. \]  

Finally the identification $\mathfrak{m} \cong T_o \mathbb{C}P^{2n+1}$ is given by
\begin{equation}\label{m=ToCP3}
	\begin{aligned}
		\mathfrak{m} &\to T_o \mathbb{C}P^{2n+1}\subset T_{\textbf{e}_1}\mathbb{S}^{4n+3} \\ 
		(y,z,w) & \mapsto (0,z,y,w).  
	\end{aligned}   
\end{equation}

Once we have all the proper identifications, it is trivial to verify the next result. 

\begin{prop}
	With respect to the Fubini-Study metric $g_{FS}$ on $\mathbb{C}P^{2n+1}$, the decomposition $\mathfrak{m}=\mathfrak{m}_0\oplus \mathfrak{m}_1$ is orthogonal. Moreover, the induced metrics on $\mathfrak{m}_0$ and $\mathfrak{m}_1$ are invariant by the adjoint action of $\mathrm{Sp}(n)\times \mathrm{U}(1)$.   
\end{prop}

\begin{obs}
	In what follows, $g_{FS}$ will always denote the Fubini-Study metric on the complex projective space and $\Omega$ will denote the associated fundamental form. Moreover, we assume that the Fubini-Study metric is normalized to satisfy $\int_{\mathbb{C}P^1} \Omega=1$.     
\end{obs}

We are now in position to introduce the family of homogeneous metrics in $\mathbb{C}
P^{2n+1}$. The invariant decomposition $\mathfrak{m}=\mathfrak{m}_0\oplus \mathfrak{m}_1$ suggests the following family of metrics on $\mathfrak{m}$:   
$$\restr{g_{t}}{\mathfrak{m}}=t\restr{g_{FS}}{\mathfrak{m}_0} +\restr{g_{FS}}{\mathfrak{m}_1}, $$
for $t \in \mathbb{R}_{>0}$. As a consequence of the previous propositions these metrics extend to a family of Riemannian metrics on $\mathbb{C}P^{2n+1}$, which we will denote by $\{g_t\}_{t \in \mathbb{R}_{>0}}$. Furthermore, this family exhaust the set of homogeneous metric on $\mathbb{C}P^{2n+1}$, up to isometries and homothety, as proved by W. Ziller in \cite{Ziller1982}.      

In what follows, we present an alternative construction for this family. First, we note that the inclusions $\mathrm{Sp}(n) \times \mathrm{U}(1) \subset \mathrm{Sp}(n) \times \mathrm{Sp}(1) \subset \mathrm{Sp}(n+1),$ induces a fibration:   
$$\frac{\mathrm{Sp(1)}}{\mathrm{U}(1)} \to \frac{\mathrm{Sp(n+1)}}{\mathrm{Sp}(n) \times\mathrm{U}(1)} \xrightarrow{\pi} \frac{\mathrm{Sp(n+1)}}{\mathrm{Sp}(n) \times\mathrm{Sp}(1)}.$$
This fibration is know as the \textit{Penrose fibration}. Up to canonical identifications, it is given by:
\begin{equation}\label{pensrosefibration}
	\begin{aligned}
		\mathbb{C}P^1 \to \mathbb{C}P^{2n+1} &\xrightarrow{\pi} \mathbb{H}P^n\\
		[z_0:...:z_n:w_0:...:w_n] &\mapsto [z_0 +w_0j:...:z_n+w_nj]. 
	\end{aligned}
\end{equation}

The relation between the Penrose fibration and the aforementioned invariant decomposition of the tangent space of $\mathbb{C}P^{2n+1}$ can be understood in the subsequent manner. Let $\Lambda^0=\mathrm{ker}d\pi$ be the horizontal distribution defined by the submersion $\pi$, and $\Lambda^1$ its orthogonal complement with respect to the Fubini-Study metric. Given $p\in \mathbb{C}P^{2n+1}$ and $U \in \mathrm{Sp}(n+1)$, with $U\cdot o =p$, we have
$$\Lambda^0_p=\restr{dL_U}{o}(\mathfrak{m}_0), \; \Lambda^1_p=\restr{dL_U}{o}(\mathfrak{m}_1),$$
where $L_{U}:  \mathbb{C}P^{2n+1} \to \mathbb{C}P^{2n+1}$ is the left action induced by the multiplication of $U \in  \mathrm{Sp}(n+1)$.   

In particular, $\Lambda^0_o = \mathfrak{m}_0$ and $\Lambda^1_o = \mathfrak{m}_1$. Consequently, the family of metrics $\{g_t\}_{t\in \mathbb{R}_{>0}}$ can be expressed as:
\begin{equation}\label{formulaforgt}
	g_t= t g_0+g_1,
\end{equation}
where $g_0 \doteq \restr{g_{FS}}{\Lambda^0}$ and $g_1 \doteq \restr{g_{FS}}{\Lambda^1}$. 

An immediate consequence of this approach, is that for the Fubini-Study metric $g_{\mathbb{H}P^n}$ of $\mathbb{H}P^n$ the projection $\pi: (\mathbb{C}P^{2n+1},g_t) \to (\mathbb{H}P^n,g_{\mathbb{H}P^n})$ is a Riemannian submersion for every $t \in \mathbb{R}_{>0}$.    

Notice the similarity in construction between the family of metrics $\{g_t\}_{t\in \mathbb{R}_{>0}}$ and the Berger metrics on $\mathbb{R}P^3$. For instance, the parameter $t>0$ controls the volume of the fiber of the Penrose fibration. This comparison allows us to draw parallels between our results and those presented by L. Ambrozio and R. Montezuma in \cite{lucas_rafael_sistole_projcspace}.

Subsequently we focus on proving that $(\mathbb{C}P^{2n+1},J_{\mathrm{can}},g_t)$ is balanced for every $t>0$. We begin by justifying the compatibility condition of the canonical complex structure $J_{\mathrm{can}}$ with the metrics $\{g_t\}_{t \in \mathbb{R}_{>0}}$, and describing its fundamental forms.

Using the identification $(\ref{m=ToCP3})$, the decomposition $\mathfrak{m}=\mathfrak{m}_0\oplus \mathfrak{m}_1$ is preserved by the canonical complex structure $J_{\mathrm{can}}$ and by the family $\{g_t\}_{t \in \mathbb{R}_{>0}}$. Consequently, $(\mathbb{C}P^{2n+1}, J_{\mathrm{can}}, g_t)$ defines a Hermitian manifold. Furthermore, the decomposition $T\mathbb{C}P^{2n+1}=\Lambda^0\oplus \Lambda^1$ also enjoys this invariance. Therefore, denoting the orthogonal projections onto the spaces $\Lambda^i$, by $\Pi_i: T\mathbb{C}P^{2n+1}\to \Lambda^i$, we can decompose $\Omega$, the fundamental form of the Fubini-Study metric, in the following factors:     
\begin{equation}\label{omega0omega1}
	\Omega_0(\cdot,\cdot)=\Omega(\Pi_0 \cdot,\Pi_0\cdot),\;\Omega_1(\cdot,\cdot)=\Omega(\Pi_1 \cdot,\Pi_1\cdot).
\end{equation}
It follows straightaway from the definition of the family of homogeneous metrics $\{g_t\}_{t \in \mathbb{R}_{>0}}$ on $\mathbb{C}P^{2n+1}$ that the associated fundamental forms are given by: 
$$\omega_t(\cdot,\cdot)\doteq g_t(J \cdot, \cdot)=t\Omega_0(\cdot,\cdot)+\Omega_1(\cdot,\cdot),$$  
for every $t>0$.

The previous decompositions provide the necessary tools to prove the balanced property. 

\begin{lema}\label{propW0andW1}
	If $\pi:\left(\mathbb{C}P^{2n+1},g_{FS}\right) \to \left(\mathbb{H}P^n,g_{\mathbb{H}P^n}\right)$ is the Penrose fibration, then $\Omega_0^2 = 0$ and $\Omega_1^{2n}= (2n)!\pi^*d{V}_{{g}_{\mathbb{H}P^n}}$.
\end{lema}  
\begin{proof}
	Take $X_1,...,X_4 \in \mathfrak{X}(\mathbb{C}P^{2n+1})$. By definition of $\Omega_0$,   
	$$\Omega_0^2(X_1,...,X_4)=\Omega^2(\Pi_0X_1,...,\Pi_0X_4).$$
	However the vector bundle $\Lambda^0$  has rank $2$, so that $\{\Pi_0 X_1,...,\Pi_0X_4\}$ must be a linear dependent set. Therefore $\Omega_0^2(X_1,...,X_4)=0$, as desired.
	
	Now fix $p \in \mathbb{C}P^{2n+1}$. Since $\mathrm{ker}d\pi=\Lambda^0$, the linear forms $\restr{\Omega_1^{2n}}{p}$ and $\restr{\pi^*d{V}_{{g}_{\mathbb{H}P^n}}}{p}$ are of top degree in $\Lambda^1_p$. Hence, there must exist $a \in \mathbb{R}$ such that: 
	$$\restr{\Omega_1^{2n}}{p} = a \restr{\pi^*d{V}_{{g}_{\mathbb{S}^4}}}{p}.$$ 
	Evaluating these forms on a complex orthonormal basis of $\Lambda^1_p$, and and recalling that $\restr{d\pi}{p}: (\Lambda^1_p,g_{FS}) \to (T_{\pi(p)}\mathbb{H}P^n,g_{\mathbb{H}P^n})$ is an isometry that preserves orientation, we conclude that $a=(2n)!$, completing the proof.         
\end{proof}

\begin{prop}\label{gtisbalanced}
	For every $t \in \mathbb{R}_{>0}$, the Hermitian manifold $(\mathbb{C}P^{2n+1},J_{can},g_t)$ is balanced. 
\end{prop}
\begin{proof}
	We will check that $d \omega_t^{2n}=0$ for every $t \in \mathbb{R}_{>0}$. By Proposition $\ref{propW0andW1}$   
	\begin{align*}
		\omega_t^{2n} &= (t \Omega_0 + \Omega_1)^{2n} \\
		&= \sum_{k=0}^{2n} \binom{2n}{k} t^k\Omega_0^k \Omega_1^{2n-k} \\
		&= 2nt\,\Omega_0 \Omega_1^{2n-1} + \Omega_1^{2n} \\ 
		&= 2nt\,\Omega_0 \Omega_1^{2n-1} + (2n)!\,\pi^*d{V}_{{g}_{\mathbb{H}P^n}}. 
	\end{align*} 
	Therefore, $d\omega_t^{2n} = 2nt\,d(\Omega_0 \Omega_1^{2n-1})$ for every $t \in \mathbb{R}_{>0}$. However, for the Fubini-Study metric $\Omega=\omega_1$, we have 
	$$ 0 = \frac{1}{2n}d\omega_1^{2n} = d(\Omega_0 \Omega_1^{2n-1}).$$
	Hence, $d\omega_t^{2n} = 2nt\,d(\Omega_0 \Omega_1^{2n-1}) =0$ for every $t>0$.     
\end{proof}

\begin{obs}
	For the case $n=1$, the Hermitian manifold $(\mathbb{C}P^3,J_{\mathrm{can}},g_t)$ can be viewed as the Twistor space over the anti-self-dual manifold $(\mathbb{S}^4,g_{\mathrm{can}})$. Therefore, (\cite{Jixiang15}, Theorem $3.1$) gives another proof of the fact that this space is balanced.    
\end{obs}

\subsection{2-Systole}\label{2systole}

Having established the notation and properties of the family $\{g_t\}_{t \in \mathbb{R}_{>0}}$ of homogeneous metrics on the complex projective space $\mathbb{C}P^{2n+1}$, for $n \geq 1$, we now proceed to demonstrate Theorems $\ref{systofhommetrics}$ and $\ref{confclassofhommetr}$ for the dimension two systole case.

We intend to prove a stronger version of Theorem $\ref{systofhommetrics}$ by explicitly exhibiting the submanifold that realizes the systole. Taking inspiration in the well-studied case of the Fubini-Study metric \cite{berger} and \cite{gromov_systole}, together with the fact that the homogeneous family $\{g_t\}_{t \in \mathbb{R}_{>0}}$ is parameterized by the volume of the fiber of the Penrose fibration, it is intuitive to guess that, for $t\leq 1 $, the systole should be achieved at the fiber of this fibration. On the other hand, since there exists a linear projective plane in $\mathbb{C}P^{2n+1}$ with tangent bundle contained in the distribution $\Lambda^1$ (see Proposition $\ref{2sysbound}$), the intuition suggests that this linear projective plane should realize the systole for $t \geq 1$.       

Subsequently, we properly verify that the intuitions above are corrected. This entails analyzing the two distinct cases $t \leq 1$ and $t \geq 1$. As suggested, these cases differ significantly in nature, and their dichotomy will persist throughout the section. We begin by exhibiting the aforementioned linear projective plane.              

\begin{prop}\label{2sysbound}
	There exists a linear projective plane $\mathbb{C}P^1_T\subset \mathbb{C}P^{2n+1}$ such that $T\mathbb{C}P^1_T \subset \Lambda^1$. Moreover, there exists a subgroup $\mathrm{Sp}_{T}(1)$ of $\mathrm{Sp}(n+1)$ isomorphic to $\mathrm{Sp}(1)$, such that $\mathbb{C}P^1_T$ is invariant under its action, and the action is transitive.         
\end{prop}     
\begin{proof}
	Define $\mathbb{C}P^1_T=\{[p_0:p_1:0:...:0]\in \mathbb{C}P^{2n+1}\}$. Therefore, $o=[\textbf{e}_1] \in \mathbb{C}P^1_T$ and 
	$$T_o \mathbb{C}P^1_T=\{(0,\xi,0,0)\in T_{\textbf{e}_1}\mathbb{S}^{4n+3}\; :\; \xi \in \mathbb{C} \} \subset \mathfrak{m}_1.$$
	Consider the subgroup $\mathrm{Sp}_T(1) \subset \mathrm{Sp}(n+1)$, given by: 
	$$\mathrm{Sp}_T(1)=\left\{
	\begin{pmatrix}
		A & 0 \\ 0 & \bar{A} 
	\end{pmatrix} \; : \; A = \begin{pmatrix}
		A_0 & 0 \\ 0 & \mathrm{Id}_{n-1} 
	\end{pmatrix}, \;  A_0 \in \mathrm{Sp}(1) \right\}.$$
	Clearly $\mathbb{C}P^1_T$ is invariant by the action of the subgroup $\mathrm{Sp}_T(1)$ of $\mathrm{Sp}(n+1)$, and moreover the action is transitive. Hence for each $p \in \mathbb{C}P^1_T$ exist $U \in \mathrm{Sp}_T(1)$, such that $U\cdot o = p$, and then $T_p \mathbb{C}P^1_T = \restr{dL_U}{o}(T_o\mathbb{C}P^1_T) \subset \Lambda^1_p$, as desired.   
\end{proof}

Formalizing our intuition above, our candidates to realize the two-dimensional systole of the family $\{ g_{t} \}_{t \in \mathbb{R}_{>0}}$ are $\mathbb{C}P^1_b \doteq \pi^{-1}(b)$ for $t \leq 1$ and $\mathbb{C}P^1_T$ for $t \geq 1$, where $\pi: \mathbb{C}P^{2n+1} \to \mathbb{H}P^n$ is the Penrose fibration and $b \in \mathbb{H}P^n$.  Incidentally, we observe that, by the Koszul Formula, these families of linear projective planes are totally geodesic in $(\mathbb{C}P^{2n+1},g_t)$ for every $t \in \mathbb{R}_{>0}$.

Now that we have well-understood our contestants to realize the systole, and since the volume of $(\mathbb{C}P^{2n+1},g_t)$ can be readily computed to be $\mathrm{vol}_{g_t}(\mathbb{C}P^{2n+1})=t\mathrm{vol}_{g_{FS}}(\mathbb{C}P^{2n+1})$, for every $t\in \mathbb{R}_{>0}$, we can formulate the following refined version of Theorem $\ref{systofhommetrics}$.
\begin{teo}\label{2systofhommetrics}
	Let $\pi: \mathbb{C}P^{2n+1} \to \mathbb{H}P^n$ be the Penrose fibration, and for every $b \in \mathbb{H}P^n$ set $\mathbb{C}P^1_{b} = \pi^{-1}(b)$. Hence
	\begin{enumerate}[label=\alph*),ref=(\alph*)]
		\item If $0< t \leq 1$ then $\mathrm{Sys}_2(\mathbb{C}P^{2n+1},g_t) = |\mathbb{C}P^1_b|_{g_t} = t$;
		\item If $t \geq 1$ then $\mathrm{Sys}_2(\mathbb{C}P^{2n+1},g_t) = |\mathbb{C}P^1_T|_{g_t} = 1$. 
	\end{enumerate}          
\end{teo}    

We set forth the proof noticing that $\mathbb{C}P^1_{b}$ and $\mathbb{C}P^1_{T}$ are linear projective planes in $\mathbb{C}P^{2n+1}$, for every $b \in \mathbb{H}P^n$. Therefore, their homology classes are non-trivial, and the following bound follows by the definition of systole: 
\begin{equation}\label{initbounfor2sys}
	\mathrm{Sys}_2(\mathbb{C}P^{2n+1},g_t) \leq \min \{|\mathbb{C}P^1_b|_{g_t},|\mathbb{C}P^1_T|_{g_t} \}.
\end{equation}
This simple observation leads to the following result.

\begin{lema}\label{2systofhommetricslemma}
	For every $b \in \mathbb{H}P^n$ and $t>0$, we have: 
	\begin{enumerate}[label=\alph*),ref=(\alph*)]
		\item $\mathrm{Sys}_2(\mathbb{C}P^{2n+1},g_t) \leq |\mathbb{C}P^1_b|_{g_t} = t$.
		\item $\mathrm{Sys}_2(\mathbb{C}P^{2n+1},g_t) \leq |\mathbb{C}P^1_T|_{g_t} = 1$. 
	\end{enumerate}  
\end{lema}
\begin{proof}
	From $(\ref{initbounfor2sys})$ it is clear that the desired result follows by computing the volume of these submanifolds. Since $\mathbb{C}P^1_b$ is a complex submanifold of $(\mathbb{C}P^{2n+1}, J_{\mathrm{can}})$, with $T\mathbb{C}P^1_b \subset \Lambda^0$ for every $b \in \mathbb{H}P^n$, we have
	\begin{align*}
		|\mathbb{C}P_b^1|_{g_t} =\int_{\mathbb{C}P_b^1}\omega_t =t\int_{\mathbb{C}P_b^1}\Omega_0 =t\int_{\mathbb{C}P_b^1}\Omega =t
	\end{align*}
	for every $b \in \mathbb{H}P^n$ and $t{>0}$.
	Again, $\mathbb{C}P^1_b$ is a complex submanifold of $(\mathbb{C}P^{2n+1}, J_{\mathrm{can}})$. However, since $\mathbb{C}P^1_T$ is transversal to the fiber of the Penrose fibration, the following identity holds
	\begin{align*}
		|\mathbb{C}P_T^1|_{g_t} =\int_{\mathbb{C}P_T^1}\omega_t =\int_{\mathbb{C}P_T^1}\Omega_1 =\int_{\mathbb{C}P_T^1}\Omega =1,
	\end{align*} 
	for every $t>0$.
\end{proof}

It is clear, by Lemma $\ref{2systofhommetricslemma}$, that Theorem $\ref{2systofhommetrics}$ is equivalent to equality in $(\ref{initbounfor2sys})$. In order to prove that the equality must hold, we will follow the approach presented in \cite{gromov_systole} and show that if a closed $2$-cycle $C \subset \mathbb{C}P^{2n+1}$ has less area than the bound given in $(\ref{initbounfor2sys})$, then $C$ has a trivial homology class in $H_2(\mathbb{C}P^{2n+1},\mathbb{Z})$. The foundation of this argument is the following Crofton formula.

\begin{lema}
	Let $C \subset \mathbb{C}P^{2n+1}$ be a closed $2$-cycle. Then 
	$$[C]\cdot [\mathbb{C}P^{2n}]=\int_{C} \Omega,$$
	where $\cdot:H_2(\mathbb{C}P^{2n+1},\mathbb{Z})\times H_{4n}(\mathbb{C}P^{2n+1},\mathbb{Z}) \to \mathbb{Z}$ denotes the intersection pairing.   
\end{lema}
\begin{proof}
	Let $C$ be a closed $2$-cycle. Since $[\mathbb{C}P^1]$ is the generator of $H_2(\mathbb{C}P^{2n+1},\mathbb{Z})$, there exists a $3$-chain $R$ and an integer $k$ such that, in homology, $C=k \mathbb{C}P^1 + \partial R$. Consequently
	$$[C]\cdot [\mathbb{C}P^{2n}]=k[\mathbb{C}P^1]\cdot[\mathbb{C}P^{2n}] + [\partial R]\cdot[\mathbb{C}P^{2n}]=k,$$
	since $[\mathbb{C}P^1]\cdot[\mathbb{C}P^{2n}]=1$ and $[\partial R]=0$. On the other hand, by Stokes' Theorem, 
	$$\int_C \Omega = \int_{k \mathbb{C}P^1 + \partial R}\Omega = k \int_{\mathbb{C}P^1} \Omega = k,$$
	which concludes the proof.     
\end{proof}

Recalling the Wirtinger inequality, we obtain the following. 

\begin{coro}\label{intersecboundbyvol}
	Let $C \subset \mathbb{C}P^{2n+1 }$ be a closed $2$-cycle. Then 
	$$\left| [C]\cdot [\mathbb{C}P^{2n}] \right| \leq |C|_{g_{FS}}.$$
\end{coro}

Finally, we provide the demonstration for Theorem $\ref{2systofhommetrics}$.
\begin{proof}[Proof of Theorem $\ref{2systofhommetrics}$]
	In view of Lemma $\ref{2systofhommetricslemma}$, it is enough to prove that we have an equality in equation $(\ref{initbounfor2sys})$. Or, equivalently to prove that if $C$ is a closed $2$-cycle satisfying $|C|_{g_t}<\min\{ 1,t \}$, then $[C]=0$ in homology. 
	
	Consider initially the case $t\leq 1$, and assume $|C|_{g_t}<t$. Given $X \in \mathfrak{X}(\mathbb{C}P^{2n+1})$, we can compare the metrics $g_t$ and $g_{FS}$ as follows: 
	$$g_t(X,X)=tg^0(X,X) + g^{1}(X,X) \geq tg^0(X,X) + tg^{1}(X,X) \geq tg_{FS}(X,X).$$
	This implies the comparison between volumes $t |C|_{g_{FS}} \leq |C|_{g_t}$. Hence, applying Corollary $\ref{intersecboundbyvol}$ we have
	$$\left|[C]\cdot[\mathbb{C}P^{2n}] \right| \leq |C|_{g_{FS}}  
	\leq \frac{1}{t} |C|_{g_t} 
	< 1.$$
	Now, since $[C] \cdot [\mathbb{C}P^{2n}]$ is an integer, it must be zero. However, we know that the homology $H_{4n}(\mathbb{C}P^{2n+1},\mathbb{Z})$ is generated by $[\mathbb{C}P^{2n}]$, and the intersection paring is non-degenerated, so we must have $[C]=0$, as claimed.
	
	For the case $t \geq 1$, a similar argument as before shows that 
	$$\left|[C]\cdot[\mathbb{C}P^2] \right| \leq |C|_{g_{FS}}  
	\leq |C|_{g_t}. $$     
	Then again, we conclude that $[C]=0$ if $|C|_{g_t}<1$.    
\end{proof}  

We now prove Theorem \ref{confclassofhommetr}, which states that every homogeneous metric maximizes the normalized systole in its conformal class. Our strategy, inspired by \cite{berger} and \cite{lucas_rafael_sistole_projcspace}, is to parametrize the linear projective planes which realize the systole in such a way that we can apply the coarea formula in order to prove the existence of integral geometric formula (Appendix \ref{Appendix:IntegralGeometricFormulasandSystolic Inequalities}), making Theorem \ref{confclassofhommetr} a direct corollary of Theorem \ref{FIGimpSCC}. More precisely, we summarize our objective in the following result. 

\begin{prop}\label{IGF2sys}
	For a fixed $t\in \mathbb{R}_{>0}$, there exists a family $\{\Sigma_{\sigma}\}_{\sigma \in B}$ of linear complex projective spaces with complex dimension $1$ in $(\mathbb{C}P^{2n+1},g_t)$, parameterized by a closed Riemannian manifold $(B,g_B)$, such that, for every function $\varphi \in C^{\infty}(\mathbb{C}P^{2n+1})$, the following formula holds: 
	$$\int_{B} \left( \int_{\Sigma_{\sigma}} \varphi\, dA_{g_t} \right) dV_{g_B}= \int_{\mathbb{C}P^{2n+1}} \varphi \,dV_{g_t}.$$
	Moreover, for each $\sigma \in B$, we have that $\mathrm{Sys}_2(\mathbb{C}P^{2n+1},g_t)=|\Sigma_{\sigma}|_{g_t}.$    
\end{prop}

As seen previously, we can explicitly find the linear projective planes that realize the $2$-systole for each $t \in \mathbb{R}_{>0}$. Therefore, we have natural candidates to comprise those families (see Theorem $\ref{2systofhommetrics}$). Inherently, we will have to analyze two cases: $0 < t \leq 1$ and $t \geq 1$. In the first case, the Penrose fibration provides a simple way to perform this construction:

\begin{proof}[Proof of Proposition $\ref{IGF2sys}$ (case $0<t\leq 1$).] Fix $0<t\leq 1$. Recall that, the Penrose fibration $\pi:(\mathbb{C}P^{2n+1},g_t) \to (\mathbb{H}P^n,g_{\mathbb{H}P^n})$ is a Riemannian submersion. Therefore, by the coarea formula, for each function $\varphi \in C^{\infty}(\mathbb{C}P^{2n+1})$, the following identity holds: 
	$$\int_{\mathbb{H}P^n} \left( \int_{\mathbb{C}P^1_{b}} \varphi\, dA_{g_t} \right) dV_{g_{\mathbb{H}P^n}}= \int_{\mathbb{C}P^{2n+1}} \varphi\, dV_{g_t}.$$ 
	Here $\mathbb{C}P^1_{b} = \pi^{-1}(b)$ for each $b \in \mathbb{H}P^n$. Finally, Theorem $\ref{2systofhommetrics}$ ensures that the fibers of the Penrose fibration realize the $2$-systole. That is, $\mathrm{Sys}_2(\mathbb{C}P^{2n+1},g_t)=|\mathbb{C}P^1_b|_{g_t}$ for each $b \in \mathbb{H}P^n$.    
\end{proof}

Let us proceed to the case $t \geq 1$. In this situation, the $2$-systole is realized by the linear projective space $\mathbb{C}P^1_T$ (see Proposition $\ref{2sysbound}$ and Theorem $\ref{2systofhommetrics}$). As before, our objective is to find a parameterized family of linear projective spaces that are isometric to $\mathbb{C}P^1_T$ and admit an integral geometric formula. In order to do so, we will apply the double fibration argument, which was already known and well-understood by M. Pu and M. Berger (see, for instance, \cite{Berger1992}).   

Following \cite{paiva_fernandes}, Definition $2.6$ and subsequently Example $3$, we have that the inclusions  $\mathrm{Sp}(n) \times \mathrm{U}(1),\mathrm{Sp}_T(1) \subset \mathrm{Sp(n+1)}$ induces the \textit{double fibration}: 
\begin{equation}\label{doublefibration}
	\begin{split}	
		\begin{xy}\xymatrix{
				& E \doteq \frac{\mathrm{Sp}(n+1)}{L} \ar[dl]_{\nu}  \ar[dr]^{\rho} & \\
				\mathbb{C}P^{2n+1}&   & N \doteq \frac{\mathrm{Sp}(n+1)}{\mathrm{Sp}_T(1)}
			}
		\end{xy} 
	\end{split}
\end{equation}  
where $L\doteq (\mathrm{Sp}(n) \times \mathrm{U}(1)) \cap \mathrm{Sp}_T(1) \cong \mathrm{U}(1)$. 

Note that the fibers of $\rho:E \to N $ are modeled by $\mathbb{C}P^1_T = \mathrm{Sp}_T(1)/\mathrm{U}(1)$. Therefore, the parameterized family $\{\nu(\rho^{-1}(\sigma))\}_{\sigma \in N}$ consists of linear complex projective planes, each one diffeomorphic to $\mathbb{C}P^1_T$ by a left translation of $\mathrm{Sp}(n+1)$. Now, the existence of an integral geometric formula will follow from this parameterized family, by applying the coarea formula twice in the double fibration $(\ref{doublefibration})$. However, first we need to introduce appropriate Riemannian metrics on the manifolds $E$ and $N$.      

\begin{prop}
	For each $t\geq 1$ there are Riemannian metrics $g_E$ and $g_N$ on the manifolds $E$ and $N$ such that: 
	\begin{enumerate}[label=\alph*),ref=(\alph*)]\label{metgegb}
		\item \label{metgegb1} $g_E$ and $g_N$ are $\mathrm{Sp}(n+1)$-invariants;
		\item \label{metgegb3}The Jacobian of the maps $\rho: (E,g_E) \to (N,g_N)$ and $\nu:(E,g_E) \to (\mathbb{C}P^{2n+1},g_t)$ are constant;
		\item \label{metgegb4}For each $\sigma \in N$, $\restr{\nu}{\rho^{-1}(\sigma)}: (\rho^{-1}(\sigma),g_E) \to (\mathbb{C}P^{2n+1},g_t)$ is an isometric embedding. Moreover, $\Sigma_{\sigma} \doteq \nu(\rho^{-1}(\sigma))$ is isometric to $\mathbb{C}P^1_T$, for all $\sigma \in N$.     
	\end{enumerate}  
\end{prop}    

\begin{proof}
	We begin by constructing the metrics $g_N$ and $g_E$. First, we define $g_N$ as any $\mathrm{Sp}(n+1)$-invariant metric. Since the Lie group $\mathrm{Sp}(n+1)$ is compact such metric exists. 
	
	In order to define the metric $g_E$ we recall that, by Example $3$ in \cite{paiva_fernandes}, we can regard $E$ as a submanifold of $\mathbb{C}P^{2n+1} \times N$. Therefore we define $g_E$ as the induced metric from the product metric $g_t \times g_N$.   
	
	Property $\ref{metgegb1}$ follows directly by construction. Property $\ref{metgegb3}$ is a simple consequence of property $\ref{metgegb1}$ together with the fact that $\mathrm{Sp}(n+1)$ acts transitively in $\mathbb{C}P^{2n+1}$, $N$ and $E$. Therefore, it remains only to prove $\ref{metgegb4}$. However, under the identification $E \subset \mathbb{C}P^{2n+1} \times N$, the projections $\nu$ and $\rho$ are given by the projections on the first and second variables. Therefore $\ref{metgegb4}$ is a simple consequence of the construction of the metric $g_E$.   
\end{proof}

Now we are in conditions to prove Proposition $\ref{IGF2sys}$, for the case $t \geq 1$.

\begin{proof}[Proposition $\ref{IGF2sys}$ (case $t\geq 1$).]
	Applying the coarea formula for $\rho$ and $\nu$ in the double fibration $\ref{doublefibration}$, and using that the Jacobians of  the map $\rho: (E,g_E) \to (N,g_N)$ and the map $\nu:(E,g_E) \to (\mathbb{C}P^{2n+1},g_t)$ are constant, we obtain
	\begin{align*}
		\frac{1}{|\mathrm{Jac}\rho|} \int_{N}\left(\int_{\rho^{-1}(\sigma)} \tilde{\varphi} dA_{g_E} \right) dV_{g_N} &= \int_{E} \tilde{\varphi} dV_{g_{E}}\\ &= \frac{1}{\left| \mathrm{Jac}\nu \right|}\int_{\mathbb{C}P^{2n+1}} \left( \int_{\nu^{-1}(p)} \tilde{\varphi} dA_{g_E} \right) dV_{g_t} ,
	\end{align*}
	for every $\tilde{\varphi} \in C^{\infty}(E)$. Since the metric $g_E$ is $\mathrm{Sp}(n+1)$-invariant, the fibers of $\nu:E \to \mathbb{C}P^{2n+1}$ have the same area. Therefore, for a given $\varphi \in C^{\infty}(\mathbb{C}P^{2n+1})$, defining $\tilde{\varphi}= \nu^{*}\varphi$, we obtain
	$$\frac{\left| \mathrm{Jac}\nu \right|}{|\mathrm{Jac}\rho|}\int_{N}\left(\int_{\rho^{-1}(\sigma)} \nu^{*} \varphi dA_{g_E} \right)dV_{g_N} = |\nu^{-1}(o)|_{g_E} \int_{\mathbb{C}P^{2n+1}} \varphi dV_{g_t}.$$ 
	On the other hand, since $\nu: (\rho^{-1}(\sigma),g_E) \to (\Sigma_{\sigma},g_t)$ is an isometry, for each $\sigma \in N$, we can rewrite the above formula as 
	\begin{align*}
		\int_{\mathbb{C}P^{2n+1}} \varphi dV_{g_t} &= \frac{\left| \mathrm{Jac}\nu \right|}{|\mathrm{Jac}\rho| |\nu^{-1}(o)|_{g_E}}\int_{N}\left(\int_{\rho^{-1}(\sigma)} \nu^{*} \varphi dA_{g_E} \right)dV_{g_N} \\
		&= \frac{\left| \mathrm{Jac}\nu \right|}{|\mathrm{Jac}\rho| |\nu^{-1}(o)|_{g_E}}\int_{N}\left(\int_{\Sigma_{\sigma}} \varphi dA_{{g_t}} \right)dV_{g_N}.
	\end{align*}      
	To conclude the proof, we define $(B,g_B)$ as $\left(N,\lambda g_N\right)$, where $\lambda$ is given by $\left({\left| \mathrm{Jac}\nu \right|}/{|\mathrm{Jac}\rho| |\nu^{-1}(o)|_{g_E}} \right)^{-\frac{2}{\mathrm{dim}(N)}}$ is constant. The fact that every element of $\{\Sigma_{\sigma}\}_{\sigma \in N}$, realizes the systole follows from Proposition $\ref{metgegb}$ $\ref{metgegb4}$ and Theorem $\ref{2systofhommetrics}$.  
\end{proof}  

\subsection{4n-Systole}\label{4nsystole}

Following what was done in the previous section, we will complete the proof of Theorems $\ref{systofhommetrics}$ and $\ref{confclassofhommetr}$, by studying the $4n$-systole case. 

As we will see, Theorem $\ref{systofhommetrics}$ is a simple consequence of a classical calibration argument based on the fact that each homogeneous metric is balanced.   

\begin{prop}\label{4sysofgt}
	Given $t \in \mathbb{R}_{>0}$, 
	$$\mathrm{Sys}_{4n}(\mathbb{C}P^{2n+1},g_t)=|\mathbb{C}P^{2n}_{\sigma}|_{g_t}=\frac{2nt+1}{(2n+1)!},$$
	where $\mathbb{C}P^{2n}_{\sigma} \doteq \{[p]: p \in \mathbb{S}^{4n+3} \text{ and } p \perp \sigma\},$ for each $\sigma \in \mathbb{C}P^{2n+1}$.  
\end{prop} 
\begin{proof}
	Fix $t \in \mathbb{R}_{>0}$ and $\sigma \in \mathbb{C}P^{2n+1}$. Since $(\mathbb{C}P^{2n+1}, J_{can}, g_t)$ is a balanced manifold, the $2$-form $\omega_t(\cdot, \cdot) = g_t(J_{can}\cdot, \cdot)$ satisfies that $\omega_t^{2n}$ is closed.
	
	Now, every homologically non-trivial, closed $4n$-cycle $C$ in $\mathbb{C}P^{2n+1}$ can be decomposed as $C = k \mathbb{C}P^{2n}_{\sigma} + \partial R$, where $R$ is a $(4n+1)$-cycle and $k$ is a non-zero integer. Therefore, by the Wirtinger inequality and the Stokes' Theorem, we have
	
	$$|C|_{g_t} \geq \frac{|k|}{(2n)!} \int_{\mathbb{C}P^{2n}_{\sigma}} {\omega_t^{2n}}+ \frac{1}{(2n)!}\int_{\partial R} {\omega_t^{2n}} = |k| \,|\mathbb{C}P^{2n}_{\sigma}|_{g_t}.$$
	Since $k$ is non-zero we obtain that $\mathrm{Sys}_{4n}(\mathbb{C}P^{2n+1},g_t)=|\mathbb{C}P^{2n}_{\sigma}|_{g_t}$. It remains to compute the volume of $\mathbb{C}P^{2n}_{\sigma}$. For that, we recall that $\Omega$, the Kähler form associated to the Fubini-Study metric, was normalized so that $\int_{\mathbb{C}P^{2n}_{\sigma}} \Omega^{2n} = 1$, and also that $\Omega^{2n} = 2n\Omega_0 \Omega_1^{2n-1} + \Omega_1^{2n}$, where $\Omega_0$ and $\Omega_1$ are defined in $(\ref{omega0omega1})$. Therefore, we have the following identities: 
	\begin{align*}
		|\mathbb{C}P^{2n}_{\sigma}|_{g_t}&=\frac{1}{(2n)!} \int_{\mathbb{C}P^{2n}_{\sigma}}\omega_t^{2n}
		\\ &= \frac{1}{(2n)!}\left( \int_{\mathbb{C}P^{2n}_{\sigma}} 2nt\Omega_0 \Omega_1^{2n-1} + \Omega_1^{2n} \right) \\
		&= \frac{1}{(2n)!}\left( \int_{\mathbb{C}P^{2n}_{\sigma}} t\Omega^{2n} + (1-t)\Omega_1^{2n} \right) \\
		&=\frac{1}{(2n)!}\left(t+(1-t)\int_{\mathbb{C}P^{2n}_{\sigma}} \Omega_1^{2n} \right).
	\end{align*}
	Therefore it is enough to compute $\int_{\mathbb{C}P^{2n}_{\sigma}} \Omega^{2n}_1$. But since $\mathrm{Sp}(n+1)$ acts transitively in $\mathbb{C}P^{2n+1}$ by $g_t$-isometries, it suffices to compute the integral for a fixed element $\sigma_0 \in \mathbb{C}P^{2n+1}$. For convenience we take $\sigma_0$ generated by $\textbf{e}_{n+2}\in \mathbb{C}^{2n+2}$.
	
	By Proposition $\ref{propW0andW1}$  we have $\Omega_1^{2n} = (2n)!\, \pi^{*}dV_{g_{\mathbb{H}P^n}}$, where $\pi:\mathbb{C}P^{2n+1}\to \mathbb{H}P^n$ is the Penrose fibration. So bearing in mind the following commutative diagram  
	\[
	\begin{xy}\xymatrix{
			& \mathbb{C}^{2n} \ar[dl]_{\Phi}  \ar[dr]^{\Psi} & \\
			\mathbb{C}P^{2n}_{\sigma_0} \subset \mathbb{C}P^{2n+1} \ar[rr]_{\pi}&   & \mathbb{H}P^n
		}
	\end{xy} 
	\] 
	where $\Phi: \mathbb{C}^{2n} \to \mathbb{C}P^{2n}_{\sigma_0}$, $(x_1,..,x_n,y_1,...,y_n) \mapsto [1:x_1:...:x_n:0:y_1:...:y_n]$, and $\Psi: \mathbb{C}^{2n} \to \mathbb{H}P^n$, $(x_1,..,x_n,y_1,...,y_n) \mapsto [1:x_1+jy_1:...:x_n+jy_n]$, are coordinates charts with dense image, we have that
	
	\begin{align*}
		\int_{\mathbb{C}P^{2n}_{\sigma_0}}\frac{\Omega_1^{2n}}{(2n)!}= \int_{\mathbb{C}P^{2n}_{\sigma_0}} \hspace{-3pt}\pi^{*}dV_{g_{\mathbb{H}P^n}}= \int_{\Phi(\mathbb{C}^{2n})} \hspace{-3pt}\pi^{*}dV_{g_{\mathbb{H}P^n}}
		= \int_{\Psi(\mathbb{C}^{2n})} \hspace{-3pt}dV_{g_{\mathbb{H}P^n}} =|\mathbb{H}P^n|_{g_{\mathbb{H}P^n}}.  
	\end{align*} 
	Finally, the area of $\mathbb{H}P^n$ can be computed by applying the coarea formula to the Riemannian submersion $\pi:(\mathbb{C}P^{2n+1},g_{FS}) \to (\mathbb{H}P^{n},g_{\mathbb{H}P^n})$. Explicitly we have that $|\mathbb{H}P^n|_{g_{\mathbb{H}P^n}}=\frac{1}{(2n+1)!}$, concluding the proof.               
\end{proof}

We conclude this section proving Theorem $\ref{confclassofhommetr}$ in the context of the $4n$-systole. Similarly to Section $\ref{2systole}$, we present a family of linear projective spaces admitting an integral geometric formula. Therefore, once more, the desired result will follow from Theorem $\ref{FIGimpSCC}$. As before, the integral geometric formula is derived through an argument using a double fibration and the coarea formula.

In light of Proposition $\ref{4sysofgt}$ the family $\{\mathbb{C}P^{2n}_{\sigma}\}_{\sigma \in \mathbb{C}P^{2n+1}}$ is a natural choice for the family of linear projective spaces, since every element of the family realizes the $4n$-systole. Moreover, in order to assist the construction of the integral geometric formula, we define the \textit{incidence set}  $\mathcal{I}=\{(p,\sigma) \in \mathbb{C}P^{2n+1} \times \mathbb{C}P^{2n+1} : p \in \mathbb{C}P^{2n}_{\sigma}\}$. It is a well-known fact that the incidence set $\mathcal{I}$ induces the double fibration 

\begin{equation}\label{doublefibrationI}
	\begin{split}
		\begin{xy}\xymatrix{
				& \mathcal{I} \ar[dl]_{\nu}  \ar[dr]^{\rho} & \\
				\mathbb{C}P^{2n+1}&   & \mathbb{C}P^{2n+1}
			}
		\end{xy}
	\end{split}
\end{equation}
where $\nu$ and $\rho$ are, respectively, the projections onto the first and second coordinates (\cite{paiva_fernandes}). For every $t \in \mathbb{R}_{>0}$, the inclusion $\mathcal{I} \subset (\mathbb{C}P^{2n+1} \times \mathbb{C}P^{2n+1}, g_t \times g_t)$, induces a Riemannian metric $\tilde{g}_t$ in the incidence set. In what follows, we underline some properties of these double fibration and its Riemannian metrics.            

\begin{prop}\label{propDFforIS}
	Let $t \in \mathbb{R}_{>0}$. The following assertions hold: 
	\begin{enumerate}[label=\alph*),ref=(\alph*)]
		\item \label{propDFforIS1} The action of $\mathrm{Sp}(n+1)$ on $\mathbb{C}P^{2n+1}\times \mathbb{C}P^{2n+1}$ induces an action by isometries on $(\mathcal{I},\tilde{g}_t)$;
		\item \label{propDFforIS3}For each $(p,\sigma) \in \mathbb{C}P^{2n+1} \times \mathbb{C}P^{2n+1}$, the maps $\restr{\nu}{\rho^{-1}(\sigma)}: (\rho^{-1}(\sigma),\tilde{g}_t) \to (\mathbb{C}P_{\sigma}^{2n},g_t)$ and $\restr{\rho}{\nu^{-1}(p)}: (\nu^{-1}(p), \tilde{g}_t) \to (\mathbb{C}P^{2n}_{p}, g_t)$ are isometries;
		\item \label{propDFforIS4}The map $\mathbb{C}P^{2n+1} \ni p \mapsto \int_{\nu^{-1}(p)} \frac{|\mathrm{Jac} \rho|}{|\mathrm{Jac}\nu |} dA_{\tilde{g}_t} \in \mathbb{R}$ is constant.     
	\end{enumerate}  
\end{prop} 
\begin{proof}
	We provide a proof for $\ref{propDFforIS4}$, which is the only part not straightforward to check. Firstly, due to the $\mathrm{Sp}(n+1)$-invariance of the metrics $g_t$ and $\tilde{g}_t$, the Jacobians $|\mathrm{Jac}\nu|, |\mathrm{Jac}\rho|: \mathcal{I} \to \mathbb{R}$ are also $\mathrm{Sp}(n+1)$-invariant. Now, fix $p = U \cdot o \in \mathbb{C}P^{2n+1}$, for $U \in \mathrm{Sp}(n+1)$. Using that $\restr{\rho}{\nu^{-1}(p)}: (\nu^{-1}(p), \tilde{g}_t) \to (\mathbb{C}P^{2n}_{p}, g_t)$ is an isometry, we obtain 
	\begin{align*}
		\int_{\nu^{-1}(p)}{\frac{|\mathrm{Jac} \rho|}{|\mathrm{Jac}\nu |} dA_{\tilde{g}_t}} 
		&= \int_{\mathbb{C}P^{2n}_{p}}{\frac{|\mathrm{Jac} \rho|}{|\mathrm{Jac}\nu |}(p,\sigma) dA_{g_t}(\sigma)} \\
		&= \int_{\mathbb{C}P^{2n}_o}{\frac{|\mathrm{Jac} \rho|}{|\mathrm{Jac}\nu |}(p, U \cdot \eta ) dA_{g_t}(\eta)} \\
		&= \int_{\mathbb{C}P^{2n}_o}{\frac{|\mathrm{Jac} \rho|}{|\mathrm{Jac}\nu |}(o, \eta ) dA_{g_t}(\eta)} = \int_{\nu^{-1}(o)}{\frac{|\mathrm{Jac} \rho|}{|\mathrm{Jac}\nu |} dA_{\tilde{g}_t}},
	\end{align*}  
	as desired.  
	
\end{proof}

A straightforward application of the coarea in the double fibration $\eqref{doublefibrationI}$ allows us to prove the existence of an integral geometric formula for the family $\{\mathbb{C}P^{2n}_{\sigma}\}_{\sigma \in \mathbb{C}P^{2n+1}}$. 

\begin{prop}
	For each $t\in \mathbb{R}_{>0}$, there exists a Riemannian metric $\hat{g}_t$ homothetic to $g_t$ such that, for each $\varphi \in C^{\infty}(\mathbb{C}P^{2n+1})$, the following formula holds:
	\begin{equation}\label{IGFfocd1}
		\int_{\mathbb{C}P^{2n+1}}{\left( \int_{\mathbb{C}P^{2n}_{\sigma}}{ \varphi\, dA_{g_t}}\right) dV_{\hat{g}_t}(\sigma)}= \int_{\mathbb{C}P^{2n+1}} \varphi\, dV_{g_t}.
	\end{equation}  
\end{prop} 
\begin{proof}
	Applying the coarea formula twice in the double fibration $\eqref{doublefibrationI}$ and Proposition $\ref{propDFforIS}$ $\ref{propDFforIS3}$, it holds that:
	$$\int_{\mathbb{C}P^{2n+1}}{\left( \int_{\mathbb{C}P^{2n}_{\sigma}}{ \varphi\, dA_{g_t}}\right) dV_{{g}_t}(\sigma)} = \int_{\mathbb{C}P^{2n+1}} \left( \int_{\mathbb{C}P^{2n}_p}{\frac{|\mathrm{Jac} \rho|}{|\mathrm{Jac}\nu |}}\, dA_{\tilde{g}_t} \right)\varphi \, dV_{g_t}(p),$$
	for every $\varphi \in C^{\infty}(\mathbb{C}P^{2n+1})$. Moreover, Proposition $\ref{propDFforIS}$ $\ref{propDFforIS4}$ establishes that the function $\mathbb{C}P^{2n+1} \ni p \mapsto \int_{\nu^{-1}(p)} \frac{|\mathrm{Jac} \rho|}{|\mathrm{Jac}\nu |} dA_{\tilde{g}_t} \in \mathbb{R}$ is constant. As a result, calling this constant $\theta = \theta(t)$ and defining $\hat{g}_t \doteq (\theta)^{\frac{1}{2n+1}}g_t$ we obtain the desired result.      
\end{proof} 

\section{Systole of Balanced Metrics}\label{section:BalancedMetrics}

\indent In Chapter \ref{section:homogeneousmetrics} we proved that the Fubini-Study metric is the global minimum, among homogeneous metrics, of the normalized $(2n-2)$-systole functional on $\mathbb{C}P^n$, $n \geq 3$. A crucial step of the proof was to determine the submanifold that realizes the systole for each homogeneous metric, which was possible due to the fact that each of these metrics is balanced. Therefore, a natural question is if the Fubini-Study metric remains a point of minimum for the normalized $(2n-2)$-systole functional among all the balanced metrics, that are balanced with respect to the canonical complex structure of complex projective space. This section is devoted to study this problem. More precisely, we will prove Theorem $\ref{sysofbal}$, which is stated below after introducing notation. 

Let $\mathscr{B}$ denote the space of smooth balanced metrics with respect to the canonical complex structure of $\mathbb{C}P^n$. We endow this space with the ${C}^{2}$-topology. We denote by $\mathscr{K} \subset \mathscr{Bal}$ the subspace of smooth Kähler metrics.   

\begin{teo}{\label{sysofbal}}
	Let $n \geq 3$. There exists an open set $ \mathscr{K} \subset \mathcal{U} \subset \mathscr{B}$, in the ${C}^{2}$-topology, such that for every metric $g\in \mathcal{U}$,
	$$\mathrm{Sys}_{2n-2}^{\mathrm{nor}}(\mathbb{C}P^n,g) \geq \mathrm{Sys}_{2n-2}^{\mathrm{nor}}(\mathbb{C}P^n,g_{FS}).$$
	Moreover, $g \in \mathcal{U}$ satisfies the equality if and only if $g \in \mathscr{K}$.       
\end{teo}

The proof of this theorem relies on an analysis of the Taylor expansion of the functional $\mathrm{Sys}^{\mathrm{nor}}_{2n-2}: \mathscr{B} \to \mathbb{R}$ over the set of Kähler metrics. In order to formalize this argument, we must first endow the spaces of Kähler and balanced metrics with structures of smooth Banach manifolds, in such a way that the inclusion is an embedding in a neighborhood of each smooth metric. The next section is devoted to define these structures.

\subsection{Manifold Structure of the space of Balanced Metrics}\label{manfstrctoBalMet}
In this section, we fix $n \geq 3$ and set $J$ to be the canonical complex structure of $\mathbb{C}P^n$. Accordingly, the Hermitian condition will always be defined with respect to the canonical complex structure.

In order to endow the space of balanced metrics with a structure of Banach manifold, rather than a structure of Fréchet manifold, we will have to be less restrictive and work in the space of ${C}^{1,\nu}$ Riemannian metrics, for some $0<\nu<1$ fixed. We choose to work in the Hölder topology instead of directly work in the $C^2$-topology to facilitate the use of regularity theorems.   

Let $(\mathrm{Riem}^{1,\nu}(\mathbb{C}P^n),{C}^{1,\nu})$ denote the space of ${C}^{1,\nu}$ Riemannian metrics endowed with the ${C}^{1,\nu}$-topology. For clarity, we will denote by $\mathscr{K}^{1,\nu}$, $\mathscr{B}^{1,\nu}$ and $\mathscr{H}^{1,\nu}$ the spaces of Kähler, balanced and Hermitian metrics with regularity $C^{1,\nu}$, respectively, equipped with the subset topology induced by inclusion in $(\mathrm{Riem}^{1,\nu}(\mathbb{C}P^n),{C}^{1,\nu})$.       

Recall that we have a duality between the space of Hermitian metrics $\mathscr{H}^{1,\nu}$ and the space of differential forms. Indeed, endowing the space of ${C}^{1,\nu}$ complex valued differential forms $(C^{1,\nu}(\Lambda_{\mathbb{C}}^{\bullet}),{C}^{1,\nu})$ with the ${C}^{1,\nu}$-topology, we have the following homeomorphism: 
\begin{align}\label{formxmetrics}
	\begin{split}
		\mathcal{J} : C^{1,\nu}(\Lambda_+^{1,1}) &\to \mathscr{H}^{1,\nu}  \\
		\omega &\mapsto g_{\omega}(\cdot,\cdot)\doteq \omega(\cdot,J\cdot),
	\end{split}     
\end{align}
where 
$$\Lambda_+^{p,p} = \{\alpha \in \Lambda_{\mathbb{R}}^{p,p} : \alpha(v_1,...,v_p,Jv_1,...,Jv_p)>0 \mbox{, for every $\{v_j,Jv_j\}_{j=1}^{p}$ l.i. set} \},$$ 
denotes the open cone of positive $(p,p)$-forms inside $\Lambda_{\mathbb{R}}^{p,p}$, the bundle of real $(p,p)$-forms.  

Thus, in order to define the manifold structure for the set of balanced metrics it is enough to define a Banach manifold structure in the space $\mathcal{B}$ of  \textit{balanced forms} (of class $C^{1,\nu}$): 
$$\mathcal{B}\doteq \mathcal{J}^{-1}\left( \mathscr{B}^{1,\nu} \right) = \{\omega \in C^{1,\nu}(\Lambda_+^{1,1}) : d\omega^{n-1}=0\}.$$

\begin{prop}\label{balismanif}
	The space of balanced forms $\mathcal{B}$ has a structure of smooth Banach manifold modelled over $C^{1, \nu}_{{cl}}\big(\Lambda^{n-1,n-1}_{\mathbb{R}}\big)$, the Banach space of real closed $(n-1,n-1)$-forms.
\end{prop}
\begin{obs}
	Note that the closed property of differential forms is a closed condition in the ${C}^{1,\nu}$-topology. Therefore, the space $C^{1,\nu}_{{cl}}(\Lambda^{p,p}_{\mathbb{R}}) $ of real and closed $(p,p)$-forms is a closed subspace of $C^{1,\nu}(\Lambda^{p,p}_{\mathbb{R}})$, consequently, a Banach vector space.
\end{obs}
\begin{proof}
	Regarding $C^{1,\nu}(\Lambda_+^{1,1})$ and $C^{1,\nu}(\Lambda_+^{n-1,n-1})$ as open sets of Banach vector spaces, it is easily seen that the following map is smooth
	\begin{align*}
		\Phi: C^{1,\nu}\big(\Lambda_+^{1,1}\big) &\to C^{1,\nu}\big(\Lambda_+^{n-1,n-1}\big) \\
		\omega &\mapsto \omega^{n-1}.
	\end{align*}
	This map is also known to be bijective, see \cite{Michelsohn_1982}. Even more, for each $\omega \in C^{1,\nu}\big(\Lambda_+^{1,1}\big)$ the map $\restr{d\Phi}{\omega}: C^{1,\nu}(\Lambda_{\mathbb{R}}^{1,1}) \to C^{1,\nu}(\Lambda_{\mathbb{R}}^{n-1,n-1})$, $ \alpha \mapsto (n-1)\alpha \wedge \omega^{n-2}$, is continuous. On the other hand, Theorem $\ref{lefdecompthrm}$ $\ref{lefdecompthrm1}$ implies that this map is also bijective. Hence, it is a Banach space isomorphism. Therefore, by the inverse function theorem for Banach spaces, the map $\Phi$ is a smooth diffeomorphism. In particular, denoting by $C_{cl}^{1,\nu}(\Lambda_{+}^{n-1,n-1})$ the space of positive, closed $(n-1,n-1)$-forms, we have that $\Phi: \mathcal{B} \to C_{cl}^{1,\nu}(\Lambda_{+}^{n-1,n-1})$ is a homeomorphism. Since $C_{cl}^{1,\nu}\big(\Lambda_{+}^{n-1,n-1}\big) \subset C_{cl}^{1,\nu}\big(\Lambda_{\mathbb{R}}^{n-1,n-1}\big)$ is an open set of a Banach vector space, the map $\restr{\Phi}{\mathcal{B}}: \mathcal{B} \to C_{cl}^{1,\nu}\big(\Lambda_{+}^{n-1,n-1}\big)$ defines a global chart. Then, the space of balanced forms has a structure of smooth Banach manifold modelled over $C_{cl}^{1,\nu}\big(\Lambda_{\mathbb{R}}^{n-1,n-1}\big)$.            
\end{proof}

\begin{coro}\label{globalchrtbalmetr}
	The space of balanced metrics $\mathscr{B}^{1,\nu}$ has a structure of smooth Banach manifold such that the map 
	\begin{align*}
		\hat{\Phi}: \mathscr{B}^{1,\nu} &\to C_{cl}^{1,\nu}\big(\Lambda_{+}^{n-1,n-1}\big) \\
		g&\mapsto \Phi\left(g(J\cdot,\cdot)\right) 
	\end{align*}
	defines a smooth diffeomorphism onto the set $C_{cl}^{1,\nu}\big(\Lambda_{+}^{n-1,n-1}\big)$. 
\end{coro}
Corollary $\ref{globalchrtbalmetr}$ establishes the manifold structure of the space of balanced metrics. Therefore, it remains to prove that the space of Kähler metrics has a structure of Banach manifold with the property that the inclusion $\iota: \mathscr{K}^{1,\nu} \hookrightarrow \mathscr{B}^{1,\nu}$ is a smooth embedding around every smooth metric.

Since the \textit{space of Kähler forms (of class $C^{1,\nu}$)} $\mathcal{K} \doteq \mathcal{J}^{-1}\left(\mathscr{K}^{1,\nu} \right) = C_{cl}^{1,\nu}\big(\Lambda^{1,1}_+\big)$ is an open set of the Banach space $C_{cl}^{1,\nu}\big(\Lambda^{1,1}_{\mathbb{R}}\big)$, it has a natural smooth Banach manifold structure, in such a way that the inclusion $\iota: \mathcal{K} \hookrightarrow \mathcal{B}$ is a topological embedding. The aforementioned smooth embedding property can be stated as the following proposition. The remaining portion of this section will be dedicated to proving it.

\begin{prop}\label{KissplitinB}
	Let $j\doteq {\Phi}\circ{\iota}: \mathcal{K} \to C_{cl}^{1,\nu}\big(\Lambda^{n-1,n-1}_{+}\big)$. For each smooth Kähler form $\omega_0 \in \mathcal{K}$, there exists a closed subspace $A_{\omega_0} \subset T_{\omega_0} \mathcal{B}$, open neighborhoods $U \subset \mathcal{K}$ of $\omega_0$ and $V\subset A_{\omega_0}$ of $0$, and an open set $W$ containing $j(\omega_0)$ in $C_{cl}^{1,\nu}\big(\Lambda_+^{n-1,n-1}\big)$, along with a smooth diffeomorphism $\rho:U\times V \to W$, satisfying the following properties:
	\begin{enumerate}[label=\alph*),ref=(\alph*)]
		\item $T_{\omega_0}\mathcal{B}=T_{\omega_{0}}\mathcal{K} \oplus A_{\omega_0}$;
		\item $\rho(\omega_0, 0) = j(\omega_0)$; 
		\item $\rho\left(U \times \{0\} \right) = W \cap j\left(U \right)$;
		\item For every $(\omega, \xi) \in U \times V$ and $\eta \in A_{\omega_0}$, we have that $\restr {d\rho }{(\omega,\xi)} \, \eta = \restr{d \Phi}{\omega_0} \, \eta$.    
	\end{enumerate}   
\end{prop} 

The non-trivial aspect of Proposition \ref{KissplitinB} lies in finding the appropriate complement of the tangent space of $\mathcal{K}$. To accomplish this, we begin by presenting a characterization of these tangent spaces.

\begin{lema}\label{TgspaceofKB}
	Let $\mathcal{K}$ and $\mathcal{B}$ denote, respectively, the Banach manifolds of Kähler forms and balanced forms, endowed with the ${C}^{1,\nu}$-topology. Then:
	\begin{enumerate}[label=\alph*),ref=(\alph*)]
		\item \label{TgspaceofKB1}For each $\omega \in \mathcal{K}$, we have $T_{\omega}\mathcal{K}=C_{cl}^{1,\nu}\big(\Lambda^{1,1}_{\mathbb{R}}\big)$;
		\item \label{TgspaceofKB2}For each $\omega \in \mathcal{B}$, we have $T_{\omega}\mathcal{B}=\left\{\eta \in C^{1,\nu}\big(\Lambda^{1,1}_{\mathbb{R}}\big) : d(\eta \wedge \omega^{n-2})=0 \right\}$;
		\item \label{TgspaceofKB3}For each $\omega \in \mathcal{K}$, the map $d\iota_{\omega} : T_{\omega}\mathcal{K} \to T_{\omega}\mathcal{B}$ is given by the canonical inclusion.  
	\end{enumerate}
\end{lema}
\begin{proof}
	The prove of $\ref{TgspaceofKB1}$ follows immediately from the fact that $\mathcal{K}$ is an open set of $C_{cl}^{1,\nu}\big(\Lambda^{1,1}_{\mathbb{R}}\big)$. 
	
	To prove $\ref{TgspaceofKB2}$, fix $\omega \in \mathcal{B}$ and let $\mathcal{V}_{\omega}=\left\{\eta \in C^{1,\nu}\big(\Lambda^{1,1}_{\mathbb{R}}\big) : d(\eta \wedge \omega^{n-2})=0 \right\}$. The desired isomorphism is explicit given by 
	\begin{align*}
		T_{\omega} : \mathcal{V}_{\omega} &\to T_{\omega} \mathcal{B} \\ 
		\eta &\mapsto [\hat{\eta}],	
	\end{align*}    
	where $\hat{\eta}$ is the only curve in $\mathcal{B}$ defined by $\hat{\eta}(t)^{n-1}=\omega^{n-1}+ t (n-1)\eta \wedge \omega^{n-2}$, for $|t|$ sufficiently small. Finally, $\ref{TgspaceofKB3}$ follows by  $\ref{TgspaceofKB1}$ and $\ref{TgspaceofKB2}$.  
\end{proof}

In \cite{Morrey1956}, B. Morrey and J. Eells generalized the Hodge decomposition theorem for forms with distinct types of regularity. In particular, since the space of harmonic two-forms in $\mathbb{C}P^n$ is one dimensional they proved that for any smooth Kähler metric $g_\omega \in \mathscr{K}$, the space $C^{1,\nu}\left(\Lambda^2_{\mathbb{R}}\right)$ can be decomposed as follows: 
$$C^{1,\nu}\left(\Lambda^2_{\mathbb{R}}\right)= \mathbb{R}\omega \oplus \mathrm{Im}d \oplus \mathrm{Im} \delta_{\omega},$$
where the exterior derivative has domain $C^{2,\nu}\left(\Lambda^1_{\mathbb{R}}\right)$, and $\delta_{\omega}$ is the codifferential induced by $g_\omega$, with domain $C^{2,\nu}\left(\Lambda^3_{\mathbb{R}}\right)$.        

On the other hand, by Lemma $\ref{TgspaceofKB}$ $\ref{TgspaceofKB1}$ we have that $T_{\omega} \mathcal{K}=\left( \mathbb{R} \omega \oplus \mathrm{Im}d\right)\cap C^{1,\nu}\big(\Lambda^{1,1}_{\mathbb{R}}\big)$. Therefore, the aforementioned Hodge decomposition Theorem implies the splitting $T_{\omega}\mathcal{B}=T_{\omega}\mathcal{K} \oplus \left( \mathrm{Im} \delta_{\omega} \cap T_{\omega}\mathcal{B}\right)$, under the assumption that the projection $\pi_{\delta_{\omega}}: C^{1,\nu}\left(\Lambda^2_{\mathbb{R}}\right) \to \mathrm{Im}\delta_{\omega} $ preserves the subspace $T_{\omega}\mathcal{B}$. In the next result, we prove that this assumption is satisfied, thus proving the first part of Proposition \ref{KissplitinB}.

\begin{lema}\label{coexactpartoftoB}
	Let $\omega \in \mathcal{K}$ be a smooth Kähler form and $\eta \in T_{\omega}\mathcal{B}$. Then, if $\pi_{\delta_{\omega}}: C^{1,\nu}\left(\Lambda^2_{\mathbb{R}}\right) \to \mathrm{Im}\delta_{\omega}$ denotes the projection into the space of coexact forms, induced by the Hodge decomposition, we have that
	\begin{enumerate}[label=\alph*),ref=(\alph*)]
		\item\label{coexactpartoftoB1} $ \pi_{\delta_{\omega}}(\eta) \wedge \omega^{n-1} =0$;
		\item\label{coexactpartoftoB2} $\pi_{\delta_{\omega}}(\eta) \in T_{\omega} \mathcal{B}$.
	\end{enumerate}  
	In particular, $T_{\omega}\mathcal{B}=T_{\omega}\mathcal{K} \oplus A_{\omega}$ for $A_{\omega}\doteq\mathrm{Im}(\delta_{\omega})\cap T_{\omega}\mathcal{B}$.      
\end{lema} 
\begin{proof}
	Let $\omega$ and $\eta$ be as in the statement. Consider also $\eta = a\omega + d\alpha + \delta_{\omega} \theta$ the Hodge decomposition of $\eta$, where $a\in \mathbb{R}$, $\alpha \in C^{2,\nu}(\Lambda^{1}_{\mathbb{R}})$, and $ \theta \in C^{2,\nu}(\Lambda^{3}_{\mathbb{R}})$.
	
	First, we prove $\ref{coexactpartoftoB1}$. According to the Lefschetz decomposition Theorem (see Theorem $\ref{lefdecompthrm}$), it is enough to establish that $\Lambda_{\omega}(\delta_{\omega} \theta) = 0$, where $\Lambda_{\omega}$ denotes the dual of the Lefschetz operator associated with the Kähler structure $\omega$ (see Definition \ref{duallefop}). Nevertheless, since $\eta \in T_{\omega}\mathcal{B}$, we observe that $d\delta_{\omega} \theta \wedge \omega^{n-2} = 0$. Consequently, invoking again the Lefschetz decomposition Theorem, we see that this condition is equivalent to $\Lambda_{\omega}(d\delta_{\omega} \theta) = 0$. Moreover, we can commute the operators $d$ and $\Lambda_{\omega}$ by means of Proposition \ref{propofdeltac}, leading to
	$$0=(\Lambda_{\omega}d)(\delta_{\omega} \theta)=(d\Lambda_{\omega}-\delta_{\omega}^{c})(\delta_{\omega}\theta)
	= d\Lambda_{\omega} \delta_{\omega}\theta + \delta_{\omega} \delta_{\omega}^c\theta,$$
	where the operator $\delta_{\omega}^c$ is given by Definition \ref{deltacdef}, and we have applied the identity $\delta_{\omega} \delta_{\omega}^c=-\delta^c_{\omega} \delta_{\omega}$. Since $\mathrm{Im}(d) \perp_{L^2} \mathrm{Im}(\delta_{\omega})$, we further obtain
	\begin{equation}\label{deltadeltac}
		d\Lambda_{\omega} \delta_{\omega}\theta =0= \delta_{\omega} \delta_{\omega}^c\theta. 
	\end{equation}  
	To conclude that the constant function $\Lambda_{\omega}(\delta_{\omega}\theta)$ is zero, it suffices to show that it has zero mean. But, indeed
	$$\int_{\mathbb{C}P^{n}}\Lambda_{\omega}{\delta_{\omega}\theta} \,dV_{g_{\omega}}= \int_{\mathbb{C}P^{n}} \Lambda_{\omega}{\delta_{\omega}\theta} \wedge \star_{g_{\omega}} 1
	= \int_{\mathbb{C}P^{n}} \theta \wedge\left(  \star_{g_{\omega}} d \omega\right) = 0, $$
	where $\star_{g_{\omega}}$ denote the Hodge star associated with the metric $g_{\omega}$.
	
	Now, let us proceed to the proof of $\ref{coexactpartoftoB2}$. To demonstrate that $\delta_{\omega}\theta \in T_{\omega}\mathcal{B}$, we need to prove that $d(\delta_{\omega} \theta \wedge \omega^{n-2})=0$ and $\delta_{\omega}\theta \in C^{1,\nu}\big( \Lambda^{1,1}_{\mathbb{R}}\big)$. However, recalling the Hodge decomposition of $\eta$ and using the fact that $\omega$ is a closed form, we obtain  
	$$ 0 = d(\eta \wedge \omega^{n-2}) = d\left((a\omega+d\alpha)\wedge\omega^{n-2}\right) + d\left(\delta_{\omega}\theta \wedge \omega^{n-2} \right) = d\left(\delta_{\omega}\theta \wedge \omega^{n-2} \right).$$ 
	Therefore, it only remains to show that $\delta_{\omega}\theta$ is of type $(1,1)$. Denoting the projection into the space of $(p,q)$-forms by $[\cdot]_{p,q}: \Lambda^{\bullet}_{\mathbb{C}} \to \Lambda^{p,q}$, we observe that $[d\alpha + \delta_{\omega}\theta]_{2,0} = [\eta-a\omega]_{2,0} = 0$. Since, $d = \partial + \bar{\partial}$ and $\alpha= [\alpha]_{1,0}+[\alpha]_{0,1}$ we reach the following equality
	\begin{equation}\label{deldeltatheta}
		\partial [\alpha]_{1,0}= - [\delta_{\omega} \theta]_{2,0}.
	\end{equation}
	On the other hand, $\partial^* = \frac{1}{2} (\delta_\omega - i \delta^c_{\omega})$, once that $\delta_{\omega} = \partial^* + \bar{\partial}^*$ and $\delta^c_{\omega} = i(\partial^* - \bar{\partial}^*)$, where $\partial^*$ and $\bar{\partial}^*$ denote the $L^2$-dual operators of $\partial$ and $\bar{\partial}$, respectively. Hence, by $(\ref{deltadeltac})$ we see that $\partial^* (\delta_{\omega}\theta) = 0$. Decomposing the form $\delta_{\omega}\theta$, we further obtain  
	\begin{equation*}
		0 = \partial^* (\delta_{\omega}\theta) = \partial^*([\delta_{\omega}\theta]_{2,0}) + \partial^*([\delta_{\omega}\theta]_{1,1}) + \partial^*([\delta_{\omega}\theta]_{0,2}).
	\end{equation*}
	Keeping in mind that $\partial^*\left( C^{1,\nu}(\Lambda^{p,q})\right) \subset C^{0,\nu}(\Lambda^{p-1,q})$, the above equality translates to
	\begin{equation}\label{deldualdeltatheta}
		\partial^*[\delta_{\omega}\theta]_{2,0} = 0.
	\end{equation}
	
	Since the Hodge Laplacian in a Kähler manifold can be written as $\frac{1}{2}\Delta = \partial \partial^* + \partial^*\partial$ (Proposition $3.1.12$, \cite{huybrechts2005complex}), by $(\ref{deldeltatheta})$ and $(\ref{deldualdeltatheta})$ the form $[\delta_{\omega} \theta]_{2,0}$ is harmonic in $\mathbb{C}P^n$. However, since every harmonic form in $\mathbb{C}P^n$ is of type $(1,1)$, the form $[\delta_{\omega} \theta]_{2,0}$ must be zero. Additionally, $[\delta_{\omega} \theta]_{0,2} = \overline{[\delta_{\omega} \theta]}_{2,0} = 0$, completing the argument.       
\end{proof}

The previous Lemma establishes the property that, over smooth forms, the tangent space of the Kähler forms is complemented in the tangent space of balanced forms. As a consequence, the proof of Proposition \ref{KissplitinB}, that we provide bellow, reduces to a simple application of the inverse function theorem for Banach spaces.

\begin{proof}[Proof of Proposition $\ref{KissplitinB}$]
	Fix $\omega_0 \in \mathcal{K}$ a smooth Kähler form. And consider the global chart of the space of balanced metrics $\Phi: \mathcal{B} \to C_{cl}^{1,\nu}(\Lambda_{+}^{n-1,n-1})$, defined in Proposition \ref{balismanif}, also let $j = \Phi \circ \iota: \mathcal{K} \to C_{cl}^{1,\nu}(\Lambda_{+}^{n-1,n-1})$ be its restriction to the space of Kähler forms, and finally let $A_{\omega_0}$ be the complement of $T_{\omega_0}\mathcal{K}$ as defined in Lemma \ref{coexactpartoftoB}. 
	
	Since $C_{cl}^{1,\nu}(\Lambda_{+}^{n-1,n-1})$ is an open set of $C_{cl}^{1,\nu}(\Lambda_{\mathbb{R}}^{n-1,n-1})$ we can define the following smooth map
	\begin{align*}
		\rho: \mathcal{K} \times A_{\omega_0} &\to C_{cl}^{1,\nu}\big(\Lambda_{\mathbb{R}}^{n-1,n-1}\big) \\
		(\omega, \eta)  &\mapsto j(\omega) + d\Phi_{\omega_0}(\eta),
	\end{align*}  
	whose derivative at the point $(\omega_0,0) \in \mathcal{K} \times A_{\omega_0}$ is given by 
	\begin{align}\label{deriofrho}
		\begin{split}
			\restr{d\rho}{(\omega_0,0)}: T_{\omega_0}\mathcal{K} \oplus A_{\omega_0} &\to C_{cl}^{1,\nu}\big(\Lambda_{\mathbb{R}}^{n-1,n-1}\big) \\
			(\alpha, \eta)  &\mapsto \restr{d\Phi}{\omega_0}\left(\restr{d\iota}{\omega_0} \alpha + \eta\right).
		\end{split}
	\end{align}
	Therefore, combining the decomposition $T_{\omega} \mathcal{B} = T_{\omega}\mathcal{K} \oplus A_{\omega}$ with Proposition $\ref{balismanif}$ we conclude that $\restr{d\rho}{(\omega_0,0)}$ is a Banach space isomorphism. By the inverse function theorem for Banach spaces, there exist open neighborhoods $U \subset \mathcal{K}$ of $\omega_0$ and $V \subset A_{\omega_0}$ of $0$, such that $W \doteq \rho(U \times V)$ is an open set and the map $\rho: U \times V \to W$ is a smooth diffeomorphism. The listed properties of this diffeomorphism follows directly from its explicit definition.    
\end{proof}  

\subsection{First and Second variation of the normalized Systole}\label{firstandsecvariation}

As mentioned earlier in this section, in order to establish Theorem \ref{sysofbal}, we must study the Taylor expansion of the normalized systole function. To proceed with this analysis, we require the formulas for the first and second derivatives of this map. 

Before we carry on with these computations, it is necessary to establish and fix some notations. We begin by noticing that our definition of systole naturally extends to metrics of lower regularity. More specifically, if $g$ is a metric in $\mathrm{Riem}(\mathbb{C}P^n)^{1,\nu}$, we set
$$\mathrm{Sys}_k(M,g)=\inf\{\mathrm{vol}_g(C)  :  \mbox{where $[C] \neq 0$ in ${H}_k(M,\mathbb{Z})$}\},$$
where the volume of a cycle is computed with respect to the Hausdorff measure induced by the distance function of the $C^{1,\nu}$ Riemannian manifold $(\mathbb{C}P^n,g)$.

With a consistent definition of the normalized systole $\mathrm{Sys}^{\mathrm{nor}}_{2n-2}: \mathscr{B}^{1,\nu} \to \mathbb{R}$ in the space of $C^{1,\nu}$ balanced metrics, we can employ the balanced condition to establish its smoothness in the Fréchet sense.

\begin{lema}
	Let $g_{\omega} \in \mathscr{B}^{1,\nu}$ be a balanced metric, then
	\begin{equation}\label{normsystexpl}
		\mathrm{Sys}^{\mathrm{nor}}_{2n-2}(\mathbb{C}P^n,g_{\omega})=\frac{(n!)^{\frac{n-1}{n}}}{(n-1)!} \frac{\int_{\mathbb{C}P^{n-1}} \omega^{n-1}}{\left( \int_{\mathbb{C}P^{n}} \omega^{n}\right)^{\frac{n-1}{n}}}.
	\end{equation}
	In particular, $\mathrm{Sys}^{\mathrm{nor}}_{2n-2}: \mathscr{B}^{1,\nu} \to \mathbb{R}$ is a smooth map in the Fréchet sense.
\end{lema}
\begin{proof}
	The formula $(\ref{normsystexpl})$ follow from a similar argument as the done in Proposition \ref{4sysofgt}. The smoothness is direct consequence of the given formula.   
\end{proof}
For our purposes, the most suitable way to approach the calculations of the first and second derivatives, and further on, the Taylor expansion of the normalized systole, is by doing it in charts. To achieve this, we rewrite the map $\mathrm{Sys}^{\mathrm{nor}}_{2n-2}: \mathscr{B}^{1,\nu} \to \mathbb{R}$, modulo constants, in terms of the global chart $\hat{\Phi}: \mathscr{B}^{1,\nu} \to C_{cl}^{1,\nu}(\Lambda_+^{n-1,n-1})$ (see Corollary $\ref{globalchrtbalmetr}$), leading to the following definition:
\begin{align}\label{fmap}
	\begin{split}
		\mathcal{F}: C_{cl}^{1,\nu}&(\Lambda_+^{n-1,n-1}) \to \mathbb{R} \\
		\sigma &\mapsto \frac{\int_{\mathbb{C}P^{n-1}} \sigma}{ \left( \int_{\mathbb{C}P^{n}} \sigma \wedge \Psi(\sigma)\right)^{\frac{n-1}{n}}}, 
	\end{split}
\end{align}     
where $\Psi \doteq \Phi^{-1}: C_{cl}^{1,\nu}(\Lambda_+^{n-1,n-1}) \to \mathcal{B}$. Below, we will elucidate basic properties of the functional $\mathcal{F}$.
\begin{prop}\label{basicproproff}
	The functional $\mathcal{F}: C_{cl}^{1,\nu}(\Lambda_+^{n-1,n-1}) \to \mathbb{R}$ satisfies the following properties:
	\begin{enumerate}[label=\alph*),ref=(\alph*)]
		\item\label{basicproproff1} $\mathcal{F}$ is invariant under homothety;
		\item\label{basicproproff2} $\mathcal{F}$ is constant over the Kähler forms, i.e., within the set $\Phi(\mathcal{K})$.
	\end{enumerate}
\end{prop} 
\begin{proof}
	The prove of $\ref{basicproproff1}$ follow from the homothety invariance of the normalized systole together with the fact that $\Phi(\lambda \omega)= \lambda^{n-1}\Phi(\omega)$, for every $\lambda>0$ and $\omega \in \mathcal{B}$.
	
	In order to prove $\ref{basicproproff2}$, fix $\omega \in \mathcal{K}$. The Hodge decomposition Theorem implies that $\omega = a \Omega + d\beta$. Here $\Omega$ denotes the fundamental form of the Fubini-Study metric, as always. Therefore by Stoke's Theorem
	$$\mathcal{F}\left(\Phi(\omega)\right)= \frac{\int_{\mathbb{C}P^{n-1}} \omega^{n-1}}{\left( \int_{\mathbb{C}P^{n}} \omega^{n}\right)^{\frac{n-1}{n}}} =   \frac{a^{n-1}\int_{\mathbb{C}P^{n-1}} \Omega^{n-1}}{\left( a^n\int_{\mathbb{C}P^{n}} \Omega^{n}\right)^{\frac{n-1}{n}}}=\mathcal{F}\left(\Phi(\Omega)\right).$$
\end{proof}

The last piece of notation that we will introduce is the space of \textit{normalized balanced forms} 
\begin{equation*}
	\mathcal{B}_1 \doteq \left\{\omega \in \mathcal{B} \; : \; \int_{\mathbb{C}P^n} \omega^n = 1 \right\}.
\end{equation*}
Given the invariance of $\mathcal{F}$ under homothety, considering normalized balanced forms imposes no restriction and greatly simplifies the computations. Moreover, recall that we have normalized the Fubini-Study form $\Omega$ to ensure its inclusion within this space.

Once we settle the notation, we follow through with the computations of the first and second derivatives of the functional $\mathcal{F}: C_{cl}^{1,\nu}(\Lambda_+^{n-1,n-1}) \to \mathbb{R}$.

\begin{teo}[First Variational Formula of $\mathcal{F}$]\label{fvfforF}
	If $\omega$ is a normalized balanced form and $\mu \in C_{cl}^{1,\nu}(\Lambda_{\mathbb{R}}^{n-1,n-1})$, then 
	\begin{equation*}
		\restr{d \mathcal{F}}{\Phi(\omega)} \, \mu = \left( \int_{\mathbb{C}P^{n-1}}{\mu} \right) - \left(\int_{\mathbb{C}P^{n-1}} \omega^{n-1} \right)\left( \int_{\mathbb{C}P^n} \mu \wedge \omega \right), 
	\end{equation*}
	where $\Phi: \mathcal{B} \to C_{cl}^{1,\nu}(\Lambda_+^{n-1,n-1})$, is given by $\Phi(\omega)=\omega^{n-1}$. .  
\end{teo}
\begin{proof}
	We start by defining the smooth curve $t \mapsto \mu_t = \omega^{n-1} + t \mu$ in $C^{1,\nu}(\Lambda_+^{n-1,n-1})$, for a short time interval. Making use that $\Phi$ is a diffeomorphism we also can define the smooth curve $t \mapsto \omega_t \in \mathcal{B}$ satisfying $\omega_t^{n-1}=\mu_t$. Now, introducing the auxiliary functions
	\begin{equation*}
		\phi(t) = \int_{\mathbb{C}P^{n-1}} \mu_t\;\; \mbox{and} \;\; \psi(t) =\left( \int_{\mathbb{C}P^{n}} \mu_t \wedge \omega_t \right)^{\frac{n-1}{n}},
	\end{equation*}
	we can express the functional $\mathcal{F}$ along the curve $\mu_t$ as $\mathcal{F}(\mu_t) = \phi(t)/\psi(t)$. Since, the functional $\mathcal{F}$ is Fréchet differentiable and the curve $t \mapsto \mu_t$ has initial conditions $\mu_0 = \Phi(\omega)$ and $\dot{\mu}_0 = \mu$, the first derivative of $\mathcal{F}$ is expressed as 
	\begin{equation*}
		\restr{d \mathcal{F}}{\Phi(\omega)} \, \mu = \restr{\frac{d}{dt}\mathcal{F}(\mu_t)}{t=0} = \phi'(0)-\phi(0)\psi'(0),
	\end{equation*} 
	where we used that $\psi(0)=1$.
	
	A straightforward computation shows that $\phi'(t)$ and $\psi'(t)$ can be expressed as
	\begin{align}\label{psiphider1}
		\begin{split}
			\phi'(t)&=\int_{\mathbb{C}P^{n-1}}\mu,  \\ \psi'(t)&=\frac{n-1}{n}\left(\int_{\mathbb{C}P^n} \mu_t \wedge \omega_t\right)^{-\frac{1}{n}}\left( \int_{\mathbb{C}P^n} \mu \wedge \omega_t + \mu_t \wedge \frac{\partial}{\partial t} \omega_t\right).
		\end{split}
	\end{align} 
	Taking a derivative of $\omega^{n-1}_t$ we obtain $ (n-1)\omega^{n-2}_t \wedge \frac{\partial}{\partial t} \omega_t$. Wedging this equality with $\omega_t$, we further obtain
	$$\frac{1}{n-1} \mu \wedge \omega_t = \omega^{n-1}_t \wedge \frac{\partial}{\partial t} \omega_t = \mu_t \wedge \frac{\partial}{\partial t} \omega_t,$$
	allowing us to reach the following simplification of $\psi'(t)$: 
	\begin{equation}\label{psiphider2}
		\psi'(t)=\left(\int_{\mathbb{C}P^n} \mu_t \wedge \omega_t\right)^{-\frac{1}{n}}\left( \int_{\mathbb{C}P^n} \mu \wedge \omega_t\right).
	\end{equation}   
	
	Finally, evaluating the equations $(\ref{psiphider1})$ and $(\ref{psiphider2})$ at $t=0$ together with the fact that $\restr{(\mu_t\wedge \omega_t)}{t=0}=\omega^{n}$, we obtain
	\begin{equation*}
		\restr{d \mathcal{F}}{\Phi(\omega)}\, \mu= \phi'(0)-\phi(0)\psi'(0)=\left( \int_{\mathbb{C}P^{n-1}}{\mu} \right) - \left(\int_{\mathbb{C}P^{n-1}} \omega^{n-1} \right)\left( \int_{\mathbb{C}P^n} \mu \wedge \omega \right), 
	\end{equation*}
	as desired. 
\end{proof}

An immediate consequence of the first variational formula is that the Kähler metrics are critical points for the normalized systole functional. 

\begin{coro}\label{fvfforkal}
	Every Kähler metric is a critical point for the normalized systole functional $\mathrm{Sys}^{\mathrm{nor}}_{2n-2}: \mathscr{B}^{1,\nu} \to \mathbb{R}$. 
\end{coro}
\begin{proof}
	In view of the previous identifications, is enough to show that $\restr{d \mathcal{F}}{\Phi(\omega)}\equiv 0$ for every $\omega \in \mathcal{K}$. Even more, since $\mathcal{F}$ is invariant under homothety, there is no lost of generality in restring ourselves to the space of normalized Kähler forms.  
	
	Therefore, fix $\omega \in \mathcal{K} \cap \mathcal{B}_1$ and $\mu \in C_{cl}^{1,\nu}(\Lambda_{\mathbb{R}}^{n-1,n-1})$. Since both forms are closed and $\omega$ is normalized, the Hodge decomposition Theorem implies the existence of $a \in \mathbb{R}$, $\alpha \in C^{2,\nu}(\Lambda^1_{\mathbb{R}})$ and $\beta \in C^{2,\nu}(\Lambda^{2n-3}_{\mathbb{R}})$, such that $\omega= \Omega + d \alpha$ and $\mu = a \Omega^{n-1}+d\beta$. Now recalling that $\int_{\mathbb{C}P^k} \Omega^k=1$, for every $k \geq 1$, and applying the first variational formula for $\mathcal{F}$ together with Stokes' Theorem, we obtain 
	\begin{align*}
		\restr{d \mathcal{F}}{\Phi(\omega)}\, \mu &= \left( \int_{\mathbb{C}P^{n-1}}{\mu} \right) - \left(\int_{\mathbb{C}P^{n-1}} \omega^{n-1} \right)\left( \int_{\mathbb{C}P^n} \mu \wedge \omega \right) \\ 
		&=a \left(\int_{\mathbb{C}P^{n-1}} \Omega^{n-1}\right) - a \left(\int_{\mathbb{C}P^{n-1}} \Omega^{n-1}\right)\left(\int_{\mathbb{C}P^{n}} \Omega^{n}\right)=0.
	\end{align*}
	Since $\mu \in C_{cl}^{1,\nu}(\Lambda_{\mathbb{R}}^{n-1,n-1})$ is  arbitrary we conclude the proof.         
\end{proof}

We proceed with the computation of the second derivative of the functional $\mathcal{F}$. 

\begin{teo}[Second variational formula of $\mathcal{F}$]\label{svfforf}
	If $\omega \in \mathcal{B}_1$, $\eta \in T_{\omega} \mathcal{B}$ and $\mu= \restr{d\Phi}{\omega} \, \eta \in C_{cl}^{1,\nu}(\Lambda_{\mathbb{R}}^{n-1,n-1})$, then
	\begin{align*}
		\frac{1}{2}\restr{d^2 \mathcal{F}}{\Phi(\omega)}(\mu,\mu) = &\left(\int_{\mathbb{C}P^{n-1}}\omega^{n-1}\right) \left( \int_{\mathbb{C}P^{n}}\mu \wedge \omega \right)^2 
		- \left(\int_{\mathbb{C}P^{n-1}}\mu\right) \left( \int_{\mathbb{C}P^{n}}\mu \wedge \omega \right)\\
		&+\left(\int_{\mathbb{C}P^{n-1}}\omega^{n-1}\right) \left(\frac{1}{(n-1)} \left(\int_{\mathbb{C}P^{n}}\mu \wedge \omega \right)^2 - \int_{\mathbb{C}P^{n}}\mu \wedge \eta \right),
	\end{align*}  
	where $\Phi: \mathcal{B} \to C_{cl}^{1,\nu}(\Lambda_+^{n-1,n-1})$, is given by $\Phi(\omega)=\omega^{n-1}$.  
\end{teo}   
\begin{proof}
	Keeping in mind the notation of Theorem \ref{fvfforF}, and making use that $\mathcal{F}$ is smooth in the Fréchet sense together with the fact that $t \mapsto \mu_t$ is a linear variation, we have that 
	$$\restr{d^2 \mathcal{F}}{\Phi(\omega)}(\mu,\mu) = \restr{\frac{d^2}{dt^2}\mathcal{F}(\mu_t)}{t=0}.$$   
	Since $\psi(0)=1$, $\phi''(0)=0$, $\omega$ is normalized and $(\ref{psiphider1})$ holds, we can take the derivative of $\frac{d \mathcal{F}(\mu_t)}{dt}=(\phi'(t)\psi(t)-\phi(t)\psi'(t))/\psi^2(t)$ at $t=0$ to obtain 
	\begin{equation}\label{derivsegdeF}
		\restr{\frac{d^2}{dt^2}\mathcal{F}(\mu_t)}{t=0} = -2 \psi'(0) \left(\restr{d \mathcal{F} }{\Phi(\omega)} \mu\right) - \phi(0) \psi''(0).
	\end{equation}
	All the terms on the right-hand side of this equation have already been computed, with the exception of $\psi''(0)$. To calculate this term, we refer back to formula $(\ref{psiphider2})$ and differentiate it:
	\begin{align*}
		\psi''(t)=-\frac{1}{n}&\left(\int_{\mathbb{C}P^{n}}\mu_t \wedge \omega_t \right)^{\frac{-1-n}{n}} \left(\int_{\mathbb{C}P^{n}}\mu \wedge \omega_t + \mu_t \wedge \frac{\partial}{\partial t} \omega_t \right)\left( \int_{\mathbb{C}P^{n}}\mu \wedge \omega_t \right) \\ 
		&+\left(\int_{\mathbb{C}P^{n}} \sigma_t \wedge \omega_t\right)^{-\frac{1}{n}} \left( \int_{\mathbb{C}P^{n}}\mu \wedge  \frac{\partial}{\partial t } \omega_t \right).
	\end{align*}
	By retrieving the identities $\restr{(\mu_t\wedge \omega_t)}{t=0}=\omega^{n}$ and $\restr{\left(\mu_t \wedge \frac{\partial}{\partial t} \omega_t\right)}{t=0} = \frac{1}{(n-1)} \mu \wedge \omega$, we further obtain 
	$$\psi''(0)=-\frac{1}{(n-1)}\left( \int_{\mathbb{C}P^{n}}\mu \wedge \omega \right)^2 + \left( \int_{\mathbb{C}P^{n}}\mu \wedge \left( \restr{\frac{\partial}{\partial t } \omega_t}{t=0} \right) \right).$$
	
	Additionally, $\restr{d\Phi}{\omega}  \eta = \mu = \restr{d\Phi}{\omega} \left(\restr{\frac{\partial}{\partial t } \omega_t}{t=0}\right)$, which implies that $\eta = \restr{\frac{\partial}{\partial t } \omega_t}{t=0}$. Hence,
	\begin{equation}\label{secderofpsi}
		\psi''(0)=-\frac{1}{(n-1)}\left( \int_{\mathbb{C}P^{n}}\mu \wedge \omega \right)^2 + \left( \int_{\mathbb{C}P^{n}}\mu \wedge \eta \right).
	\end{equation} 
	Therefore, the desired result follows by combining $\eqref{psiphider2}$, $\eqref{derivsegdeF}$, and $\eqref{secderofpsi}$, as well as the first variation formula in Theorem \ref{fvfforF}.
\end{proof}                  

A non-trivial consequence of the second variational formula is that the Hessian of $\mathcal{F}$, over a Kähler form, is coercive in the $L^2$-norm when restricted to the transversal direction of the Kähler forms. To show this, we apply Theorem $\ref{svfforf}$ to the case of Kähler metrics.

\begin{lema}\label{svfforkfrst}
	Let $\omega \in \mathcal{K}$ be a normalized Kähler form, $\eta \in T_{\omega} \mathcal{B}$ and $\mu= \restr{d\Phi}{\omega} \eta \in C_{cl}^{1,\nu}(\Lambda_{\mathbb{R}}^{n-1,n-1})$. Then:
	\begin{enumerate}[label=\alph*),ref=(\alph*)]
		\item\label{svfforkfrst1} $\frac{1}{(n-1)} \restr{d^2 \mathcal{F}}{\Phi(\omega)}(\mu,\mu) = \left( \int_{\mathbb{C}P^{n}} \eta \wedge \omega ^{n-1}\right)^2 - \int_{\mathbb{C}P^{n}} \eta\wedge \eta \wedge \omega^{n-2}$;
		\item\label{svfforkfrst2} If $\alpha \in T_{\omega}\mathcal{K}$, then $\restr{d^2 \mathcal{F}}{\Phi(\omega)}\left(\restr{d\Phi}{\omega}\, \alpha , \mu\right)=0$,   
	\end{enumerate}
	where $\Phi: \mathcal{B} \to C_{cl}^{1,\nu}(\Lambda_+^{n-1,n-1})$, is given by $\Phi(\omega)=\omega^{n-1}$. 
\end{lema}
\begin{proof}
	First we prove $\ref{svfforkfrst1}$. Let $\omega \in \mathcal{K} \cap \mathcal{B}_1$. Then, by Corollary \ref{fvfforkal}, we see that $\restr{d \mathcal{F}}{\Phi(\omega)}\equiv 0$. Therefore, by recollecting equations $(\ref{derivsegdeF})$ and $(\ref{secderofpsi})$, we have
	\begin{equation*}
		\restr{d^2 \mathcal{F}}{\Phi(\omega)}(\mu,\mu) = \left(\int_{\mathbb{C}P^{n-1}}\omega^{n-1}\right) \left(\frac{1}{(n-1)} \left(\int_{\mathbb{C}P^{n}}\mu \wedge \omega \right)^2 - \int_{\mathbb{C}P^{n}}\mu \wedge \eta \right).
	\end{equation*}
	
	On the other hand, since $\omega$ is Kähler and normalized, we can apply the Hodge Decomposition Theorem to write it as $\omega=\Omega + d\beta$, implying that $\int_{\mathbb{C}P^{n-1}}\omega^{n-1}=1$. Furthermore, by the definition of $\mu$, we have $\mu= \restr{d \Phi}{\omega} \eta = (n-1) \eta \wedge \omega^{n-2}$, leading to the desired equality.
	
	Now we prove $\ref{svfforkfrst2}$. Since $\alpha \in T_{\omega} \mathcal{K}$ and $\eta \in T_{\omega} \mathcal{B}$, the forms $\alpha$ and $\eta \wedge \omega^{n-2}$ are closed. Then, again by the Hodge Decomposition Theorem, we can express them as $\alpha= a \Omega + d \beta$ and $\eta \wedge \omega^{n-2} = b \Omega^{n-1} + d \tilde{\beta}$. Hence, by $\ref{svfforkfrst1}$ and Stokes' Theorem, we see that
	\begin{align*}
		\frac{\restr{d^2 \mathcal{F}}{\Phi(\omega)}\left(\restr{d \Phi}{\omega}\, \alpha,\mu\right)}{(n-1)}  &= \left( \int_{\mathbb{C}P^{n}} \alpha \wedge \omega ^{n-1}\right)\left( \int_{\mathbb{C}P^{n}} \eta \wedge \omega ^{n-1}\right) - \int_{\mathbb{C}P^{n}} \alpha \wedge \eta \wedge  \omega^{n-2} \\
		&= \left( a\int_{\mathbb{C}P^{n}} \Omega^{n}\right)\left(b \int_{\mathbb{C}P^{n}} \Omega^n\right) - ab\int_{\mathbb{C}P^{n}} \Omega^n = 0,
	\end{align*}     
	as intended.
\end{proof}

\begin{coro}\label{sffforfoverkal}
	Let $\omega \in \mathcal{K}$ be a normalized smooth Kähler form, and let $\eta \in T_{\omega} \mathcal{B}$. Suppose that $\eta$ has the Hodge decomposition with respect to the metric $g_\omega$ given by $\eta = a\omega + d\alpha + \delta_{\omega}\theta$. Then, if $\mu = \restr{d \Phi}{\omega} \, \eta$, we have
	\begin{equation}\label{hessofnormsys}
		\frac{1}{(n-1)} \restr{d^2 \mathcal{F}}{\Phi(\omega)}(\mu,\mu) = \int_{\mathbb{C}P^{n}} || \delta_{\omega} \theta ||_{g_{\omega}}^2 dV_{g_{\omega}}.
	\end{equation}
	Here, $\Phi: \mathcal{B} \to C_{cl}^{1,\nu}(\Lambda_+^{n-1,n-1})$, is given by $\Phi(\omega)=\omega^{n-1}$, and the Riemannian metric $g_\omega$ has been extended to the space of differential forms.
\end{coro}
\begin{proof}
	Let $\omega \in \mathcal{K}$ be a normalized smooth Kähler form, and $\eta = a\omega + d \alpha + \delta_{\omega} \theta$. By Lemma $\ref{coexactpartoftoB}$, we have that $\delta_{\omega} \theta \in T_{\omega} \mathcal{B}$. Moreover, applying Lemma $\ref{svfforkfrst}$ $\ref{svfforkfrst2}$ and observing that $a\omega + d\alpha \in T_{\omega}\mathcal{K}$, we obtain the following simplification for the Hessian of $\mathcal{F}$:
	\begin{align*}
		\frac{1}{(n-1)} \restr{d^2 \mathcal{F}}{\Phi(\omega)}(\mu,\mu) &= \frac{1}{(n-1)} \restr{d^2 \mathcal{F}}{\Phi(\omega)}\left(\restr{d \Phi}{\omega} \, \delta_{\omega} \theta,\restr{d \Phi}{\omega} \, \delta_{\omega} \theta\right) 
		\\&= \left( \int_{\mathbb{C}P^{n}} \delta_{\omega}\theta \wedge \omega ^{n-1}\right)^2 - \int_{\mathbb{C}P^{n}} \delta_{\omega}\theta\wedge \delta_{\omega}\theta \wedge \omega^{n-2} 
		\\&= - \int_{\mathbb{C}P^{n}} \delta_{\omega}\theta\wedge \delta_{\omega}\theta \wedge \omega^{n-2}.
	\end{align*}
	In the last equality, we used the fact that $\delta_{\omega}$ is the $L^2$-dual of $d$, and $\omega$ is closed. Moreover, taking the Riemann-Hodge pairing (see definition \ref{RiemmanHodgepair}) of $\delta_{\omega}\theta$ with it self we obtain the term $-\delta_{\omega}\theta \wedge \delta_{\omega}\theta \wedge \omega^{n-2}$. Since $\delta_{\omega} \theta$ is a primitive form of type $(1,1)$, as stated in Lemma \ref{coexactpartoftoB}, the desired result follows from Theorem $\ref{RHtheorem}$.
\end{proof}

\subsection{Main Theorem}
Gathering the results of Sections $\ref{manfstrctoBalMet}$ and $\ref{firstandsecvariation}$ we present a proof of Theorem \ref{sysofbal}. However, before providing a rigorous demonstration, we will discuss a useful intuition. For convenience, we start by summarizing the previous results in the following Lemma.
\begin{lema}\label{summarizinglemmablc} 
	Let $\mathcal{F}: C_{cl}^{1,\nu}\big(\Lambda_+^{n-1,n-1}\big) \to \mathbb{R}$ denote the normalized systole functional under the global chart $\Phi: \mathcal{B} \to C_{cl}^{1,\nu}\big(\Lambda_+^{n-1,n-1}\big)$. Then, for a fixed smooth normalized Kähler form $\omega_0 \in \mathcal{K}$, there exist open neighborhoods $U \subset \mathcal{K}$ of $\omega_0$ and $V\subset A_{\omega_0}$ of $0$, along with a smooth diffeomorphism $\rho:U\times V \to \rho(V\times U)\subset  C_{cl}^{1,\nu}(\Lambda_+^{n-1,n-1})$, such that the map $F \doteq \mathcal{F}\circ \rho:U \times V  \to \mathbb{R}$ satisfies the following properties:  
	\begin{enumerate}[label=\alph*),ref=(\alph*)]
		\item\label{summarizinglemmablc1}  ${F}$ is constant over the set $U\times \{0\}$.
		\item\label{summarizinglemmablc2}  $\restr{d {F}}{\omega}\equiv 0$, for every $\omega \in U$.
		\item\label{summarizinglemmablc3}  The Hessian map $\restr{d^2 {F}}{\omega_0}: T_{\omega_0}\mathcal{K} \oplus A_{\omega_0} \times T_{\omega_0} \mathcal{K} \oplus A_{\omega_0} \to \mathbb{R}$ is a symmetric, semi-positive definite bilinear form. Moreover, its kernel is given by $T_{\omega_0}\mathcal{K}$.
		\item\label{summarizinglemmablc4}  Given $(\omega,\xi) \in U \times V$, the restriction $\restr{d^2 {F}}{(\omega,\xi)}: A_{\omega_0} \times A_{\omega_0} \to \mathbb{R}$ is given by 
		\begin{equation}\label{transfooftheHess}
			\restr{d^2 {F}}{(\omega, \xi)}(\eta,\eta) =\restr{d^2 \mathcal{F}}{\rho(\omega,\xi)}\left(\restr{d\Phi}{\omega_0}\,\eta,\restr{d\Phi}{\omega_0}\,\eta\right),
		\end{equation}
		for every $\eta \in A_{\omega_0}$. Hence, $\restr{d^2 {F}}{\omega_0}(\eta,\eta) =(n-1) ||\eta ||^2_{L^2_{g_{\omega_0}}}$, where $  ||\eta ||^2_{L^2_{g_{\omega_0}}}$ is given by $L^2$-norm induced by the Kähler metric $g_{\omega_0}$ associated to the Kähler form $\omega_0$.    
	\end{enumerate} 
\end{lema}
\begin{proof}
	Fix a smooth normalized Kähler form $\omega_0 \in \mathcal{K}$. Then, the desired conditions are satisfied by the open neighborhoods $U \subset \mathcal{K}$ of $\omega_0$ and $V\subset A_{\omega_0}$ of $0$, along with the smooth diffeomorphism $\rho:U\times V \to \rho(V\times U)$ provided in Proposition $\ref{KissplitinB}$.
	
	Indeed, $\ref{summarizinglemmablc1}$ and $\ref{summarizinglemmablc2}$ are a direct consequence of Proposition $\ref{basicproproff}$ and Corollary $\ref{fvfforkal}$. Furthermore, to prove $\ref{summarizinglemmablc3}$, we notice that  $\Phi(\omega_0)$ is a critical point of $\mathcal{F}$, and the Hessian of $F$ is given by
	$$\restr{d^2 F}{\omega_0}((\alpha,\eta),(\alpha,\eta))= \restr{d^2 \mathcal{F}}{\Phi(\omega_0)}{\left(\restr{d\rho}{\omega_0}\, (\alpha,\eta),\restr{d\rho}{\omega_0} \, (\alpha,\eta) \right)}, $$
	for every $\alpha \in T_{\omega_0}\mathcal{K}$ and $\eta \in A_{\omega_0}$. Therefore, $\ref{summarizinglemmablc3}$ result from equation $(\ref{deriofrho})$  and Corollary $\ref{sffforfoverkal}$.
	
	In order to prove $\ref{summarizinglemmablc4}$, note that for $\eta \in A_{\omega_0}$, the definition of $\rho$ in $(\ref{deriofrho})$ implies that
	$$\frac{d}{dt}F(\omega,\xi+t \eta)= \restr{d \mathcal{F}}{\rho\left(\omega,\xi+t \eta\right)} \left(\restr{d\Phi}{\omega_0}\, \eta \right).$$
	Then, the transformation law given in $(\ref{transfooftheHess})$ follows by taking a derivative of the above equation. The second part of $\ref{summarizinglemmablc4}$ follow from Corollary~\ref{sffforfoverkal} along with the definition of $A_{\omega_0}$.        
\end{proof}  

It is interesting to observe, as an intuition, that if ${G}: U \times V \subset \mathcal{K} \times {A_{\omega_0}} \to \mathbb{R}$ is a smooth map that satisfies properties $\ref{summarizinglemmablc1}$ through $\ref{summarizinglemmablc3}$ of Lemma $\ref{summarizinglemmablc}$, together with the fact that the restriction $\restr{d^2 {G}}{(\omega,\xi)}: A_{\omega_0} \times A_{\omega_0} \to \mathbb{R}$ is coercive in the $C^{1,\nu}$-topology, then ${G}(\omega,\eta) \geq {G}(\omega_0)$ in a neighborhood of $\omega_0$.

In fact, since $G: U \times V \to \mathbb{R}$ is smooth in the Fréchet sense with respect to the $C^{1,\nu}$-norm, the second-order Taylor expansion with Lagrange remainder around $\omega \in U$ implies the existence of a constant $\lambda = \lambda(\eta) \in (0,1)$, resulting in the following bound:
\begin{align*}
	G(\omega,\eta) &=G(\omega) + \restr{dG}{\omega}\, \eta + \frac{1}{2}\restr{d^2 G}{(\omega,\lambda \eta)}(\eta,\eta)  \\
	& = G(\omega) + \restr{dG}{\omega}\, \eta + \frac{1}{2}\restr{d^2 G}{\omega_0}(\eta,\eta) + \frac{1}{2}\left( \restr{d^2 G}{(\omega,\lambda \eta)}(\eta,\eta) - \restr{d^2 G}{\omega_0}(\eta,\eta) \right)  \\ 
	& \geq G(\omega) + \restr{dG}{\omega}\, \eta + C||\eta||^2_{C^{1,\nu}} + \frac{1}{2}\left( \restr{d^2 G}{(\omega,\lambda \eta)}(\eta,\eta) - \restr{d^2 G}{\omega_0}(\eta,\eta) \right),   
\end{align*} 
where the constant $C=C(\omega_0)>0$ arises form the coercivity condition. Now, by the assumed properties of the map $G$ together with the continuity of $d^2G$ in the $C^{1,\nu}$-topology, we further obtain   
\begin{align*}
	G(\omega,\eta) & \geq G(\omega) + \restr{dG}{\omega}\cdot \eta + C||\eta||^2_{C^{1,\nu}} + \frac{1}{2}\left( \restr{d^2 G}{(\omega,\lambda \eta)}(\eta,\eta) - \restr{d^2 G}{\omega_0}(\eta,\eta) \right)   \\
	& \geq G(\omega_0) + C||\eta||^2_{C^{1,\nu}} -\frac{C}{2} ||\eta||^2_{C^{1,\nu}} \\ 
	&= G(\omega_0) + \frac{C}{2} ||\eta||^2_{C^{1,\nu}},
\end{align*} 
after shrinking $U$ and $V$, if necessary. Therefore, the desired result follows from classical arguments in view of the last inequality.

In the situation we want to analyze, however, the function $F: U \times V \subset \mathcal{K} \times A_{\omega_0} \to \mathbb{R}$ has a Hessian that is not coercive in the $C^{1,\nu}$-topology, but satisfies the weaker property stated in Lemma $\ref{summarizinglemmablc}$ $\ref{summarizinglemmablc4}$ instead. In order to bypass this problem, our strategy is to estimate the $L^2$-norm of the Hessian, and then mimic the previous argument. Specifically, we present the following.
\begin{lema}\label{L2boundofHess}
	Let $\omega_0 \in \mathcal{K}$ be a smooth and normalized Kähler form. Then, there exist a neighborhood $\mathcal{N} \subset \mathcal{B}$ of $\omega_0$, in the $C^{1,\nu}$-topology, such that for each $\omega \in \mathcal{N}$ and $\eta \in T_{\omega_0}\mathcal{B}$ the equality
	\begin{equation*}
		\left| \restr{d^2 \mathcal{F} }{\Phi(\omega_0)}\left(\mu,\mu\right) - \restr{d^2 \mathcal{F} }{\Phi(\omega)}\left(\mu,\mu\right) \right| \leq \frac{n-1}{2} \,||\eta||^2_{L^2_{g_{\omega_0}}},
	\end{equation*}
	holds, where $\mu= \restr{d \Phi}{\omega_0} \, \eta \in C^{1,\nu}_{cl}\left(\Lambda^{n-1,n-1}_{\mathbb{R}} \right)$. 
\end{lema}           
\begin{proof}
	We begin by observing that it is enough to prove the existence of a neighborhood $\mathcal{N}_1 \subset \mathcal{B}_1$ of $\omega_0$, such that
	\begin{equation}\label{hessinforB1}
		\left| \restr{d^2 \mathcal{F} }{\Phi(\omega_0)}\left(\mu,\mu\right) - \restr{d^2 \mathcal{F} }{\Phi(\omega)}\left(\mu,\mu\right) \right| \leq \frac{n-1}{4} \,||\eta||^2_{L^2_{g_{\omega_0}}},
	\end{equation}
	for every $\omega \in \mathcal{N}_1$ and $\eta \in T_{\omega_0}\mathcal{B}$. 
	
	Indeed, consider the continuous map $v: \mathcal{B} \to \mathbb{R}_{>0}$ defined by $v(\omega) = \int_{\mathbb{C}P^n} \omega^n$, which allows us to normalize any balanced form ${\omega} \in \mathcal{B}$ as $\tilde{\omega} \doteq v(\omega)^{-\frac{1}{n}}\omega \in \mathcal{B}_1$. Furthermore, the homothety invariance property of $\mathcal{F}$ implies the following relation between the Hessians over $\omega$ and $\tilde{\omega}$
	$$ \restr{d^2 \mathcal{F}}{\Phi(\omega)} = v(\omega)^{\frac{2n}{n-1}} \restr{d^2 \mathcal{F}}{\Phi(\tilde{\omega})} .$$
	
	\noindent Consequently, for every $\omega$ in the open set $\mathcal{N}' \doteq \left\{ \omega \in \mathcal{B} \, : \, v(\omega)^{-\frac{1}{n}}\omega \in \mathcal{N}_1 \right\}$ and $\mu= \restr{d \Phi}{\omega_0} \, \eta \in C^{1,\nu}_{cl}\big(\Lambda^{n-1,n-1}_{\mathbb{R}} \big)$, the following inequality holds
	\begin{align*}
		&\left| \restr{d^2 \mathcal{F} }{\Phi(\omega_0)}\left(\mu,\mu\right) - \restr{d^2 \mathcal{F} }{\Phi(\omega)}\left(\mu,\mu\right) \right| \leq\\ 
		&\leq \left|1- v(\omega)^\frac{2n}{n-1} \right| \restr{d^2 \mathcal{F} }{\Phi(\omega_0)}\left(\mu,\mu\right) + v(\omega)^\frac{2n}{n-1}\left| \restr{d^2 \mathcal{F} }{\Phi(\omega_0)}\left(\mu,\mu\right) - \restr{d^2 \mathcal{F} }{\Phi(\tilde{\omega})}\left(\mu,\mu\right) \right| \\
		&\leq \left|1- v(\omega)^\frac{2n}{n-1} \right| \restr{d^2 \mathcal{F} }{\Phi(\omega_0)}\left(\mu,\mu\right)  +\frac{(n-1)}{4}\,v(\omega)^\frac{2n}{n-1}\,||\eta||^2_{L^2_{g_{\omega_0}}} \\
		&\leq (n-1) \left(\left|1- v(\omega)^\frac{2n}{n-1} \right|+ \frac{v(\omega)^\frac{2n}{n-1}}{4} \right)||\eta||^2_{L^2_{g_{\omega_0}}},
	\end{align*} 
	where we have applied $\eqref{hessinforB1}$ along with Lemma $\ref{summarizinglemmablc}$ $\ref{summarizinglemmablc3}$ and $\ref{summarizinglemmablc4}$. Moreover, since $v$ is continuous and $v(\omega_0) = 1$, the desired neighborhood can be define as $\mathcal{N} \doteq \left\{ \omega \in \mathcal{N}'\; : \; \left|1- v(\omega)^\frac{2n}{n-1} \right|+ \frac{1}{4}{v(\omega)^\frac{2n}{n-1}} < 1/2  \right\}$.
	
	It remains to prove the existence of the neighborhood $\mathcal{N}_1 \subset \mathcal{B}_1$. Fix $\omega \in \mathcal{B}_1$ and $\mu= \restr{d \Phi}{\omega_0} \, \eta \in C^{1,\nu}_{cl}\left(\Lambda^{n-1,n-1}_{\mathbb{R}} \right)$. By Theorem $\ref{svfforf}$, we can write the Hessian of $\mathcal{F}$ over $\omega$ as
	\begin{equation*}
		\restr{d^2 \mathcal{F}}{\Phi(\omega)}\left(\mu,\mu\right) = 2P_{\omega}(\mu,\mu) + \frac{1}{(n-1)}\, R^1_{\omega}(\mu,\mu) + R^{2}_{\omega}(\mu,\mu),  
	\end{equation*}
	where the operators $P_\omega$, $R^{1}_\omega$ and $R^{2}_{\omega}$ are given by
	\begin{align*}
		&P_\omega(\mu,\mu) =\left( \int_{\mathbb{C}P^{n}}\mu \wedge \omega \right) \left(\left(\int_{\mathbb{C}P^{n-1}}\omega^{n-1}\right) \left( \int_{\mathbb{C}P^{n}}\mu \wedge \omega \right) 
		- \left(\int_{\mathbb{C}P^{n-1}}\mu\right) \right),\\
		&R^{1}_{\omega}(\mu,\mu)=\left(\int_{\mathbb{C}P^{n-1}}\omega^{n-1}\right) \left(\int_{\mathbb{C}P^{n}}\mu \wedge \omega \right)^2,\\
		&R^{2}_{\omega}(\mu,\mu)= \left(\int_{\mathbb{C}P^{n-1}}\omega^{n-1}\right) \left(\int_{\mathbb{C}P^{n}}\mu \wedge \restr{d\Psi}{\Phi(\omega)}\,\mu \right),
	\end{align*}     
	where, $\Psi = \Phi^{-1}:C_{cl}^{1,\nu}(\Lambda_+^{n-1,n-1}) \to  \mathcal{B}$.
	
	As showed in Corollary $\ref{fvfforkal}$ we have that $P_{\omega_0}=0$, since $\omega_0$  is Kähler. This leads to the estimate 
	\begin{align*}
		\left| \restr{d^2 \mathcal{F} }{\Phi(\omega_0)}\left(\mu,\mu\right) - \restr{d^2 \mathcal{F} }{\Phi(\omega)}\left(\mu,\mu\right) \right| \leq 2\left|P_\omega(\mu,\mu) \right| &+ \frac{\left|R^1_{\omega_0}(\mu,\mu)- R^1_{\omega}(\mu,\mu)\right|}{(n-1)}\\
		&+ \left|R^2_{\omega_0}(\mu,\mu) - R^2_{\omega}(\mu,\mu) \right|.
	\end{align*}
	Therefore, it suffices to study each of the terms on the right-hand side independently.
	
	We start with the operator $P_\omega$. First, note that every given closed form $\alpha \in C^{1,\nu}(\Lambda^{2n-2})$ can be written as $\alpha = a \omega_0^{n-1} + d\xi$. Therefore, by applying Stokes's Theorem and recalling that $\omega_0$ is normalized, we have that 
	$$\int_{\mathbb{C}P^{n-1}} \alpha = a\int_{\mathbb{C}P^{n-1}} \omega_0^{n-1} = \int_{\mathbb{C}P^{n}} \alpha \wedge \omega_0.$$ 
	Consequently, we can rewrite $P_{\omega}$ as 
	$$P_\omega(\mu,\mu) =\left( \int_{\mathbb{C}P^{n}}\mu \wedge \omega \right) \left(\left(\int_{\mathbb{C}P^{n-1}}\omega^{n-1}\right) \left( \int_{\mathbb{C}P^{n}}\mu \wedge \omega \right) 
	- \left(\int_{\mathbb{C}P^{n}}\mu \wedge \omega_0\right) \right) $$
	
	To further simplify notation, we introduce the continuous map: $w : \mathcal{B} \to \mathbb{R}$, given by $w(\omega) = \int_{\mathbb{C}P^{n-1}} \omega^{n-1}$. This implies that 
	\begin{align*}
		\left| P_{\omega}(\mu,\mu)\right| &= \left| \left( \int_{\mathbb{C}P^{n}}\mu \wedge \omega \right)\right|\, \left| w(\omega) \left( \int_{\mathbb{C}P^{n}}\mu \wedge \omega \right) 
		- \left(\int_{\mathbb{C}P^{n}}\mu \wedge \omega_0\right) \right| \\
		&= \left| \left( \int_{\mathbb{C}P^{n}}\mu \wedge \omega \right)\right|\, \left|\int_{\mathbb{C}P^{n}}\mu \wedge \left( w(\omega)\omega- \omega_0\right) \right|. 
	\end{align*}

	Recall that for any top form $\xi \in C^{1,\nu}(\Lambda^{2n}_{\mathbb{R}})$, it holds that $\left|\int_{\mathbb{C}P^n} \xi \right| \leq \int_{\mathbb{C}P^{n}} ||\xi||_{g_{\omega_0}} dV_{g_{\omega_0}}$. Furthermore, since the complex projective space is compact, there exists a universal constant $C > 0$ such that $||\alpha \wedge \beta ||_{g_{\omega_0}} \leq C ||\alpha||_{g_{\omega_0}} ||\beta||_{g_{\omega_0}}$. Moreover, $C>0$ will also denote a constant that may possibly change throughout the calculations but depends only on $\omega_0$ and $n$. Considering the previous observations, we have 
	\begin{align*}
		\left| P_{\omega}(\mu,\mu)\right| &= \left| \left( \int_{\mathbb{C}P^{n}}\mu \wedge \omega \right)\right|\, \left|\int_{\mathbb{C}P^{n}}\mu \wedge \left( w(\omega)\omega- \omega_0\right) \right| \\
		&\leq C  \left( \int_{\mathbb{C}P^{n}}||\mu||_{g_{\omega_0}} ||\omega||_{g_{\omega_0}} \right) \left(\int_{\mathbb{C}P^{n}}||\mu||_{g_{\omega_0}} ||w(\omega)\omega- \omega_0||_{g_{\omega_0}} \right)\\
		&\leq C ||\omega||_{C^{1,\nu}}||w(\omega)\omega- \omega_0||_{C^{1,\nu}} \left(\int_{\mathbb{C}P^{n}}||\mu||_{g_{\omega_0}}  \right)^2.
	\end{align*}    
	However, $\mu= \restr{d\Phi}{\omega_0} \,\eta = (n-1) \omega^{n-2}_0 \wedge \eta $. Hence, applying Hölder inequality we can find a new constant $C>0$, such that
	\begin{align*}
		\left| P_{\omega}(\mu,\mu)\right| &\leq C ||\omega||_{C^{1,\nu}}||w(\omega)\omega- \omega_0||_{C^{1,\nu}} \left((n-1)\int_{\mathbb{C}P^{n}} || \omega_0||^{n-2}_{g_{\omega_0}} ||\eta||_{g_{\omega_0}}  \right)^2\\
		&\leq C ||\omega||_{C^{1,\nu}}||w(\omega)\omega- \omega_0||_{C^{1,\nu}} ||\eta||^2_{L^2_{g_{\omega_0}}}.    	
	\end{align*}      
	
	Noticing that the function $\mathcal{B}_1 \ni \omega \mapsto C ||\omega||_{C^{1,\nu}}||w(\omega)\omega - \omega_0||_{C^{1,\nu}} \in \mathbb{R}$ is continuous and vanishes at $\omega_0$, there exists a neighborhood $\mathcal{W}_1 \subset \mathcal{B}_1$ of $\omega_0$, where the following inequality holds
	$$ \left| P_{\omega}(\mu,\mu)\right| \leq \frac{n-1}{12} \, ||\eta||^2_{L^2_{\omega_0}}, $$
	for any $\omega \in \mathcal{W}_{1}$ and  $\mu= \restr{d \Phi}{\omega_0} \, \eta \in C^{1,\nu}_{cl}\big(\Lambda^{n-1,n-1}_{\mathbb{R}} \big)$.    
	
	Now, we estimate the operator $R^1$. Upon noticing that $w(\omega_0) = 1$ and employing the same reasoning as before, we obtain for each $\omega$ in the non-empty open set $\{\omega \in \mathcal{B}_1: w(\omega) > 0\}$, and for $\mu = \restr{d\Phi}{\omega_0} \, \eta = (n-1) \eta \wedge \omega_0^{n-2}$, where $\eta \in T_{\omega_0} \mathcal{B}$ that
	\begin{align*}
		\left| R^{1}_{\omega}(\mu,\mu)  - R^{1}_{\omega_0}(\mu,\mu) \right| &= \left| w(\omega)\left( \int_{\mathbb{C}P^{n}} \mu \wedge \omega \right)^2 - \left( \int_{\mathbb{C}P^{n}} \mu \wedge \omega_0 \right)^2 \right| \\
		&= \left|  \int_{\mathbb{C}P^{n}}\mu\wedge (w(\omega)^{\frac{1}{2}}\omega + \omega_0 ) \right|\, \left| \int_{\mathbb{C}P^{n}}\mu\wedge (w(\omega)^{\frac{1}{2}}\omega - \omega_0 ) \right| \\ 
		& \leq C || w(\omega)^{\frac{1}{2}}\omega - \omega_0 ||_{C^{1,\nu}} ||w(\omega)^{\frac{1}{2}}\omega + \omega_0||_{C^{1,\nu}} \left(\int_{\mathbb{C}P^{n}}\mu \right)^2\\
		&\leq C || w(\omega)^{\frac{1}{2}}\omega - \omega_0 ||_{C^{1,\nu}} ||w(\omega)^{\frac{1}{2}}\omega + \omega_0||_{C^{1,\nu}} || \eta ||^2_{L^2_{\omega_0}}.
	\end{align*}    
	As before, notice that the map $\omega \mapsto C || w(\omega)^{\frac{1}{2}}\omega - \omega_0 ||_{C^{1,\nu}} ||w(\omega)^{\frac{1}{2}}\omega + \omega_0||_{C^{1,\nu}}$ is a continuous function that vanishes at $\omega_0$. We can define a neighborhood $\mathcal{W}_2 \subset \mathcal{B}_1$ of $\omega_0$, in such way that
	\begin{equation*}
		\left| R^{1}_{\omega}(\mu,\mu)  - R^{1}_{\omega_0}(\mu,\mu) \right| \leq \frac{n-1}{12} \, ||\eta||^2_{L^2_{g_{\omega_0}}},
	\end{equation*}  
	for every $\omega \in \mathcal{W}_{2}$ and  $\mu= \restr{d \Phi}{\omega_0} \, \eta \in C^{1,\nu}_{cl}\big(\Lambda^{n-1,n-1}_{\mathbb{R}} \big)$.
	
	Finally we estimate the operator $R^2$. Once more, taking $\omega \in \mathcal{B}_1$ such that $w(\omega) > 0$, and $\mu = \restr{d\Phi}{\omega_0} \, \eta = (n-1) \eta \wedge \omega_0^{n-2}$, where $\eta \in T_{\omega_0} \mathcal{B} \subset C^{1,\nu}\big(\Lambda_{\mathbb{R}}^{1,1}\big)$, we have
	\begin{align}\label{R2equat}
		\begin{split}
			\left| R^{2}_{\omega}(\mu,\mu)  - R^{2}_{\omega_0}(\mu,\mu) \right| &= \left| w(\omega)\left( \int_{\mathbb{C}P^{n}} \mu \wedge \restr{d \Psi}{\Phi(\omega)}\, \mu \right) - \left( \int_{\mathbb{C}P^{n}} \mu \wedge \eta \right) \right|\\
			& =\left|  \int_{\mathbb{C}P^{n}} \mu \wedge \left( w(\omega)\restr{d \Psi}{\Phi(\omega)}\, \mu - \eta \right)  \right|.
		\end{split}
	\end{align}
	Turning our attention to the map $\restr{d \Phi}{\omega}$, we recall that it is induced by the bundle isomorphism $\Lambda_{\mathbb{R}}^{1,1} \ni \alpha \mapsto (n-1)\alpha \wedge \omega^{n-2} \in \Lambda^{n-1,n-1}_{\mathbb{R}}$. Therefore,  its inverse is induced by the inverse of this bundle isomorphism. Consequently, if we denote such bundle map by $S_{\omega}: \Lambda^{n-1,n-1}_{\mathbb{R}} \to \Lambda^{1,1}_{\mathbb{R}}$, we obtain the following pointwise bound
	\begin{align*}
		||w(\omega) \restr{d \Psi}{\Phi(\omega)}\, \mu - \eta ||_{g_{\omega_0}} 
		& =||w(\omega)S_{\omega}\mu - \eta ||_{g_{\omega_0}}\\ &\leq ||S_{\omega} ||_{g_{\omega_0}} ||w(\omega)\mu - S_{\omega}^{-1} \eta||_{g_{\omega_0}} \\
		&=  (n-1)||S_{\omega} ||_{g_{\omega_0}} ||w(\omega)\eta \wedge {}\omega_0^{n-2} - \eta\wedge {\omega}^{n-2} ||_{g_{\omega_0}}  \\
		&\leq C||S_{\omega} ||_{g_{\omega_0}} ||w(\omega) \omega_0^{n-2} - {\omega}^{n-2} ||_{g_{\omega_0}} || \eta ||_{g_{\omega_0}}.  
	\end{align*}
	
	Using the compactness of the complex projective space and the equivalence of Euclidean products, we can obtain a neighborhood $\mathcal{W}_3 \subset \mathcal{B}_1$ of $\omega_0$ such that $||S_\omega||_{g_{\omega_0}} \leq 2 ||S_{\omega} ||_{g_\omega}$ at every point. Moreover, since we can put any linear Kähler form in canonical form, and $S^{-1}_{\omega}$ is wedging with the fundamental form, we conclude that $||S_{\omega} ||_{g_\omega} = ||S_{\omega_0} ||_{g_{\omega_0}}$ for every $\omega \in \mathcal{B}$.
	
	Combining the aforementioned pointwise information with $(\ref{R2equat})$, we obtain the following inequality for every $\omega \in \mathcal{W}_3$
	
	\begin{align*}
		\left| R^{2}_{\omega}(\mu,\mu)  - R^{2}_{\omega_0}(\mu,\mu) \right| &= \left| w(\omega)\left( \int_{\mathbb{C}P^{n}} \mu \wedge \restr{d \Psi}{\Phi(\omega)}\, \mu \right) - \left( \int_{\mathbb{C}P^{n}} \mu \wedge \eta \right) \right|\\
		& \leq C \int_{\mathbb{C}P^{n}} ||\eta||_{g_{\omega_0}} \, ||w(\omega)\restr{d \Psi}{\Phi(\omega)}\, \mu - \eta ||_{g_{\omega_0}} \\ 
		&\leq C \int_{\mathbb{C}P^{n}} ||\eta||^2_{g_{\omega_0}} \, \left(||S_{\omega_0} ||_{g_{\omega_0}} ||w(\omega) \omega_0^{n-2} - {\omega}^{n-2} ||_{g_{\omega_0}} \right) \\
		&\leq C ||w(\omega) \omega_0^{n-2} - {\omega}^{n-2} ||_{C^{1,\nu}} \left(\int_{\mathbb{C}P^n} ||\eta||^2_{g_{\omega_0}}\right) \\
		& \leq C ||w(\omega) \omega_0^{n-2} - {\omega}^{n-2} ||_{C^{1,\nu}} \,||\eta||^2_{L^2_{g_{\omega_0}}},
	\end{align*}  
	where, in the last line, we applied Hölder inequality, and $C>0$ is a constant that depends of $n$, $\omega_0$ and $\mathcal{W}_3$. As the map $\mathcal{W}_3 \ni \omega \to C' ||w(\omega) \omega_0^{n-2} - {\omega}^{n-2} ||_{C^{1,\nu}}\in \mathbb{R}$ is continuous and vanishes at $\omega_0$, we can shrink $\mathcal{W}_3$ to ensure that
	$$\left| R^{2}_{\omega}(\mu,\mu)  - R^{2}_{\omega_0}(\mu,\mu) \right|  \leq \frac{n-1}{12} \, ||\eta||^2_{L^2_{g_{\omega_0}}}, $$
	for every $\omega \in \mathcal{W}_{3}$ and  $\mu= \restr{d \Phi}{\omega_0} \, \eta \in C^{1,\nu}_{cl}\big(\Lambda^{n-1,n-1}_{\mathbb{R}} \big)$. We can conclude the proof defining $\mathcal{N}_1=\mathcal{W}_1 \cap \mathcal{W}_2 \cap \mathcal{W}_3$.         
	
\end{proof}

Now that Lemma $\ref{L2boundofHess}$ is established, we can apply the computations provided in the beginning of the section to give a proof of Theorem $\ref{sysofbal}$. Which we restate below in terms of the finer topology $C^{1,\nu}$ and of the functional $\mathcal{F}: C^{1,\nu}_{cl}\big(\Lambda_+^{n-1,n-1}\big) \to \mathbb{R}$.
\begin{teo}
	There is an open set $ {\Phi}\left(\mathcal{K} \cap \Omega^{1,1}(\mathbb{C}P^n)\right) \subset \mathcal{U} \subset C_{cl}^{1,\nu}\big(\Lambda_{+}^{n-1,n-1}\big)$, in the ${C}^{1,\nu}$-topology, such that for every form $\sigma \in \mathcal{U}$ 
	$$\mathcal{F}(\sigma) \geq \mathcal{F}(\Omega^{n-1}).$$
	Moreover, $\sigma \in \mathcal{U}$ satisfies the equality if and only if $\sigma \in {\Phi}\left(\mathcal{K}\right)$.
\end{teo}    
\begin{proof}
	Let $\omega_0$ be a smooth Kähler form, and let $\rho: U \times V \to \rho(U \times V) \subset C_{cl}^{1,\nu}\big(\Lambda_+^{n-1,n-1}\big)$ be the smooth diffeomorphism given in Lemma $\ref{summarizinglemmablc}$. Furthermore, we denote the the functional $\mathcal{F}$ under this identification, by $F \doteq \mathcal{F}\circ \rho: U \times V \to \mathbb{R}.$ 
	
	If $\mathcal{N} \subset \mathcal{B}$ denotes the neighborhood provided in Lemma $\ref{L2boundofHess}$, we can assume that $U \subset \mathcal{K}$ and $V \subset A_{\omega_0}$ are open convex sets which satisfying $\mathcal{W}_{\omega_0} \doteq \rho(U \times V) \subset \Phi\left(\mathcal{N}\right)$.
	
	The second-order Taylor expansion with the Lagrange remainder for ${F}:U \times V \to \mathbb{R}$  around $\omega \in U$ implies that for each $\eta \in V$, there exists $\lambda=\lambda(\eta) \in (0,1)$ such that the following equality holds
	\begin{align*}
		{F}(\omega,\eta)  &={F}(\omega) + \restr{d{F}}{\omega} \eta + \frac{1}{2}\restr{d^2 {F}}{(\omega,\lambda \eta)}(\eta,\eta)  \\
		& = {F}(\omega) + \restr{d{F}}{\omega} \eta + \frac{1}{2}\restr{d^2{F}}{\omega_0}(\eta,\eta) + \frac{1}{2}\left( \restr{d^2 {F}}{(\omega,\lambda \eta)}(\eta,\eta) - \restr{d^2 {F}}{\omega_0}(\eta,\eta) \right). 
	\end{align*} 
	Since we are under the hypothesis of Lemmas $\ref{summarizinglemmablc}$ and $\ref{L2boundofHess}$ we further obtain
	\begin{align*}
		{F}(\omega,\eta) & = {F}(\omega_0) + \frac{n-1}{2}||\eta||^2_{L^2_{g_{\omega_0}}}+\\ &+ \frac{1}{2}\left( \restr{d^2 \mathcal{F}}{\rho(\omega,\lambda \eta)}\left(\restr{d\Phi}{\omega_0}\,\eta,\restr{d\Phi}{\omega_0}\,\eta\right) - \restr{d^2 \mathcal{F}}{\Phi(\omega_0)}\left(\restr{d\Phi}{\omega_0}\,\eta,\restr{d\Phi}{\omega_0}\,\eta \right) \right)   \\
		& \geq {F}(\omega_0) + \frac{n-1}{2}||\eta||^2_{L^2_{g_{\omega_0}}}-\frac{n-1}{4} ||\eta||^2_{L^2_{g_{\omega_0}}} \\ 
		&= {F}(\omega_0) + \frac{n-1}{4} ||\eta||^2_{L^2_{g_{\omega_0}}}.
	\end{align*}   
	Applying Proposition $\ref{basicproproff}$ to ensure that $\mathcal{F}$ is constant along the Kähler forms, we conclude that for every form $\sigma = \rho(\omega,\eta) \in \mathcal{W}_{\omega_0}$, the following inequality holds
	$$\mathcal{F}(\sigma) \geq \mathcal{F}(\Omega^{n-1}) + \frac{n-1}{4}||\eta||^2_{L^2_{g_{\omega_0}}}.$$
	Even more, if equality holds $\eta=0$, that is, $\sigma=\rho(\omega,0)=\Phi(\omega) \in \Phi(\mathcal{K})$. Conversely, applying Proposition $\ref{basicproproff}$, if $\sigma \in \Phi(\mathcal{K})$ then equality holds.   
	
	In conclusion, we constructed the desired neighborhood around each smooth and normalized Kähler form. To complete the proof, we need to extend this construction to non-normalized forms. For that, we recall that by Proposition $\ref{basicproproff}$, the functional $\mathcal{F}$ is invariant under homothety, allowing us to construct the aforementioned neighborhood using dilatation. Finally, we can take $\mathcal{U}$ as the union of $\mathcal{W}_{\omega_0}$, for each $\omega_0 \in \mathcal{K} \cap \Omega^{1,1}(\mathbb{C}P^n)$.
\end{proof}
\section{Deformations in \texorpdfstring{$\mathcal{Z}$}{Z}}\label{section:zolldeformation}

As discussed in the introduction,  this chapter adapts the generalized Zoll condition proposed by Ambrozio-Marques-Neves to the setting of almost complex structures on complex projective space. The results from Chapters $\ref{section:WeaklyZollmetrics}$ and $\ref{section:BalancedMetrics}$ are then applied to provide a classification of their one-parameter deformations. As a consequence of this classification, we are able to investigate the behavior of the codimension-two normalized systole along these deformations. We conclude the chapter by comparing our results with the spherical case, building upon the work of L. Ambrozio and R. Montezuma (\cite{lucas_rafael_sistole_projcspace}, \cite{lucas_rafael_minmax}).

\begin{defi}\label{zollmetrics}
	The set $\mathcal{Z}$ is defined as the class of almost Hermitian structures $(J, g)$ in $\mathbb{C}P^n$ that admit a family $\{\Sigma^{2n-2}_\sigma\}_{\sigma \in \mathbb{C}P^n}$ of $(2n-2)$-dimensional submanifolds satisfying the following properties:
	\begin{enumerate}[label=\alph*),ref=(\alph*)]
		\item\label{zollmetrics1} For every $\sigma \in \mathbb{C}P^n$ the submanifold $\Sigma_\sigma$ is diffeomorphic to $\mathbb{C}P^{n-1}$, minimal and $J$-almost complex;
		\item\label{zollmetrics2} For every $(p ,\Pi) \in \mathrm{Gr}_{n-1}^{J}(\mathbb{C}P^n)$, in the Grassmannian of $J$-almost complex hyperplanes, there exists a unique $\sigma \in \mathbb{C}P^n$ for which $p \in \Sigma_\sigma$ and $T_p \Sigma_\sigma = \Pi$. Moreover, the map $\mathrm{Gr}_{n-1}^{J}(\mathbb{C}P^n) \ni (p,\Pi) \mapsto \sigma \in \mathbb{C}P^n$ is a submersion;
		\item\label{zollmetrics3} The map $\mathbb{C}P^n \ni \sigma \mapsto \Sigma_\sigma$ is smooth in the sense of the graphical convergence.   
	\end{enumerate}
	The family $\{\Sigma^{2n-2}_\sigma\}_{\sigma \in \mathbb{C}P^n}$ is called the associated Zoll family.  
\end{defi}

Following the ideas in \cite{lucas_coda_andre}, our interest is to classify $1$-parameter deformations of the Fubini-Study metric that lie in the set $\mathcal{Z}$. More concretely, a smooth family $t \mapsto (J_t,g_t)$ of almost Hermitian structures is said to be a \textit{$1$-parameter deformation of the Fubini-Study metric in $\mathcal{Z}$} if $(J_t,g_t) \in \mathcal{Z}$ for every $t$, and there exists a family of Zoll families $\{\Sigma_{\sigma,t}\}_{\sigma \in \mathbb{C}P^n}$ such that the map $(\sigma,t) \mapsto \Sigma_{\sigma,t}$ is continuous in the sense of graphical convergence, and moreover $(J_0,g_0)$ and $\{\Sigma_{\sigma,0}\}_{\sigma \in \mathbb{C}P^n}$ are given by $(J_{\mathrm{can}},g_{FS})$ and $\{\mathbb{C}P^{n-1}_{\sigma}\}_{\sigma \in \mathbb{C}P^n}$.

The first step to classify these deformations is to notice that the notion of $(J,g) \in \mathcal{Z}$ presented in the previous definition is a stronger version of the concept of belonging in $\mathcal{W}_{n-1}$, as defined earlier in Chapter $\ref{section:WeaklyZollmetrics}$ (see Definition $\ref{defiofsetWk}$). In other words, we always have that $\mathcal{Z} \subset \mathcal{W}_{n-1}$. Therefore, we can apply Theorem $\ref{classWeaklyZoll}$ to derive basic properties of almost Hermitian structures that are in $\mathcal{Z}$.

\begin{prop}\label{classofzollweakver}
	Let $(J,g)$ be an almost Hermitian structure in $\mathbb{C}P^n$, for $n\geq2$, that belongs to $\mathcal{Z}$. Then:
	\begin{enumerate}[label=\alph*),ref=(\alph*)]
		\item If $n=2$, the almost Hermitian structure $(J,g)$ is Almost-Kähler.
		\item If $n \geq 3$, the almost complex structure $J$ is integrable and the Riemannian metric $g$ is balanced with respect to $J$.  
	\end{enumerate}   
\end{prop}    

A consequence of the previous proposition is that each element $\Sigma_{\sigma}$ in the Zoll family of $(J,g) \in \mathcal{Z}$ is non-trivial in $H_{2n-2}(\mathbb{C}P^n,\mathbb{Z})$. In fact, if $\Sigma_\sigma$ were trivial Stokes' Theorem would imply that $\mathrm{vol}_{g}(\Sigma_{\sigma}) =\frac{1}{(2n-2)!}\int_{\Sigma_{\sigma}} \omega^{n-1} = 0$, since $\omega^{n-1}$ is closed.    

From these preliminary properties we can use the classical theory of deformations of complex manifolds develop by K. Kodaira (\cite{Kodaira05}) and A. Frölicher, A. Nijenhuis (\cite{Frolicher_Nijenhuis}) to prove the following classification theorem.        

\begin{teo}\label{classifofzolldef}
	Fix $n \geq 3$. Let $\mathbb{R} \ni t \mapsto (J_t, g_t) \in \mathcal{Z}$ be a smooth $1$-parameter deformation of the Fubini-Study metric in $\mathcal{Z}$. Then, there exists $\varepsilon >0$ and a continuous map $(-\varepsilon, \varepsilon) \ni t \mapsto \theta(t) \in \mathrm{Diff}(\mathbb{C}P^n)$, such that, modulo isotopy, for every $|t| < \varepsilon$ the following properties are satisfied.
	\begin{enumerate}[label=\alph*),ref=(\alph*)]
		\item The almost complex structure $J_t$ is constant and equal to $J_{\mathrm{can}}$;
		\item The metric $g_t$ is balanced with respect to $J_{\mathrm{can}}$;
		\item The family $\{\Sigma_{\sigma,t}\}_{\sigma \in \mathbb{C}P^{n}}$ is given by $\left\{\mathbb{C}P^{n-1}_{\theta(t,\sigma)}\right\}_{\sigma\in \mathbb{C}P^{n}}.$ 
	\end{enumerate}
\end{teo}
\begin{proof}
	Applying Proposition $\ref{classofzollweakver}$, we conclude that $J_t$ is integrable, and $g_t$ is balanced with respect to $J_t$ for every $t \in \mathbb{R}$. Since $t \mapsto J_t$ is a smooth family of complex structures in $\mathbb{C}P^n$, the deformation Theorem of Kodaira (see Theorem $4.12$, \textsection 4.2 in \cite{Kodaira05}) implies that there exists an $\varepsilon > 0$ and a smooth isotopy $\phi: \mathbb{C}P^n \times (-\varepsilon, \varepsilon) \to \mathbb{C}P^n$ such that $\phi_t^*(J_t) = J_{\mathrm{can}}$. Therefore, up to the action of this isotopy, there is no loss of generality in assuming that $J_t$ is constant, given by the canonical complex structure, and that $g_t$ is balanced with respect to $J_{\mathrm{can}}$ for every $|t| < \varepsilon$.
	
	It remains to show that the family $\{\Sigma_{\sigma,t}\}_{\sigma \in \mathbb{C}P^n}$ is a reparametrization of the equatorial family $\{\mathbb{C}P^{n-1}_\sigma \}_{\sigma \in \mathbb{C}P^n}$. Since the map $(\sigma, t) \mapsto \Sigma_{\sigma,t}$ is continuous with respect to graphical convergence, the function $(\sigma,t) \mapsto \inf_{\Sigma_{\sigma,t}} \mathrm{Scal}_{(\Sigma_{\sigma,t},g_{FS})}$ is also continuous. Hence, possibly after reducing $\varepsilon>0$, we can apply the rigidity theorem for compact complex submanifolds of $(\mathbb{C}P^n,g_{FS})$ with bounded scalar curvature, proven by K. Ogiue (\cite{K_Ogiue}), to conclude that each $\Sigma_{\sigma,t}$ must be a totally geodesic $\mathbb{C}P^{n-1}$. However, the set of connected, compact, complex, totally geodesic codimension two submanifolds of $(\mathbb{C}P^n,g_{FS})$ is precisely given by $\{\mathbb{C}P^{n-1}_{\sigma}\}_{\sigma \in \mathbb{C}P^n}$ (see \cite[Lemma 2, Chapter XI, Section 4]{Kobayashi_Nomizi_vol2}). Thus, for each $(\sigma,t)$ there exists a unique $\theta(\sigma,t)\in \mathbb{C}P^n$ such that $\Sigma_{\sigma,t} = \mathbb{C}P^{n-1}_{\theta(\sigma,t)}$.    
	
	 In other words, there is a map $\theta: \mathbb{C}P^n \times (-\varepsilon, \varepsilon) \to \mathbb{C}P^n$ satisfying $\Sigma_{\sigma,t} = \mathbb{C}P^{n-1}_{\theta(\sigma,t)}$, for every $\sigma \in \mathbb{C}P^n$ and $|t|< \varepsilon$. Now, by part $\ref{zollmetrics2}$ of Definition $\ref{zollmetrics}$, the map $\theta(t,\cdot)$ is bijective and by part $\ref{zollmetrics3}$, is also smooth. Therefore, reducing $\varepsilon>0$ once more, we can assume that each one of these maps is a diffeomorphism. Finally, the continuity of $t \mapsto \theta(t) \in \mathrm{Diff}(\mathbb{C}P^n)$ follows once more from the continuity of $(\sigma, t) \mapsto \Sigma_{\sigma,t}$.              
\end{proof}

The combination of the previous classification Theorem with our analysis of the normalized systole over balanced metrics (Theorem $\ref{sysofbal}$) allow us to understand the normalized systole along a $1$-parameter deformation of the Fubini-Study metric in $\mathcal{Z}$. 

\begin{coro}\label{1pardefsyst}
	Fix $n \geq 3$. Let $\mathbb{R} \ni t \mapsto (J_t, g_t) \in \mathcal{Z}$ be a smooth $1$-parameter deformation of the Fubini-Study metric in $\mathcal{Z}$. Then there exists an $\varepsilon>0$ such that, for every $t \in (-\varepsilon, \varepsilon)$, $$\mathrm{Sys}_{2n-2}^{\mathrm{nor}}(\mathbb{C}P^n,g_t) \geq \mathrm{Sys}_{2n-2}^{\mathrm{nor}}(\mathbb{C}P^n,g_{FS}).$$  
\end{coro}  

We conclude this chapter by comparing our results on deformations of generalized Zoll structures with those discussed earlier for the sphere. Corollary~\ref{1pardefsyst} establishes that the codimension two normalized systole is locally non-decreasing for any 1-parameter deformation of the Fubini-Study metric within the class $\mathcal{Z}$. This behavior contrasts with that of the volume normalized width for $1$-parameter deformations of Zoll surfaces in $\mathbb{S}^3$, where the round metric does not exhibit a similar local minimum type structure.

Indeed, L. Ambrozio and R. Montezuma proved that the round metric is a strict local minimum for the normalized width within the $1$-parameter family of Berger metrics on $\mathbb{S}^3$~\cite[Section 5]{lucas_rafael_sistole_projcspace}. Furthermore, Ambrozio-Marques-Neves constructed $1$-parameter deformations of Zoll surfaces (with non constant sectional curvature) within the conformal class of the round metric~\cite[Theorem A]{lucas_coda_andre}, and the results of Ambrozio-Montezuma show that the normalized width strictly decreases near the round metric for such deformations~\cite[Theorem 1.2.1]{lucas_rafael_minmax}. In summary, we see that the round metric exhibits a saddle-point behavior for the normalized width functional.

\begin{appendices}

\section{Integral Geometric Formulas and Systolic Inequalities}\label{Appendix:IntegralGeometricFormulasandSystolic Inequalities}

\begin{defi}\label{IGFdef}
	Let $(M^n,g)$ be a closed Riemannian manifold, and $\{\Sigma^k_{\sigma}\}_{\sigma \in \mathcal{G}}$  a family of closed smooth $k$-submanifolds continuously parameterized by a closed manifold $\mathcal{G}$. We say that the family $\{\Sigma^k_{\sigma}\}_{\sigma \in \mathcal{G}}$ admits an integral geometric formula, if $\mathcal{G}$ admits a positive Radon measure $d\mu$ that satisfies the following two properties: 
	\begin{enumerate}[label=\alph*),ref=(\alph*)]
		\item\label{IGF1} For every $\phi \in C^{\infty}(M)$ the map $\mathcal{G} \ni \sigma \mapsto \int_{\Sigma_{\sigma}}\phi d{A}_g \in \mathbb{R}$ is continuous;   
		\item \label{IGF2} The following integral equation holds for each smooth function $ \phi \in C^{\infty}(M)$,  
		\begin{equation}\label{integralGeomForm}
			\int_{\mathcal{G}} \left(\int_{\Sigma_{\sigma}} \phi\, d {A}_g \right) d\mu(\sigma) = \int_{M} \phi \,d{V}_g.
		\end{equation}
	\end{enumerate}     
\end{defi}

Let the parameterized family $\{\Sigma^k_\sigma\}_{\sigma \in \mathcal{G}}$ be as defined above. An interesting consequence of the existence of an integral geometric formula is the denseness of the family $\{\Sigma^k_\sigma\}_{\sigma \in \mathcal{G}}$, meaning that the closed set ${\cup_{\sigma \in \mathcal{G}} \Sigma_{\sigma}}$ covers $M$. In fact, otherwise, we can choose a positive function $\phi \in C^{\infty}(M)$ with support in the non-empty open set $M \setminus {\cup_{\sigma \in \mathcal{G}} \Sigma_{\sigma}}$, leading to a contradiction with the formula $(\ref{integralGeomForm})$. 

Despite the aforementioned property of integral geometric formulas our main interest lies in its connection with systolic inequalities. This relation was conceived by M. Pu in one of the earliest papers in systolic geometry (\cite{Pu52}) and since then was largely replicated (\cite{berger},\cite{gromov_systole}, \cite{paiva_fernandes}). Next we adapt his argument to our context.    

\begin{teo}\label{FIGimpSCC}
	Let $(M^n,g)$ be a closed Riemannian manifold, and suppose that there exists a family $\{\Sigma^k_\sigma\}_{\sigma \in \mathcal{G}}$ of closed smooth $k$-submanifolds parameterized by a closed manifold $\mathcal{G}$ admitting an integral geometric formula. Moreover, suppose that $\Sigma_{\sigma}$ is homological non-trivial and $\mathrm{Sys}_k(M,g)=\mathrm{vol}_g(\Sigma_{\sigma})$ for every $\sigma \in \mathcal{G}$. Then, for every Riemannian metric $\bar{g}$ in the conformal class of $g$, we have
	$$ \mathrm{Sys}_k^{\mathrm{nor}}(M,\bar{g}) \leq \mathrm{Sys}_k^{\mathrm{nor}}(M,g).$$
	Moreover, equality holds if and only if $\bar{g}$ is homothetic to the metric $g$.    
\end{teo}
\begin{proof}
	Let $\phi \in {C}^{\infty}_{>0}(M)$ be the conformal factor of $\bar{g}$, that is $\bar{g}=\phi g$. Then for each $\sigma \in \mathcal{G}$, 
	$$\mathrm{vol}_{\bar{g}}(\Sigma_{\sigma})=\int_{\Sigma_{\sigma}} \phi^{k/2} d{A}_g.$$
	Therefore, the integral geometric formula and the fact that $\mathrm{Sys}_k(M,g)=\mathrm{vol}_g(\Sigma_{\sigma})$ gives 
	\begin{align*}
		\int_{\mathcal{G}} \mathrm{vol}_{\bar{g}}(\Sigma_{\sigma})d \mu &=  \int_M \phi^{k/2} d{V}_g \\
		&\leq \mathrm{vol}_g(M)^{\frac{n-k}{n}}\left( \int_M \phi^{n/2} d \mathrm{V}_g \right)^{\frac{k}{n}}\\
		&= \mathrm{vol}_g(M)^{\frac{n-k}{n}} \mathrm{vol}_{\bar{g}}(M)^{\frac{k}{n}}.
	\end{align*}      
	Where we used Hölder's inequality. On the other hand, by part $\ref{IGF1}$ of definition $\ref{IGFdef}$, the map $\mathcal{G} \ni \sigma \mapsto \mathrm{vol}_{\bar{g}}(M) \in \mathbb{R}$ is continuous, hence $\mathrm{Sys}_k(M,\bar{g}) \leq \dashint_{\mathcal{G}} \mathrm{vol}_{\bar{g}}(\Sigma_{\sigma}) d \mu$. However, inserting the constant function $\psi \equiv 1$ in equation $(\ref{integralGeomForm})$ we see that $\mu(\mathcal{G})\,\mathrm{Sys}_k(\Sigma_{\sigma},g)=\mathrm{vol}_g(M)$. Consequently, the following inequality holds:     
	\begin{align*}
		\mathrm{Sys}_k(M,\bar{g}) &\leq \frac{\mathrm{Sys}_k(M,g)}{\mathrm{vol}_g(M)} \left( \int_{\mathcal{G}} \mathrm{vol}_{\bar{g}}(\Sigma_{\sigma}) d \mu\right) \\
		&\leq  \mathrm{Sys}_k^{\mathrm{nor}}(M,g) \mathrm{vol}_{\bar{g}}(M)^{\frac{k}{n}}, 
	\end{align*} 
	thus proving the desired result. The equality case happens if and only if equality holds in the Hölder inequality, therefore we must have $\phi$ constant in this case. 
\end{proof}

\section{Miscellanea of Hermitian Geometry} 

Here we compile classical theorems of Hermitian geometry. We begin by considering the linear case. Let $(V^{2n},\langle \cdot, \cdot \rangle)$ be a real Euclidean vector space of dimension $2n$, endowed with a compatible linear (almost) complex structure $I \in \mathrm{End}(V)$. The \textit{fundamental $2$-form} associated to $(V^{2n},\langle \cdot, \cdot \rangle,I)$ is given by: 
$$\omega(\cdot, \cdot) \doteq\langle I\cdot, \cdot \rangle.$$
In order to fix notation, we recall that the linear complex structure $I$ induces a decomposition on $\Lambda V^{*}_{\mathbb{C}} \doteq \Lambda V^* \otimes \mathbb{C}$, the space of complex-valued forms, given by:
$$\Lambda V^{*}_{\mathbb{C}}= \oplus_{k=0}^{2n} \oplus_{p+q=k}^{} \Lambda^{p,q}V^*,$$ 
where $\Lambda^{p,q}V^*$ denote the space of forms of type $(p,q)$. On the other hand, the fundamental form $\omega$ defines the \textit{Lefschetz operator}, which has a central role in this theory.
\begin{defi}
	The Lefschetz operator $L: \Lambda V^{*}_{\mathbb{C}} \to \Lambda V^{*}_{\mathbb{C}}$ is defined by $u \mapsto u \wedge \omega$.   
\end{defi}    

As usual, we can extend the Euclidean product of $V$ to $\Lambda V^*$. This allows the definition of the \textit{dual Lefschetz operator}, as follows. 
\begin{defi}\label{duallefop}
	The dual Lefschetz operator is the unique map $\Lambda: \Lambda V^* \to \Lambda V^*$ that satisfies: 
	$$\langle \Lambda u , v \rangle = \langle  u , Lv \rangle, $$ 
	for every $u,v \in \Lambda V^*$. We also denote by $\Lambda: \Lambda V^{*}_{\mathbb{C}} \to \Lambda V^{*}_{\mathbb{C}}$, the $\mathbb{C}$-linear extension of the dual Lefschetz operator. 
\end{defi}

Associated to the dual Lefschetz operator is the concept of \textit{primitive forms}.
\begin{defi}
	A $k$-form $u \in \Lambda^k V^{*}_{\mathbb{C}}$ is called primitive if $\Lambda u = 0$, and we denote the subspace of these forms by $P^k_{\mathbb{C}}$, and the space of real primitive $k$-forms will be denoted by $P^k$. Moreover, we also define the space of primitive forms of type $(p,q)$ as $P^{p,q} = P^{p+q}_{\mathbb{C}} \cap \Lambda^{p,q}V^*$.     
\end{defi} 

In the subsequent proposition we present some properties of the set of primitive forms. These properties usually are embedded in a deeper theorem called the \textit{Lefschetz Decomposition Theorem} (Proposition $1.2.30$, \cite{huybrechts2005complex}).
\begin{teo}\label{lefdecompthrm}
	Let $(V^{2n},\langle \cdot, \cdot \rangle,I)$ be a  real euclidean vector space endowed with a compatible linear complex structure. Then: 
	\begin{enumerate}[label=\alph*),ref=(\alph*)]
		\item\label{lefdecompthrm1} The map $L^{n-k}: \Lambda^k V^{*} \to \Lambda^{2n-k} V^{*}$ is bijective, for every $k \leq n$.
		\item\label{lefdecompthrm2}If $k \leq n$, then $P^k = \{u \in \Lambda^k V^* \; : \; L^{n-k+1}u=0\}$.  
	\end{enumerate} 
\end{teo}     

To conclude the review of the linear part we introduce another important operator, called the \textit{Riemann-Hodge pairing}.
\begin{defi}\label{RiemmanHodgepair}
	For each $k\leq n$, we define the Riemann-Hodge pairing as the bilinear form $\mathcal{RH}: \Lambda^k V^* \times \Lambda^k V^* \to \mathbb{R}$ given by: 
	$$\mathcal{RH}(u,v) = (-1)^{\frac{k(k-1)}{2}} u \wedge v \wedge \omega^{n-k},$$
	where we identify $\Lambda^{2n}V^*$ with $\mathbb{R}$ using the euclidean product. We also denote by $\mathcal{RH}$ the $\mathbb{C}$-linear extension of the Riemann-Hodge paring.   
\end{defi}  

The next theorem, know as the Riemann-Hodge bilinear relations, tell us how the Riemann-Hodge pairing acts over primitive forms (Corollary $1.2.36$, \cite{huybrechts2005complex}). 

\begin{teo}\label{RHtheorem}
	Let $\mathcal{RH}: \Lambda V^{*}_{\mathbb{C}} \times \Lambda V^{*}_{\mathbb{C}} \to \mathbb{C}$ denote the Riemann-Hodge paring. Then 
	$$\mathcal{RH}\left(\Lambda^{p,q}V^*,\Lambda^{p',q'}V^* \right)=0,$$
	whenever $(p,q) \neq (q',p')$. Moreover, if $p+q \leq n$, then 
	$$(\sqrt{-1})^{p-q}\mathcal{RH}(u,\bar{u})=(n-(p+q))! \cdot ||u||^2,$$
	for every $u \in P^{p,q}$.     
\end{teo}  

In what follows $(M^{2n}, g, J, \omega)$ denotes a closed and connected Kähler manifold. Clearly, the pointwise theory developed earlier generalizes to forms on the manifold $M$. Therefore, we have well-defined the Lefschetz operator and its dual. 

An important question is how these operators commute with the differential and codifferential on $M$. The Kähler condition imposes important relations between these operators, which are called \textit{Kähler identities}. In what follows, we present some of these relations. However, before that, we need to introduce the $\delta^c$ operator.

\begin{defi}\label{deltacdef}
	For each $1\leq k \leq 2n$, we define $\delta^c: \Omega^k_{\mathbb{C}}(M) \to \Omega^{k-1}_{\mathbb{C}}(M)$ as $\delta^{c}=i (\partial^* - \bar{\partial}^*)$. 
\end{defi} 

The next proposition shows how $\Lambda$ commutes with the exterior differential and $\delta$ with $\delta^c$.

\begin{prop}\label{propofdeltac}(cf. Proposition $3.1.12$ in \cite{huybrechts2005complex})
	Let $(M^{2n},g,J,\omega)$ be a closed and connected Kähler manifold, and $\Lambda$ the dual Lefschetz operator. Then:
	\begin{enumerate}[label=\alph*),ref=(\alph*)]
		\item $[\Lambda,d]=-\delta^c$;
		\item $\delta \circ \delta^c + \delta^c \circ \delta = 0$.  
	\end{enumerate}
\end{prop}  

\end{appendices}

\hfill
\begin{flushleft}
{\small \textbf{Acknowledgments} I would like to thank my PhD advisor, Lucas Ambrozio, for all the insightful conversations and enriching suggestions on earlier drafts of these paper. Additionally, I extend my thanks to CNPq - Conselho Nacional de Desenvolvimento Científico e Tecnológico and FAPERJ - Fundação Carlos Chagas Filho de
	Amparo à Pesquisa do Estado do Rio de Janeiro for their support of this work.}
\end{flushleft}
\hfill

\begin{flushleft}
{\small \textbf{Funding} The author Luciano Luzzi Junior was supported by CNPq - Conselho Nacional de Desenvolvimento Científico e Tecnológico (Process: 141589/2020-5) and by FAPERJ - Fundação Carlos Chagas Filho de
	Amparo à Pesquisa do Estado do Rio de Janeiro (Doutorado Nota 10, Process: 202.373/2022).}
\end{flushleft}
\hfill

\begin{flushleft}
{ \small \textbf{Conflict of interest/Competing interests} The author have no conflicts of interest to declare that are relevant to the content of this article.}
\end{flushleft}

\bibliography{sn-bibliography}

\end{document}